\documentclass{article}
\usepackage[utf8]{inputenc} 

\usepackage{geometry}
\usepackage{manyfoot}
\usepackage{graphicx,color}
\usepackage[outercaption]{sidecap}
\usepackage{xfrac}
\usepackage{subfig}
\usepackage{bm}
\usepackage[font={small}]{caption}

\usepackage{hyperref}

\usepackage{fancyhdr}

\usepackage{enumitem}

\usepackage[nottoc,notlof,notlot]{tocbibind}
\usepackage[titles,subfigure]{tocloft}

\usepackage{amssymb}
\usepackage{amsmath}
\usepackage{latexsym}
\usepackage{amsthm}
\usepackage{eucal}
\usepackage{eufrak}
\usepackage{amsthm}
\theoremstyle{plain}
\newtheorem{theorem}{Theorem}
\newtheorem{corollary}[theorem]{Corollary}
\newtheorem{lemma}[theorem]{Lemma}
\newtheorem*{conjecture}{Conjecture}
\newtheorem{proposition}[theorem]{Proposition}
\theoremstyle{definition}
\newtheorem{definition}[theorem]{Definition}
\newtheorem{example}[theorem]{Example}
\newtheorem{remark}[theorem]{Remark}

\makeatletter
\let\orgdescriptionlabel\descriptionlabel
\renewcommand*{\descriptionlabel}[1]{%
  \let\orglabel\label
  \let\label\@gobble
  \phantomsection
  \edef\@currentlabel{#1}%
  \let\label\orglabel
  \orgdescriptionlabel{#1}%
}
\makeatother

\def\e{\varepsilon}
\newcommand{\Ze}{\mathbb{Z}_{e}^2}
\newcommand{\Zo}{\mathbb{Z}_{o}^2}
\newcommand{\R}{\mathbb{R}}
\newcommand{\I}{\mathcal{I}}
\newcommand{\Z}{\mathbb{Z}}

\newcommand{\N}{\mathbb{N}}
\newcommand{\A}{\mathcal{A}}
\newcommand{\D}{\mathcal{D}}
\newcommand{\mN}{\mathcal{N}}
\newcommand{\mP}{\mathcal{P}}
\newcommand{\weakcs}{\overset{*}{\rightharpoonup}}
\newcommand{\jv}{{\bf j}}
\newcommand{\x}{{\bf x}}
\newcommand{\y}{{\bf y}}
\newcommand{\iv}{{\bf i}}
\newcommand{\ev}{{\bf e}}

\DeclareMathOperator{\conv}{{\rm conv}} 
\DeclareMathOperator{\sgn}{{\rm sgn}}

\DeclareMathOperator{\supp}{{\rm supp}}

\newcommand{\argmin}[1]{\underset{#1}{\operatorname{argmin}}\;}
\newcommand{\argmax}[1]{\underset{#1}{\operatorname{argmax}}\;}
\newcommand{\mres}{\mathbin{\vrule height 1.6ex depth 0pt width
0.13ex\vrule height 0.13ex depth 0pt width 1.3ex}}

\numberwithin{equation}{section} 

\begin{document}

\title{Nucleation and growth of lattice crystals}

\author{Andrea Braides, Giovanni Scilla and Antonio Tribuzio}
\newcommand{\Addresses}{{
  \bigskip
  \footnotesize

A.~Braides, \textsc{Dipartimento di Matematica, Universit\`a di Roma ``Tor Vergata'', Via della Ricerca Scientifica 1, 00133 Roma, Italy}\par\nopagebreak
  \textit{E-mail address}, A.~Braides: \texttt{braides@mat.uniroma2.it}

\medskip

G.~Scilla, \textsc{Dipartimento di Matematica ed Applicazioni ``R. Caccioppoli'' , Universit\`{a} di Napoli Federico II, Via Cintia Monte Sant'Angelo, 80126 Napoli, Italy}\par\nopagebreak
  \textit{E-mail address}, G.~Scilla: \texttt{giovanni.scilla@unina.it}

\medskip

A.~Tribuzio, \textsc{Dipartimento di Matematica, Universit\`a di Roma ``Tor Vergata'', Via della Ricerca Scientifica 1, 00133 Roma, Italy}\par\nopagebreak
  \textit{E-mail address}, A.~Tribuzio: \texttt{tribuzio@mat.uniroma2.it}
}}

\date{}

\maketitle

\begin{abstract}\noindent
{A variational lattice model is proposed} to define an evolution of sets from a single point (nucleation)
following a criterion of ``maximization'' of the perimeter.
At a discrete level, the evolution has a ``checkerboard'' structure and its shape is affected by the choice of the norm defining the dissipation term.
For every choice of the scales, the convergence of the discrete scheme to a family of expanding sets with constant velocity is proved.
\end{abstract}


\section{Introduction}

In this paper we propose a variational model for nucleation and growth of a set by {\em maximization} of its perimeter through an energy-dissipation balance at fixed time step. We follow an implicit Euler scheme used by Almgren, Taylor and Wang to prove existence of sets moving by mean curvature by {\em minimization} of the perimeter (see \cite{ATW83}). In that case, fixed a time step $\tau>0$, one can define iteratively the discrete orbits $E^\tau_k$ at fixed $\tau$ from an initial set $E_0$  as $E^\tau_0=E_0$ and $E^\tau_k$ as a solution of
\begin{equation}\label{minor}
\min\Bigl\{P(E)+{1\over \tau} D^p(E, E^\tau_{k-1})\Bigr\},\qquad D^p(E, F)=\int_{E\triangle F} {\rm dist}_p(x,\partial F)\,dx,
\end{equation}
where ${\rm dist}_p(x,\partial E)=\min\{\|x-y\|_p: y\in\partial E\}$, $p\in [1,\infty]$. The term $D^p$ is interpreted as a {\em dissipation}, and \eqref{minor} can be seen as a minimization of $P$ subject to a constraint due to the dissipation, which forces $E^\tau_k$ to be close to $E^\tau_{k-1}$ for $\tau$ small.
In \cite{ATW83} it is proved (in the case $p=2$) that the piecewise-constant interpolations  $E^\tau(t)=E^\tau_{\lfloor t/\tau\rfloor}$ converge to a decreasing family of sets $E(t)$ which move by mean curvature.

Such a scheme cannot be directly followed taking maximization of the perimeter as a driving mechanism, which would correspond to replacing $P$ with $-P$.
Indeed, we may have sets $E$ such that $E\triangle E_0$ has small measure (and hence with small dissipation) but with arbitrarily large perimeter, so that the minimum value for $k=1$ in \eqref{minor} is $-\infty$ and the scheme arrests at the first step.
In order to overcome this issue, we discretize our problem by introducing a spatial length scale $\e$.
For technical reasons explained below, we will examine only a two-dimensional setting, and for simplicity parameterize our problem on the lattices $\e\Z^2$.
We then restrict to sets that can be written as the union of squares of side length $\e$ and centers in $\e\Z^2$.
Within this class we shall consider the problem of {\em nucleation}; i.e., of motion from a minimal set, a single $\e$-square $E^\e_0$ (which we may suppose to be centered in $0$).
With fixed $\e$ and $\tau$, the discrete orbits are defined as $E^{\e,\tau}_0=E^\e_0$ and $E^{\e,\tau}_k$ as a solution of
\begin{equation}
\min\Bigl\{-P_\e(E)+{1\over \tau} D_\e^p(E, E^{\e,\tau}_{k-1}): E^{\e,\tau}_{k-1}\subseteq E\Bigr\},
\end{equation}
where $P_\e$ is the restriction of the perimeter functional to unions of $\e$-squares, and $D_\e^p$ is a discretization of the dissipation $D^p$ which, for every $E\supseteq F$, reduces to
$$
D^p_\e(E,F)=\e^2\sum_{i\in E\cap\e Z^2}{\rm dist}_p(i, (\e \Z^2\setminus F))\,.
$$

Note that we consider a {\em growing} family of sets with respect to inclusion.
With fixed $\tau=\tau_\e$  we will characterize  the cluster points $E(t)$ as $\e\to 0$ of the interpolated functions $E^\e(t)=E^{\e,\tau}_{\lfloor \e t/\tau\rfloor}$, which are the generalization to varying energies of the Almgren-Taylor-Wang scheme scaled in the time variable. Note the different scaling of the time variable, which is the one that better describes the evolution.
The  form of $E$ will depend on the interplay between $\e$ and $\tau$; more precisely,  
on the limit ratio $\alpha$ of $\e^2/\tau$ as $\e\to 0$. We remark that the chosen time scaling can be directly interpreted as giving
the {\em minimizing movements along the sequence $-\e P_\e$ at scale $\tau$}, which are defined in \cite{Bra13}. 
This scaling is also justified by the fact that the energies $-\e P_\e$ have a non-trivial $\Gamma$-limit.

We describe the case $0<\alpha<+\infty$, which is the most relevant. It is not restrictive to suppose that $\alpha\tau=\e^2$.
By the homogeneity properties of the perimeter and the dissipation, we note that $E^{\e,\tau}_k=\e A^\alpha_k$, where  $A^\alpha_0=q$ (the unit square centered in $0$), and we solve iteratively 
\begin{equation}\label{equa}
\min\Bigl\{-P_1(A)+{\alpha} D_1^p(A, A^{\alpha}_{k-1}): A^{\alpha}_{k-1}\subseteq A\Bigr\}.
\end{equation}

The first step is particularly meaningful, and consists in solving the minimum problem
\begin{equation}\label{equa}
\min\Bigl\{-P_1(A)+{\alpha} D_1^p(A, q): q\subseteq A\Bigr\}.
\end{equation}
We have

$\bullet$ the first set $A^\alpha_1$ is a part of the checkerboard of unit squares in $\R^2$ containing $0$ (which we call the {\em even checkerboard}). While this fact is clear `locally', the proof that the whole set is a single checkerboard requires a non-trivial covering argument, in which $\R^2$ is covered by sets in which the minimal set $A$ is (part of) the correct checkerboard. This argument can be avoided in the case $p=\infty$, which has been treated directly in \cite{BraSci14};

$\bullet$ since every square of the (even) checkerboard gives an independent contribution of energy and dissipation, a point $i\in\Z^2$ may belong to $A^\alpha_1$ if and only if ($i_1+i_2\in 2\Z$ and) the corresponding contribution is non positive; i.e.,
\begin{equation}\label{equap}
-4+\alpha\|i\|_p\le 0;
\end{equation}

$\bullet$ if $\alpha\not\in\{ 4/\|i\|_p: i\in\Z^2, i_1+i_2\in 2\Z\}$ then $A^\alpha_1$ is uniquely determined by \eqref{equap}, and it is the union of all squares in the even checkerboard with centers in the set $$\mathcal{N}^p_\alpha=\{i\in\Z^2\cap B^p_{4/\alpha}: i_1+i_2\in 2\Z\},
$$ where $B_r^p=\{x\in\R^2: \|x\|_p<r\}$. Note that $\mathcal{N}^p_\alpha=\{0\}$ if $\alpha>4$;

We consider only $\alpha$ with such a unique minimizer. The subset $\mathcal{N}^p_\alpha$ of $\Z^2$ will be called the {\em nucleus} of the process. Correspondingly, we have the continuum set $P^p_\alpha$ obtained as the convexification of $N^p_\alpha$. Note that $P^1_\alpha$ and $P^\infty_\alpha$ are always squares, but for the other $p$ the form of $P^p_\alpha$ does depend on $\alpha$. 

The most delicate argument in the study of the discrete scheme is the characterization of the sets $A^\alpha_k$ for $k>1$. Similarly to the case $k=1$ this is done by covering $\R^2$ with a family of small sets, mainly squares and rectangles, in each of which we prove that the minimal set is again the even checkerboard. In order to construct this covering we have to define the `edges' of the nucleus $\mathcal{N}^p_\alpha$, and consider separately the regions of $\R^2$ that project on those edges according to the $p$-distance. At this point we have a technical hypothesis to add; namely, that all such regions are infinite (which is satisfied if these edges enclose a convex shape but may not be the case for some exceptional values of $\alpha$). The complex construction of this covering is the reason why we limit our analysis to a two-dimensional setting.

With this characterization, using \eqref{equap} we immediately have that the centers of the squares in $A^\alpha_k$ are exactly the points $i\in\Z^2$ with $i_1+i_2\in2\Z$ and distance not greater than $4/\alpha$ from $A^\alpha_{k-1}$, so that
$$
A^\alpha_k\cap\Z^2=(A^\alpha_{k-1}\cap\Z^2)+ (A^\alpha_1\cap\Z^2).
$$
In a sense, every square in $A^\alpha_{k-1}$ acts as the `center' of a nucleus. Note in this step that if $A^\alpha_1$ were not unique, then we would have an `increasing non-uniqueness' of $A^\alpha_k$, which in particular may even not be the intersection of the square checkerboard with a convex region. 

Since the centers of the squares in $A^\alpha_k$ are obtained as sums of $k$ elements in $\mathcal{N}^p_\alpha$, a result on {\em Minkowsky sums} of sets shows then that the convex envelope of $A^\alpha_k\cap\Z^2$ is the convex envelope of $k \mathcal{N}^p_\alpha$, which is an interesting and not a trivial fact. At this point we can go back to the original problem and describe the discrete orbits.
$$
E^{\e,\tau}_k=\e A^\alpha_k= \e k P^p_\alpha,\qquad E^\e(t)=E^{\e,\tau}_{\lfloor\e t/\tau\rfloor}=\e \Bigl\lfloor {\alpha\over\e}\Bigr\rfloor P^p_\alpha.
$$
Letting $\e\to 0$ we then conclude that the desired evolution is a linear evolution of sets
$$
E(t)=\alpha t P^p_\alpha.
$$
Note that $P^p_\alpha=\{0\}$ and hence the evolution is {\em pinned} if $\alpha>4$. Moreover, remarking that  $\alpha P^p_\alpha\sim B^p_{4}$ for $\alpha$ small, we also recover the case $\alpha=0$, corresponding to the regime $\e^2<\!<\tau$, for which $E(t)= 4t B^p_1$.

\smallskip
We note that in \cite{BGN} the same discretization approach had been followed for the (positive) perimeter and non-trivial initial data. The resulting evolution therein is a discretized motion by square-crystalline curvature (see \cite{AT95}), which highlights the anisotropy of the lattice intervening in the perimeter part, while the effect of the dissipation is confined in the form of the mobility. In the present analysis the effect of the dissipation and of the perimeter parts are combined in the determination of the shape of the nucleus, but the perimeter term actually acts as an approximation of an area and is less relevant for small values of $\alpha$.  Note that our discretization approach can be regarded as a `backward' version of \cite{BGN} if the index $k$ is considered as parameterizing negative time (see \cite[Section 10.2]{Bra13}). Other analyses of minimizing movements on lattices related to the perimeter can be found in \cite{BraSci,Sci14,Ruf,Sci20}. We note that checkerboard, stripes and other structures arise in antiferromagnetic systems related to maximization of the perimeter (see \cite{Braides-Cicalese15} for a variational analysis in terms of $\Gamma$-convergence, and the wide literature in Statistical Mechanics, e.g.~\cite{Giuliani, Daneri}). Some cases in which microstructures on lattices are involved and produce interesting variants of motion by crystalline curvature are studied in \cite{BCY,BraSo15}. For an overview on geometric motion on planar lattices see the recent lecture notes \cite{BraSolNotes}.

Even though our interest is mainly in the analytical issues of this nucleation process, it is suggestive and interesting to connect this work with the process of biomineralization, where nucleation occurs via the formation of a small nucleus of a new phase inside the large volume of the old phase (see, e.g., \cite{DeV}). At very small size, adding even one more molecule increases the free energy of the system and this produces, on average, the dissolution of the nucleus. Above a threshold, when the contribution of the surface free energy becomes negligible, every addition of a molecule to the lattice lowers the free energy and allows for the growth of the nucleus. In this direction, lattice systems have been widely used as a simple model in simulations of complex phenomena, as the vapor-liquid nucleation (see, e.g., \cite[Section~8.9]{Kali}). From a completely different point of view, our structure results can be related to the investigation of the influences of environmental heterogeneities on the spatial self-organization of microbial communities (see, e.g., \cite{Ciccarese, MSM}); in particular, how interactions of different type (mutualism/commensalism) between competing neighboring genotypes and their mutual distance can produce spatial patterns of varying complexity and intermixing, as a random distribution, a spatial segregation or even a checkerboard, and how they may affect the collective behaviour and the rate of growth of the colony.

{\bf Outline of the paper.} 
In Section \ref{sec:notation} we fix some notation and recall some preliminaries in Discrete Geometry. We introduce the class of admissible sets that we will consider throughout the paper, and the notions of \emph{effective boundary} and \emph{discrete edge} of a set.  
In Section~\ref{sec:setting} we define perimeter energies $P_\varepsilon$ and, \emph{for a general norm} $\varphi$, dissipations $D_\varepsilon^\varphi$ we will deal with, together with the main functional $\mathcal{F}_{\e,\tau}^\varphi$. Correspondingly, we introduce the time-discrete minimization scheme for  a suitably scaled version of the energies $\mathcal{F}_{\e,\tau}^\varphi$ (Section~\ref{sec:choicescal}). 

The convergence analysis of this scheme at the regime $\varepsilon<\!<\!\tau$ is carried out in Section~\ref{sec:fastconvergence}. 
In Section~\ref{sec:scaledcheckbd} we address the problem of determining the solutions of scheme \eqref{equa} at the critical regime $\varepsilon=\alpha\tau$, under a monotonicity constraint on the discrete trajectories. 
We introduce here also a first restriction on the dissipations $D_\varepsilon^\varphi$, by requiring that $\varphi$ be an \emph{absolute norm}; i.e., $\varphi(\x)=\varphi(|x_1|,|x_2|)$. 
The explicit characterization of the first step $A^1_\alpha$ of the discrete evolution, provided with Proposition~\ref{firststep}, is based on a local analysis by means of the $2\times2$-square tilings introduced in Section~\ref{sec:bigsquares} and the key submodularity-type norm-inequality \eqref{normestimate}. 
In order to prove that an analogous structure result can be obtained for each step $A_\alpha^k$, $k\geq2$; i.e., for minimizers of the energy $\mathcal{F}_\alpha^\varphi(\cdot,A_\alpha^{k-1})$, we will assume that $\varphi$ is a \emph{symmetric} absolute \emph{normalized} norm (see Section~\ref{sec:norm}), complying with a technical assumption \ref{norm-derivative}, and that the competitors fulfill suitable geometric assumptions (see \eqref{monotone-edges}). 
The proof of this stability result, given with Propositon~\ref{steps}, is the content of Section~\ref{sec:proofsteps} and relies on a localization argument only reminiscent of that used in the proof of Proposition~\ref{firststep}, 
as we are forced to define a new covering outside every discrete edge 
contained in the effective discrete boundary of the current step $A_\alpha^{k-1}$. 
In Section~\ref{sec:nucleation}, with Theorem~\ref{thm:nucleation} we characterize the time-discrete flow $\{A_\alpha^k\}_{k\geq0}$ as a geometric iterative process, based on properties of Minkowski sums. 

In Section~\ref{sec:limitmotion} we describe the resulting limit evolutions and we prove the existence of a pinning threshold (see Definition~\ref{pin-def}). 
We conclude our analysis by exhibiting, in Section~\ref{sec:examples}, some examples where both the microscopic and the limit evolutions can be explicitly characterized. 
The closing Section~\ref{sec:conjectures} contains some conjectures on evolutions \emph{without the monotonicity constraint}. 

\section{Notation and preliminaries}\label{sec:notation}

 The generic point of $\R^2$ will be denoted by $\x=(x_1,x_2)$,  the Euclidean norm by $|\cdot|$ in any dimension.
 The space of subsets of $\R^2$ with finite perimeter endowed with the Hausdorff distance $d_{\mathcal{H}}$ is denoted by $\mathcal{X}$, and  the 1-dimensional Hausdorff measure by $\mathcal{H}^1$ (see for instance \cite{AFP}).

The function $\varphi:\R^2\to[0,+\infty)$ denotes any norm in the plane.
We use the standard notation for the $\ell^p$-\emph{norm}, for every $1\le p\le\infty$; that is,
$$
\|\x\|_p = \big(|x_1|^p+|x_2|^p\big)^{1\over p} \text{ if } 1\le p<\infty,
\quad
\|\x\|_\infty = \max\{|x_1|,|x_2|\} \text{ if } p=\infty,
$$
for every $\x\in\R^2$.
For every $r>0$, $B_r^\varphi(\x)=\{\y\in\R^2 \,:\, \varphi(\x-\y)<r\}$ is the open ball of radius $r$ and center $\x$ corresponding to the norm $\varphi$, while  $q_r(\x)=\x+[-r/2,r/2]^2$ is the \emph{$r$-square} of side-length $r$ centered at $\x$; when $\x=(0,0)$, we will use the shorthand $B_r^\varphi$ and $q_r$ in place of $B_r^\varphi(\x)$ and $q_r(\x)$, respectively.
For every $\x\in\R^2$, $E\subseteq\R^2$ we set $d^\varphi(\x,E)=\inf_{\y\in E} \varphi(\x-\y)$.
The segment connecting $\x_1,\x_2\in\R^2$ is denoted by  $[\x_1,\x_2]:=\big\{\y\in\R^2\,:\, \y=s\x_1+(1-s)\x_2,\, s\in[0,1]\big\}$.

\begin{definition}\label{angle}
Given two unit vectors ${\bf v}_1, {\bf v}_2\in\mathbb{S}^1$,  $\theta({\bf v}_2,{\bf v}_1)\in[-\pi,\pi]$ denotes the \emph{signed angle between ${\bf v}_1$ and ${\bf v}_2$}, defined as
$$
\theta({\bf v}_2,{\bf v}_1)=\big(\theta_2-\theta_1+\pi\, (\text{mod. } 2\pi)\big) - \pi,
$$
where $\theta_1$ and $\theta_2$ are the angles corresponding to the exponential representations of ${\bf v}_1$ and ${\bf v}_2$, respectively.
\end{definition}

Let $\Z^2$ be the standard square lattice. We consider the partition of $\Z^2$ given by $\Z^2=\Ze\cup\Zo$, where $\Ze=\big\{\iv\in\Z^2 \,:\, i_1+i_2\in 2\Z\big\}$ and $\Zo=(1,0)+\Ze$.

We will call a \emph{lattice set} any subset $\I\subseteq\Z^2$, and  $\#\I$ denotes its cardinality.
We also recall that the \emph{boundary} of a lattice set $\I$ is the set
$$\partial \I = \big\{\iv\in\I \,| \text{ there exists } \jv\in\Z^2\setminus\I \,:\, |\iv-\jv|=1 \big\}.$$
Given a lattice set $\I$, the \emph{convex hull} of $\I$ is the smallest convex subset of $\R^2$ containing $\I$, which is denoted by $\conv(\I)$.
A polygon whose vertices are points of the lattice is said a \emph{lattice polygon}.
The set $\conv(\I)$ is an example of a (convex) lattice polygon, for every $\I\subset\Z^2$.

Let $\e>0$ be a fixed parameter and consider the lattice $\e\Z^2$.
All the notation given above for subsets of $\Z^2$ extends also to subsets of $\e\Z^2$.
We identify any lattice set $\I\subset\e\Z^2$ with the subset $E(\I)$ of $\R^2$ given by the union of $\e$-squares centered at points of $\I$; namely,
$$E(\I):=\bigcup_{\iv\in \I}q_\e(\iv).$$
Accordingly, we define the class of \emph{admissible sets} as
\begin{equation}\label{unions}
\D_\e := \big\{E\subset\R^2 \,:\, E=E(\I) \mbox{ for some lattice set } \I\subseteq\e\Z^2\big\},
\end{equation}
and to each set $E\in\D_\e$ we associate the lattice set $Z_\e(E):=E\cap\e\Z^2$, the \emph{set of centers} of $E$.
When $\e=1$ we will simply write $\D$ and $Z(E)$ in place of $\D_1$ and $Z_1(E)$, respectively.

\begin{definition}[the classes of \emph{checkerboard sets}]\label{def:check}
We introduce the classes of \emph{even} and \emph{odd $\e$-checkerboard sets}
\begin{equation}\label{class1}
\A_\e^e = \left\{E\in\D_\e \,:\, Z_\e(E)\subseteq\e\Ze \right\},
\end{equation}
and analogously the class $\A_\e^o$ by requiring that $\I\subseteq\e\Zo$.
We refer to $E(\e\Ze)$ and $E(\e\Zo)$ as the \emph{even} and \emph{odd} $\e$-\emph{checkerboard}, respectively.
In the following we will write $\D$, $\A^e$, $\A^o$ in place of $\D_1$, $\A_1^e$, $\A_1^o$, and we will use the shorthand \emph{checkerboard set} (in place of ``$1$-checkerboard set'') to denote any set in $\A^e$ and $\A^o$.
\end{definition}

\subsection{Preliminaries on lattice geometry}\label{sec:lattice-convex}

For our purposes we fix some notation and introduce some basic definitions in lattice geometry that will be useful for the analysis performed in Subsection \ref{sec:step2}.

\begin{definition}\label{lattice-convex-def}
A lattice set $\I\subseteq\Ze$ is said to be $\Ze$-\emph{convex} if $\conv(\I)\cap\Ze=\I$.
Analogously, $\I\subseteq\Zo$ is $\Zo$-\emph{convex} if $\conv(\I)\cap\Zo=\I$.
Accordingly, we define the subclass $\A^e_{\rm conv}\subset\D$ as 
$$
\A^e_{\rm conv} = \{ E\in\D \,:\, Z(E) \text{ is } \Ze\text{-convex}\},
$$
and, analogously, the subclass $\A^o_{\rm conv}$ by requiring $Z(E)$ to be $\Zo$-convex.
We also set the class $\A_{\conv}:=\A_{\conv}^e\cup\A_{\conv}^o$.
\end{definition}

The notion of \emph{convex lattice set} has already been given for $\I\subset\Z^2$ (see for instance \cite{GGZ}).
Note that $\I$ is $\Ze$-convex if and only if there exists a convex set $K\subset\R^2$ such that $\I=K\cap\Ze$, and the same holds for $\Zo$-convex sets.

For every lattice set $\I\subseteq\Ze$ (or $\Zo$) there holds $\partial\I=\I$, since $\I$ consists of isolated points of $\Z^2$.
Since in the following we will deal with checkerboard sets we need a finer definition of boundary for such lattice sets.

\begin{figure}[h]
\centering
\def\svgwidth{300pt}
\begingroup%
  \makeatletter%
  \providecommand\color[2][]{%
    \errmessage{(Inkscape) Color is used for the text in Inkscape, but the package 'color.sty' is not loaded}%
    \renewcommand\color[2][]{}%
  }%
  \providecommand\transparent[1]{%
    \errmessage{(Inkscape) Transparency is used (non-zero) for the text in Inkscape, but the package 'transparent.sty' is not loaded}%
    \renewcommand\transparent[1]{}%
  }%
  \providecommand\rotatebox[2]{#2}%
  \newcommand*\fsize{\dimexpr\f@size pt\relax}%
  \newcommand*\lineheight[1]{\fontsize{\fsize}{#1\fsize}\selectfont}%
  \ifx\svgwidth\undefined%
    \setlength{\unitlength}{1039.31570435bp}%
    \ifx\svgscale\undefined%
      \relax%
    \else%
      \setlength{\unitlength}{\unitlength * \real{\svgscale}}%
    \fi%
  \else%
    \setlength{\unitlength}{\svgwidth}%
  \fi%
  \global\let\svgwidth\undefined%
  \global\let\svgscale\undefined%
  \makeatother%
  \begin{picture}(1,0.48809868)%
    \lineheight{1}%
    \setlength\tabcolsep{0pt}%
    \put(0,0){\includegraphics[width=\unitlength,page=1]{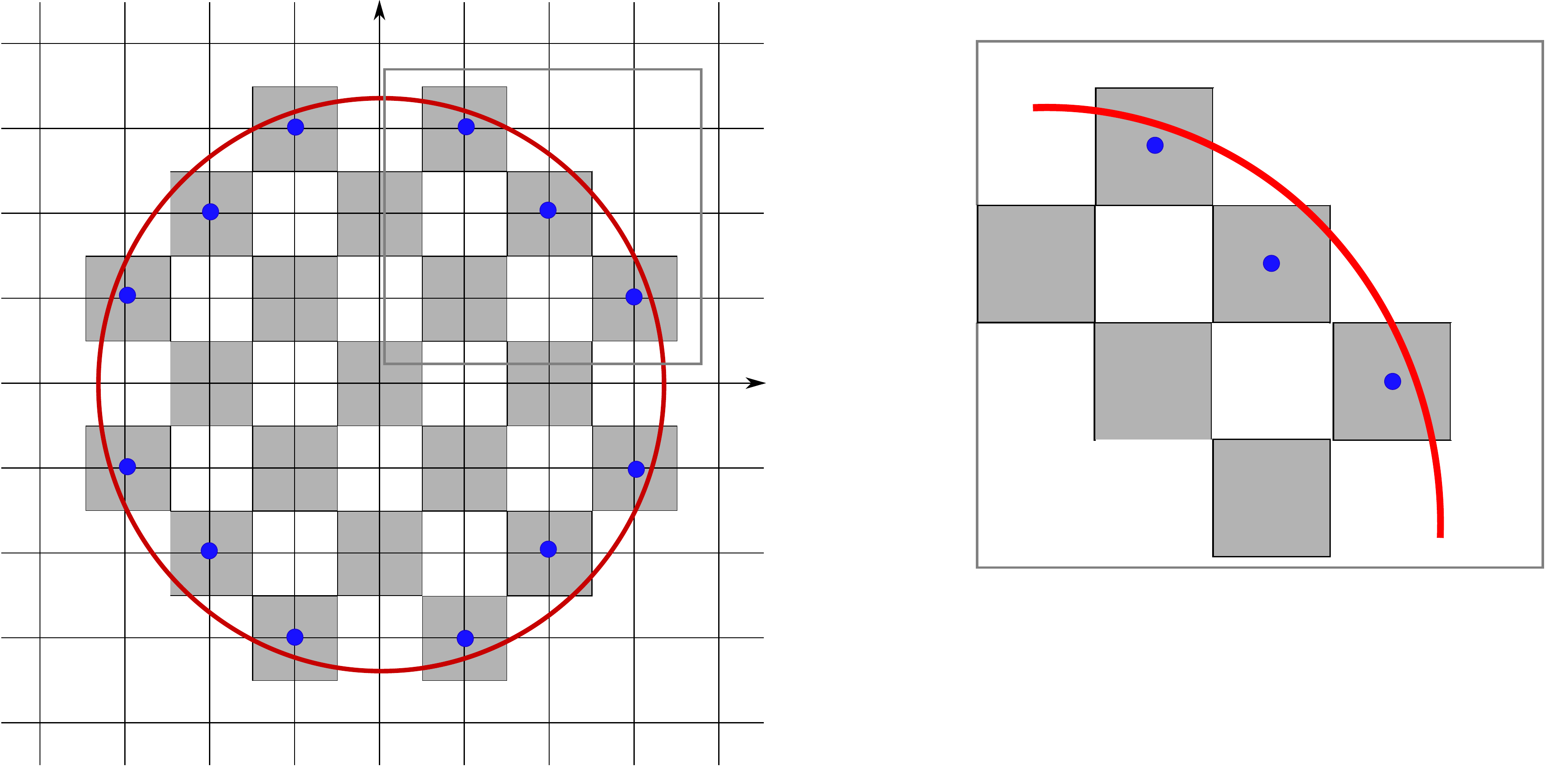}}%
    \put(0.8080797,0.41493648){\color[rgb]{0,0,0}\makebox(0,0)[lt]{\lineheight{1.25}\smash{\begin{tabular}[t]{l}$\jv^-$\end{tabular}}}}%
    \put(0.88714983,0.3341135){\color[rgb]{0,0,0}\makebox(0,0)[lt]{\lineheight{1.25}\smash{\begin{tabular}[t]{l}$\jv$\end{tabular}}}}%
    \put(0.94240543,0.24504332){\color[rgb]{0,0,0}\makebox(0,0)[lt]{\lineheight{1.25}\smash{\begin{tabular}[t]{l}$\jv^+$\end{tabular}}}}%
    \put(0,0){\includegraphics[width=\unitlength,page=2]{innerbound.pdf}}%
  \end{picture}%
\endgroup%

\caption{The (discrete) effective boundary of $E$ (in blue).}
\label{fig:innerbound}
\end{figure}

\begin{definition}\label{def:eff-boundary}
Let $\I\subset\Ze$ be a lattice set.
We define the \emph{effective (discrete) boundary} of $\I$ as
$$\partial^{\rm eff} \I=\big\{\jv\in \I \,:\, \text{there exists } {\jv_0}\in \Ze\setminus\I \text{ such that } |\jv-{\jv_0}|=\sqrt{2}\big\}.$$
The same definition is given for lattice sets $\I\subset\Zo$.
Let $E\in\A^e\cup\A^o$, we will write $\partial^{\rm eff}E=\partial^{\rm eff} Z(E)$, see Figure \ref{fig:innerbound}.
\end{definition}

\begin{figure}[h]
\centering
\includegraphics[scale=.7]{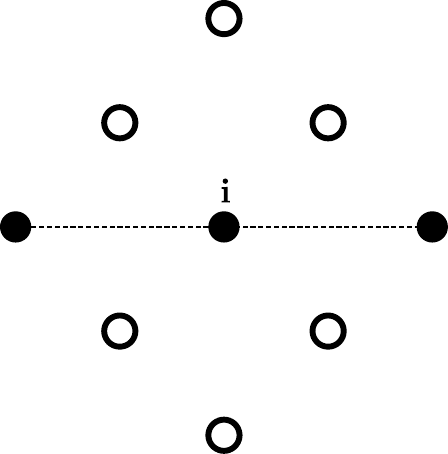}
\qquad
\includegraphics[scale=.7]{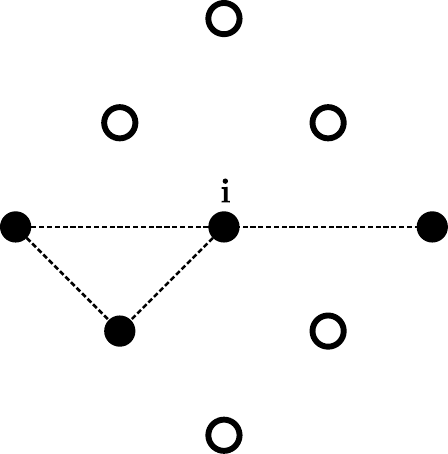}
\qquad
\includegraphics[scale=.7]{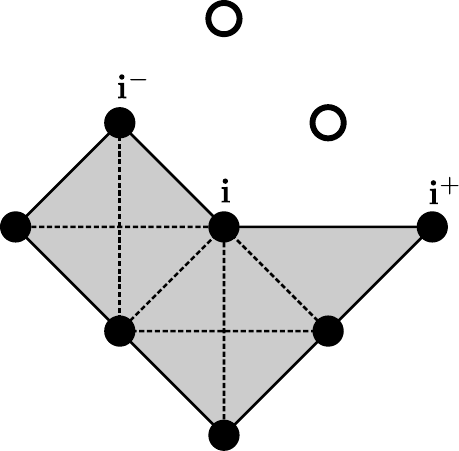}
\caption{The black dots are lattice points of $\I$.
The first two figures are different examples of ``degenerate'' $\iv$.
On the right an example of a non-degenerate $\iv$ and corresponding $\iv^-$ and $\iv^+$; in gray polygon $\mP$.}
\label{fig:non-deg}
\end{figure}

{Given $E\in\A^e\cup\A^o$, and consider $\iv\in\partial^{\rm eff}E$.
We set $\I=\{\jv\in Z(E)\,:\, \|\jv-\iv\|_1\le2\}$.
Then $\iv$ is said to be \emph{non-degenerate} if the set
$$
\bigcup_{\substack{\jv_1,\jv_2\in\I \\ |\jv_1-\jv_2|\le 2}}[\jv_1,\jv_2]
$$
is the boundary of a triangulation of a simple polygon $\mP$.
Then, we can define two boundary points $\iv^-, \iv^+\in\partial^{\rm eff}E$ as the vertices of $\mP$, respectively, preceding and following $\iv$ in the clockwise orientation of $\partial\mP$, as depicted in Figure \ref{fig:non-deg}.
We will say that $\iv^-$ \emph{precedes} $\iv$ and that $\iv^+$ \emph{follows} $\iv$.

In the sequel, we will often consider the following non-degeneracy condition on sets $E\in\A_{\conv}$;
\begin{equation}\label{non-degeneracy}
\text{every } \iv\in\partial^{\rm eff}E \text{ is non-degenerate.}
\end{equation}
Condition \eqref{non-degeneracy} allows to define an orientation of $\partial^{\rm eff}E$, since for every $\iv\in\partial^{\rm eff}E$ we can define $\iv^-$ and $\iv^+$ as above.
The following definitions are therefore well-posed.}

\begin{definition}[Discrete convex vertices]\label{pseudo-edge-def}
Let $E\in\A_{\rm conv}$ satisfy \eqref{non-degeneracy}.
Given $\jv\in\partial^{\rm eff} E$, let $\jv^+$ (resp., $\jv^-$) follow (resp., precede) $\jv$ in $\partial^{\rm eff} E$ in the clockwise orientation.
We define the \emph{right} and \emph{left outward unit normal vector} at $\jv$ as
$$
\bm\nu^+(\jv) := \frac{(j_2-j_2^+,j^+_1-j_1)}{\sqrt{(j_2^+-j_2)^2+(j_1^+-j_1)^2}}, \quad
\bm\nu^-(\jv) := \frac{(j_2^--j_2,j_1-j^-_1)}{\sqrt{(j_2-j_2^-)^2+(j_1-j_1^-)^2}}\,,
$$
respectively.
Then we say that $\jv$ is a \emph{discrete convex vertex} (or \emph{discrete vertex}) if
$$
\theta(\bm\nu^+(\jv),\bm\nu^-(\jv))<0
$$
where 
$\theta$ is introduced in Definition \ref{angle}.
\end{definition}

\begin{figure}[htp]
\centering
\includegraphics[scale=.8]{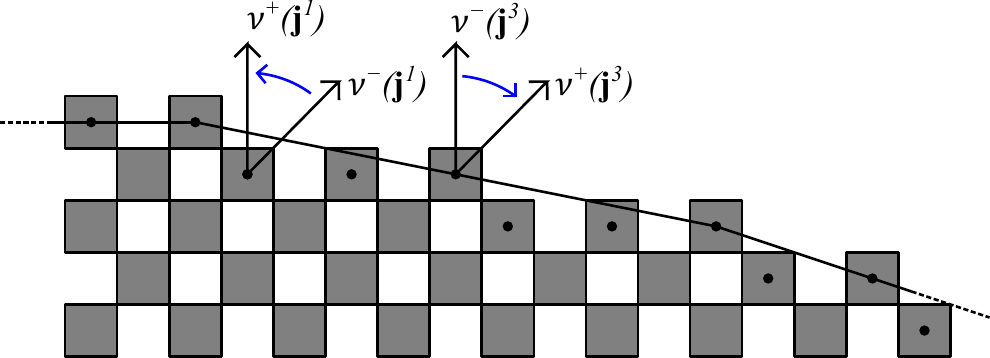}
\caption{A discrete vertex of $E$ may be a boundary point (not a vertex) of $\conv(Z(E))$.}
\label{fig:disc-edge-pol1}
\end{figure}
\begin{remark}[Vertices and discrete vertices]
The definition of discrete vertex given above is motivated by the fact that the vertices of $\conv(Z(E))$ are discrete (convex) vertices of $E$.
Whereas, points $\jv\in\partial^{\rm eff}E$ such that
$$
\theta(\bm\nu^+(\jv),\bm\nu^-(\jv))>0
$$
are always contained in the interior of $\conv(Z(E))$ (Figure \ref{fig:disc-edge-pol1}).
This choice will also facilitate the definition of \emph{discrete edge} (see Definition \ref{disc-edge-def} below).

\begin{figure}[htp]
\centering
\includegraphics[scale=.8]{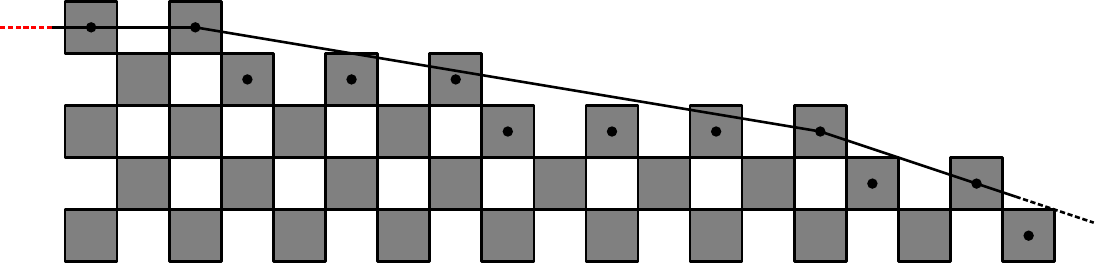}
\caption{An example of a discrete vertex of $E$ contained in the interior of $\conv(Z(E))$.}
\label{fig:disc-edge-pol2}
\end{figure}
Note that we may have discrete vertices of $E$ lying on the boundary of $\conv(Z(E))$ which are not vertices of $\conv(Z(E))$ (see Fig.~\ref{fig:disc-edge-pol1}), and discrete vertices of $E$ in the interior of $\conv(Z(E))$, as well (see Fig.~\ref{fig:disc-edge-pol2}).
\end{remark}

\begin{definition}[Discrete edges]\label{disc-edge-def}
Let $E\in\A_{\rm conv}$ satisfy \eqref{non-degeneracy}. 
We define a \emph{discrete edge} as a set of consecutive points of $\partial^{\rm eff} E$, say $\ell=\{\jv^l\}_{l=0}^L$ where $L\ge2$ and $\jv^0$ and $\jv^L$ are discrete vertices.
We define the \emph{outward unit normal vector} of the discrete edge $\ell$ as
$$\bm\nu(\ell) := \frac{(j_2^0-j^L_2,j^L_1-j_1^0)}{\sqrt{(j^L_2-j_2^0)^2+(j^L_1-j^0_1)^2}}\,.$$
We denote by $\mathcal{E}(E)$ the set of all discrete edges $\ell\subset \partial^{\rm eff}E$.
\end{definition}

Let $E\in\A_{\rm conv}$ satisfy \eqref{non-degeneracy}.
For every $\ell\in\partial^{\rm eff}E$ we define the \emph{slope} of $\ell$ as
\begin{equation}
s(\ell):=\frac{\nu(\ell)_1}{\nu(\ell)_2}\in[-\infty,+\infty]\,,
\label{eq:slope}
\end{equation}
where $\nu(\ell)_k$, $k=1,2$ indicate the components of $\bm\nu(\ell)$, with the convention that $\frac{\pm1}{0}=\pm\infty$.

\begin{figure}[h]
\begin{minipage}{0.5\linewidth}
\begin{center}
\includegraphics[scale=.8]{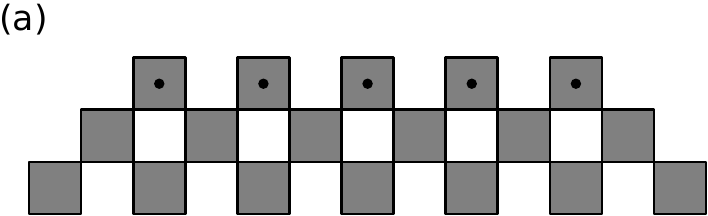}
\end{center}
\end{minipage}
\begin{minipage}{0.5\linewidth}
\begin{center}
\includegraphics[scale=.8]{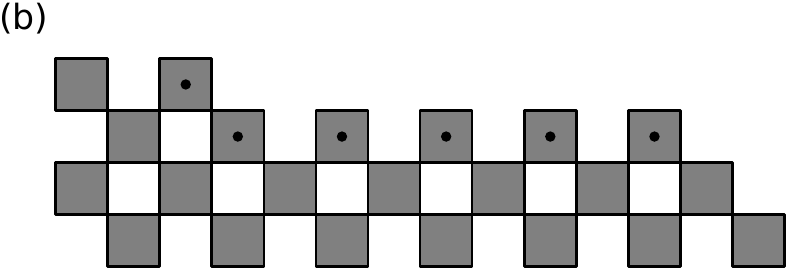}
\end{center}
\end{minipage}

\vspace{.5cm}

\begin{minipage}{0.5\linewidth}
\begin{center}
\includegraphics[scale=.8]{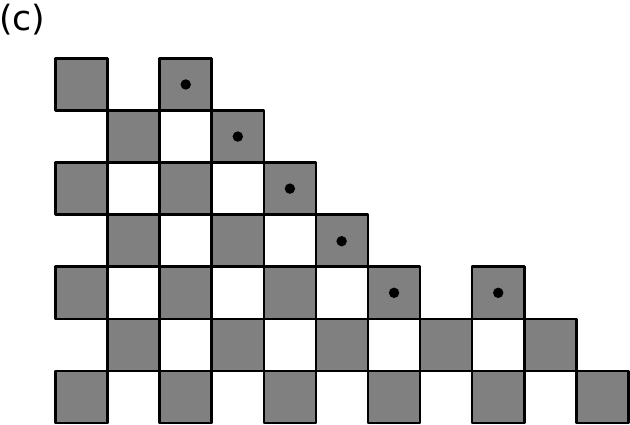}
\end{center}
\end{minipage}
\begin{minipage}{0.5\linewidth}
\begin{center}
\includegraphics[scale=.8]{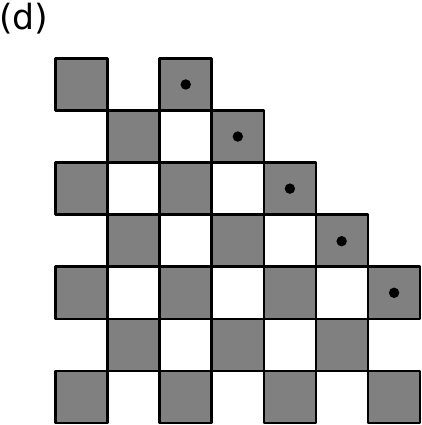}
\end{center}
\end{minipage}
\caption{Some examples of discrete edges.}
\label{fig:disc-edge}
\end{figure}

\begin{remark}\label{flat-slant-edge}
We list all the possible cases of discrete edges of sets $E\in\A_{\conv}$ satisfying \eqref{non-degeneracy} that are symmetric with respect to the axes and the bisectors $x_2=\pm x_1$.
Such symmetric sets will play a central role in the sequel of the paper.
Up to rotations of angle $k\pi$ and reflections we can restrict this characterization to discrete edges $\ell\in\mathcal{E}(E)$ such that $\ell=\{\jv^l\}_{l=0}^L\subset\{\x\in\R^2 \,:\, x_2>0\}$ having $s(\ell)\in[0,1]$.
We have the following characterization:

\smallskip
(i) if $s(\ell)=0$ then $\jv^l=\jv^{l-1}+(2,0)$ for every $1\le l\le L$;

(ii) if $s(\ell)\in(0,\frac{1}{3}]$ 
then $\jv^l=\jv^{l-1}+(2,0)$ for every $1< l\le L$ and $\jv^1=\jv^0+(1,-1)$;

(iii) if $s(\ell)\in(\frac{1}{3},1)$
then $\jv^l=\jv^{l-1}+(1,-1)$ for every $1\le l<L$ and $\jv^L=\jv^{L-1}+(2,0)$;

(iv) if $s(\ell)=1$ then $\jv^l=\jv^{l-1}+(1,-1)$ for every $1\le l\le L$.

These four types of discrete edge are pictured in Fig.~\ref{fig:disc-edge}(a), (b), (c) and (d), respectively. 
\end{remark}


\begin{definition}\label{def:projection}
For every norm $\varphi$ and every $E\in\D$, we introduce the {\em projection map of integer points on $E$}; that is, the set-valued map $\pi_E^\varphi:\Z^2\to{\mathcal{P}(\Z^2)}$ defined as
\begin{equation}\label{proj}
\pi_E^\varphi(\jv):=\argmin{\jv'\in Z(E)}\varphi(\jv-\jv')\,.
\end{equation}
\end{definition}

\subsection{Minkowski sum of sets}\label{sec:Minkowski}
We recall that the Minkowski sum of sets $A$ and $B$ is defined as $A+B=\{a+b \,|\, a\in A,b\in B\}$, and $A+\emptyset=\emptyset$.
If $m\in\N$, we denote by $mA$ the set $\{ma \,|\, a\in A\}$ and, if $A$ is non-empty, we will often write $A[m]$ to indicate the sum $A+A+\cdots+A$ $m$-times.
Among the many properties of Minkowski sum, we recall the commutability of Minkowski sum and the compatibility to the operation of taking the convex hull; that is, 
\begin{equation}
\conv(A+B)=\conv(A)+\conv(B)\,.
\label{eq:comMinkowconv}
\end{equation}
We recall without proof a result about the Minkowski sum of two convex polygons (see, e.g., \cite{BDD}).

\begin{proposition}\label{min-sum-edges}
Let $A$ and $B$ be convex polygons in $\R^2$. Let $L_A:=\{l_{i,A}\}_{i=1,\dots,n}$ and $L_B:=\{l_{j,B}\}_{j=1,\dots,m}$ be the sets of the edges of $A$ and $B$, respectively. Let $\mathcal{V}_A:=\{\nu_{i,A}\}_{i=1,\dots,n}$ and $\mathcal{V}_B:=\{\nu_{j,B}\}_{j=1,\dots,m}$ be the sets of the outer normal vectors of $A$ and $B$, respectively. Then,
\begin{enumerate}
\item[\rm(i)] if $\mathcal{V}_A\cap\mathcal{V}_B=\emptyset$, then $L_{A+B}=L_A\cup L_B$ and $\mathcal{V}_{A+B}=\mathcal{V}_{A}\cup \mathcal{V}_{B}$;
\item[\rm(ii)] if $|\mathcal{V}_A\cap\mathcal{V}_B|=p$, $1\leq p\leq \min\{n,m\}$, then $|L_{A+B}|=n+m-p$. More precisely, if $\nu_{i,A}=\nu_{j,B}$ for some $i\in\{1,\dots,n\}$ and $j\in\{1,\dots,m\}$, then $l_{i,A}+l_{j,B}\in L_{A+B}$, $l_{i,A}\not\in L_{A+B}$, $l_{j,B}\not\in L_{A+B}$ and $\nu_{i,A}=\nu_{j,B}\in \mathcal{V}_{A+B}$. If, instead, $\nu_{i,A}\neq\nu_{j,B}$, then $l_{i,A}\in L_{A+B}$, $l_{j,B}\in L_{A+B}$, $\nu_{i,A}\in \mathcal{V}_{A+B}$ and $\nu_{j,B}\in \mathcal{V}_{A+B}$.
\end{enumerate}
In particular, if $A=B$, then $L_{A+A}=\{l_{i,A}+l_{i,A}\}_{i=1,\dots,n}$ and $\mathcal{V}_{A+A}=\mathcal{V}_A$.
\end{proposition}

%
%

\subsection{The lattice point-counting problem: $m$-fold Minkowski sums}\label{sec:counting}

Let $B=\{{\bf w}_1,{\bf w}_2\}$ be a basis of $\R^2$. The set
\begin{equation*}
\Lambda=\Lambda(B):=\{z_1 {\bf w}_1 + z_2 {\bf w}_2:\,\, z_1,z_2\in\Z\}
\end{equation*}
is called a \emph{lattice} of $\R^2$ with basis $B$. The corresponding \emph{fundamental cell} is defined as
\begin{equation*}
\{\mu_1{\bf w}_1+\mu_2{\bf w}_2:\,\,\mu_1,\mu_2\in[0,1)\}
\end{equation*}
whose area is $|{\rm det}(B)|$. It can be checked that the area of the fundamental cell is independent of the choice of the basis and is referred to as the \emph{determinant} of $\Lambda$, ${\rm det}(\Lambda)$. Lattices are additive subgroups of $\R^2$ and they are discrete sets. Examples of lattices are the standard lattice $\Z^2$, with basis $\{(1,0),(0,1)\}$ and $|{\rm det}(\Z^2)|=1$, and the ``checkerboard lattice'' $\Ze$, with basis $\{(-1,1),(1,1)\}$ and $|{\rm det}(\Ze)|=2$. $\Zo$ is not a lattice, since $(1,0)+(0,1)=(1,1)\not\in\Zo$.

It will be useful in the sequel to obtain an estimate on the number of the lattice points contained in $m\mathcal{Q}$, $m\in\N$ for $\mathcal{Q}$ lattice convex polygon. For this, we first recall a fundamental result for counting the lattice points in $\mathcal{Q}$.
\begin{theorem}[Pick's Theorem, \cite{pick}]
Let $\Lambda$ be any lattice in $\R^2$, let $\I\subset\Lambda$ be a finite set and $\mathcal{Q}={\rm conv}(\I)$. Then
\begin{equation}
\#(\mathcal{Q}\cap\Lambda)=\frac{1}{|{\rm det}(\Lambda)|}|\mathcal{Q}|+\frac{1}{2}\#(\partial\mathcal{Q}\cap\Lambda)+1,
\label{pick}
\end{equation}
where $|\mathcal{Q}|$ is the area of $\mathcal{Q}$ and $\partial\mathcal{Q}$ its topological boundary.
\end{theorem}

A non-trivial problem in discrete geometry is the comparison between the set of the lattice points contained in the homothetic copy $m\mathcal{Q}$ of a convex lattice polyhedron $\mathcal{Q}$ with the $m$-fold Minkowski sum $(\mathcal{Q}\cap\Z^n)[m]$, $n\geq2$ (see, e.g., \cite{LinRoc}). It will be sufficient for our purposes here to mention that in the two dimensional setting the two lattice sets coincide (see \cite[Corollary~2.4]{LinRoc}). Moreover, an inspection of the proof reveals that the result still holds if we replace $\Z^2$ with any two-dimensional lattice $\Lambda$.

\begin{proposition}\label{prop:iden}
Let $\Lambda$ be any lattice in $\R^2$, let $\I\subset\Lambda$ be a finite set and $\mathcal{Q}={\rm conv}(\I)$ be two-dimensional. Then the equality
\begin{equation}
(\mathcal{Q}\cap\Lambda)[m]=(m\mathcal{Q})\cap\Lambda
\label{iden}
\end{equation}
holds for every $m\in\N$.
\end{proposition}

Now, in view of Proposition~\ref{prop:iden} and by iterating formula (\ref{pick}), Pick's Theorem generalizes to $m\mathcal{Q}$, $m\geq1$, as
\begin{equation}
\#((m\mathcal{Q})\cap\Lambda)=\frac{1}{|{\rm det}(\Lambda)|}|\mathcal{Q}|m^2+\frac{1}{2}\#(\partial\mathcal{Q}\cap\Lambda)m+1\,.
\label{counting}
\end{equation}

\subsection{Submodularity and absolute norms}\label{sec:submodularity}

We briefly recall the concept of submodularity which is well known in discrete convex analysis (see, e.g., \cite[Ch.~2, eq. (2.17)]{Murota}).
Setting $\R^2_+:=\{\x=(x_1,x_2)\in\R^2 \,|\, x_1,x_2\ge 0\}$, for every $\x,\y\in \R^2$ we define
$$\x\vee \y:=(\max\{x_1,y_1\},\max\{x_2,y_2\}) \quad \mbox{ and }\quad \x\wedge \y:=(\min\{x_1,y_1\},\min\{x_2,y_2\})\,.$$
A function $f:\R^2_+\rightarrow\R$ is said to be \emph{submodular} if it satisfies the following inequality
\begin{equation}\label{submodularity}
f(\x\vee \y)+f(\x\wedge \y)\leq f(\x)+f(\y),\quad \mbox{for every $\x,\y\in\R^2_+$}\,.
\end{equation}
It is known (see \cite[Proposition 5]{MM}) that every positively homogeneous function defined in the cone $\R^2_+$ is subadditive if and only if it is submodular. In particular, this yields that every \emph{absolute norm} $\varphi$ (i.e., $\varphi(\x)$ depends only on $|x_1|$ and $|x_2|$) complies with \eqref{submodularity}.
We recall that an absolute norm is \emph{monotonic}:
\begin{equation}
|x_1|\leq|y_1| \mbox{\, and \,}  |x_2|\leq|y_2| \quad \mbox{imply}\quad \varphi({\bf x})\leq\varphi({\bf y})\,.
\label{eq:monotonic}
\end{equation}

\section{Setting of the problem} \label{sec:setting}
We will deal with \emph{negative discrete perimeters}; that is, the Euclidean perimeter functional (with negative sign) restricted to $\D_\e$ relaxed to the space $\mathcal{X}$.
Namely, we define the functionals $F_\e:\mathcal{X}\to(-\infty,+\infty]$ as
\begin{equation}\label{discrper}
F_\e(E) = \begin{cases}
-\mathcal{H}^1(\partial E) & E\in\D_\e \\
+\infty & \text{otherwise.}
\end{cases}
\end{equation}
Note that these energies are related to the corresponding interaction energies defined on lattice sets
$$F_\e^{\rm lat}(\I)=-\,\e\,\# \big\{(\iv,\jv)\in\e\Z^2\times\e\Z^2 \,|\, \iv\in\I, \, \jv\not\in\I, \, |\iv-\jv|=\e\big\},$$
where $\I\subset\e\Z^2$, and $F_\e^{\rm lat}(Z_\e(E))=F_\e(E)$.
The functionals $F_\e$, in turn, may be seen as \emph{nearest-neighbor} (NN) antiferromagnetic interaction energies associated to a lattice spin-system; \emph{i.e.}, given $u:\e\Z^2\to\{-1,1\}$ one defines
$$E_\e(u) = -\frac{\e}{4}\sum_{\substack{\iv,\jv\in\e\Z^2 \\ |\iv-\jv|=\e}} (u(\iv)-u(\jv))^2,$$
whence $F_\e(E(\{u=1\}))=E_\e(u)$. The asymptotic behavior as $\varepsilon\to0$ of energies like $F_\e$  has been studied, e.g., in \cite{ABC}.
%

Let $\varphi:\R^2\to[0,+\infty)$ be a norm.
For every pair of lattice sets $E,E'\in\D_\e$, we define the dissipations
\begin{equation}\label{dissipation-def}
D_\e^\varphi(E,E')=\e^2 \sum_{\iv\in Z_\e(E)\triangle Z_\e(E')}d_\e^\varphi(\iv,\partial Z_\e(E')),
\end{equation}
where, given $\I\subset\e\Z^2$, $d_\e^\varphi$ denotes the \emph{discrete distance} of any lattice point $\iv\in\e\Z^2$ to $\partial\I$ defined as
$$d_\e^\varphi(\iv,\partial\I)=
\begin{cases}
\inf\{\varphi(\iv-\jv) \,|\, \jv\in\I\} & \text{if } \iv\not\in\I \\
\inf\{\varphi(\iv-\jv) \,|\, \jv\in\e\Z^2\setminus\I\}&\text{if } \iv\in\I.
\end{cases}$$

\begin{remark}\label{integral-diss-remk}
In the sequel, the following integral formulation of the dissipation \eqref{dissipation-def} will be useful.
Indeed,
for every $E'\in\mathcal{X}$ we set
$d_\e^\varphi(\iv,\partial E')=d_\e^\varphi\big(\iv,\partial(E'\cap\e\Z^2)\big)$.
Furthermore, we can extend $d_\e^\varphi(\cdot,\partial E')$ to $\R^2$ by setting 
$d_\e^\varphi(\x,\partial E'):=d_\e^\varphi(\iv,\partial E')$ for $\x\in q_\e(\iv)$.
Thus, for every $E,E'\in\mathcal{X}$, let $E_\e,E'_\e\in\D_\e$ be the corresponding discretizations; \emph{i.e.}, $Z_\e(E_\e)=E\cap\e\Z^2$ and the same for $E'_\e$, we may write
\begin{equation*}
\int_{E\triangle E'} d_\e^\varphi(\x,\partial E')\,\mathrm{d}\x=\e^2 \sum_{\iv\in Z_\e(E_\e)\triangle Z_\e(E'_\e)}d_\e^\varphi(\iv,\partial E'_\e)=D_\e^\varphi(E_\e,E'_\e).
\end{equation*}
We will consider the dissipation in \eqref{dissipation-def} as defined on every pair of sets of finite perimeter; \emph{i.e.}, $D_\e^\varphi:\mathcal{X}\times\mathcal{X}\to[0+\infty]$.
\end{remark}

\subsection{The time-discrete minimization scheme with a monotonicity constraint} \label{sec:choicescal}
For any $\e>0$ and $\tau>0$, let $F_\e$ and $D_\e^\varphi$ be defined as in \eqref{discrper} and \eqref{dissipation-def}, respectively.
We introduce a discrete motion with underlying time step $\tau$ obtained by successive minimization.
At each time step we will minimize an energy $\mathcal{F}_{\e,\tau}^\varphi:\mathcal{X}\times\mathcal{X}\rightarrow(-\infty,+\infty]$ defined as
\begin{equation}
\mathcal{F}_{\e,\tau}^\varphi(E,F)=\e F_\e(E)+\frac{1}{\tau}D_\e^\varphi(E,F)\,,
\label{energy}
\end{equation}
with a \emph{monotonicity constraint} on the discrete trajectories. Namely, we recursively define an increasing (with respect to inclusion) sequence $E_{\e,\tau}^k$ in $\mathcal{D}_\e$ by requiring the following:
\begin{equation}\label{MM-scheme}
\begin{cases}
E_{\e,\tau}^0=q_\e, \\
E_{\e,\tau}^{k+1}\in\argmin{E\in\D_\e,\, E\supset E_{\e,\tau}^k}\mathcal{F}_{\e,\tau}^\varphi(E,E_{\e,\tau}^k), & k\ge0.
\end{cases}
\end{equation}
In some cases we will also analyze solutions of the corresponding \emph{unconstrained} scheme; that is,
\begin{equation}\label{MM-scheme-unc}
\begin{cases}
E_{\e,\tau}^0=q_\e, \\
E_{\e,\tau}^{k+1}\in\argmin{E\in\D_\e}\mathcal{F}_{\e,\tau}^\varphi(E,E_{\e,\tau}^k), & k\ge0,
\end{cases}
\end{equation}
in which the minimization problems are performed over the whole class $\D_\e$.
The discrete orbits associated to functionals $\mathcal{F}_{\varepsilon,\tau}^\varphi$ are thus defined by
\begin{equation}\label{piece}
E_{\e,\tau}(t):=E_{\e,\tau}^{\lfloor t/\tau\rfloor},\quad t>0.
\end{equation}

We say that a curve $E:[0,+\infty)\to\mathcal{X}$ is a \emph{minimizing movement} for the problem \eqref{MM-scheme} or \eqref{MM-scheme-unc} at \emph{regime} $\tau$-$\e$ if it is pointwise limit (in the Hausdorff topology) of discrete orbits $E_{\e,\tau}$, as $\e,\tau\to0$ up to subsequences.

\begin{remark}[choice of scaling]
The scale $\varepsilon$ in the energies $\varepsilon F_\varepsilon$ above is suggested by energetic considerations (see \cite[(6)-(7)]{BraSci14}) and leads to a non-trivial limit of the discrete solutions defined in \eqref{piece}.
This choice is motivated by the fact that $\e F_\e$ has a nontrivial $\Gamma$-limit, as we will show in Section~\ref{G-lim-subsec}.
The energy scaling may also be seen as a time scaling of the discrete flow generated by taking the relaxation on $\mathcal{D}_\e$ of the energy functional $-\mathcal{H}^1$ (see \cite[Section~10.2]{Bra13}).
\end{remark}

\section{Fast convergences and the emergence of a critical regime} \label{sec:fastconvergence}
As remarked in \cite[Ch.~8]{Bra13}, minimizing movements along families of functionals will depend in general on the \emph{regime} $\tau$-$\e$; in our case, on the ratio between the two parameters $\tau$ and $\e$ that characterizes the motion.
We first provide the following result that ensures a compactness property of the minimizers of the energies $\mathcal{F}_{\e,\tau}^\varphi$. In this section $\varphi$ denotes a general norm, without any restriction.

\begin{lemma}\label{compactness-lem}
Let $F_\e$ and $D_\e^\varphi$ be defined as in \eqref{discrper} and \eqref{dissipation-def}, respectively, and $\mathcal{F}_{\e,\tau}^\varphi$ be as in \eqref{energy}.
Let $E'\in\D_\e$ be an admissible set.
For every fixed $\tau>0$ consider
$$E_{\e,\tau}\in\argmin{E\in\mathcal{X}}\mathcal{F}_{\e,\tau}^\varphi(E,E')\,.$$
Then, $Z_\e(E_{\e,\tau})\subset E'+B_{4\tau}^\varphi$ and $d_{\mathcal{H}}\big(E_{\e,\tau}, E'+B_{4\tau}^\varphi \big)<3\sqrt{2}\e$ for $\e$ small enough.
\end{lemma}
\begin{proof}
For any $E\in\D_\e$, the variation of the energy $\mathcal{F}_{\e,\tau}^\varphi$ when removing a square of center $\iv\in\e\Z^2$ is
$$\mathcal{F}_{\e,\tau}^\varphi(E,E')-\mathcal{F}_{\e,\tau}^\varphi(E\setminus q_\e(\iv),E')\le 4\e^2-\frac{\e^2}{\tau}d_\e^\varphi(\iv,\partial E')$$
which is strictly negative when $d_\e^\varphi(\iv,\partial E')>4\tau$, thus implying that $Z_\e(E_{\e,\tau})\subset E'+B_{4\tau}^\varphi$.
Furthermore, since it is always convenient to add an isolated square $q_\e(\jv)$ if $\jv\in Z_\e(E'+B_{4\tau}^\varphi)$ then, for every $\jv\in3\e\Z^2\cap E'+B_{4\tau}^\varphi$ we must have
$E_{\e,\tau}\cap q_{3\e}(\jv)\not=\emptyset$,
Since otherwise $\mathcal{F}_{\e,\tau}^\varphi(E_{\e,\tau}\cup q_\e(\jv),E')<\mathcal{F}_{\e,\tau}^\varphi(E_{\e,\tau},E')$.
\end{proof}

\begin{remark}\label{tau-fast-remk}
The regime $\tau/\e\to0$ is completely characterized by the previous lemma.
Indeed, in this case, when $\tau$ and $\e$ are small enough, $B_{4\tau}\cap\e\Z^2=\{(0,0)\}$ and the minimizing movement is trivially $E(t)\equiv \{(0,0)\}$.
This degenerate evolution is called a \emph{pinned motion}. We will focus on such motions in Section~\ref{sec:limitmotion}, where we will also introduce a ``pinning threshold''.
\end{remark}

\subsection{$\Gamma$-convergence of interaction energies}\label{G-lim-subsec}

This section is devoted to the study of the asymptotic behavior of energies $\e F_\e$.
To this end, we associate to any admissible set $E\in\D_\e$ the corresponding characteristic function $\chi_E\in L^\infty(\R^2)$ and compute the $\Gamma$-limit with respect to the local weak$^*$-topology.
We then generalize energies in \eqref{discrper} by considering $F_\e:L^\infty(\R^2)\to(-\infty,+\infty]$ as
\begin{equation}\label{discrper-gen}
F_\e(u)=\begin{cases}
F_\e(E) & u=\chi_E,\, E\in\D_\e \\
+\infty & \text{otherwise,}
\end{cases}
\end{equation}
with a slight abuse of notation. 
\begin{theorem}\label{gamma-limit-per}
Let $F_\varepsilon$ be defined as in \eqref{discrper-gen}, and set $G_\e:=\varepsilon F_\varepsilon$. 
Then $G_\e$ $\Gamma$\hbox{-}converge as $\varepsilon\to0$ to the energy
$$G(u) = \begin{cases}\displaystyle
4\int_{\R^2}\Big(\Big|u(\x)-\frac{1}{2}\Big|-\frac{1}{2}\Big) \mathrm{d}\x & u\in L^\infty(\R^2;[0,1]) \\
+\infty & \text{otherwise,}
\end{cases}$$
with respect to the local weak$^*$-topology.
\end{theorem}
\begin{proof}
It will suffice to prove the result for $u\in L^\infty(\R^2;[0,1])$, otherwise the assertion is trivial. 
We can assume, without loss of generality, that $u$ has compact support, and let $E_\e\in\D_\e$ be a sequence of sets such that $\chi_{E_\e}$ locally weakly-$^*$ converge to $u$.

\begin{figure}[ht]
\begin{center}
\includegraphics[width=0.6\textwidth]{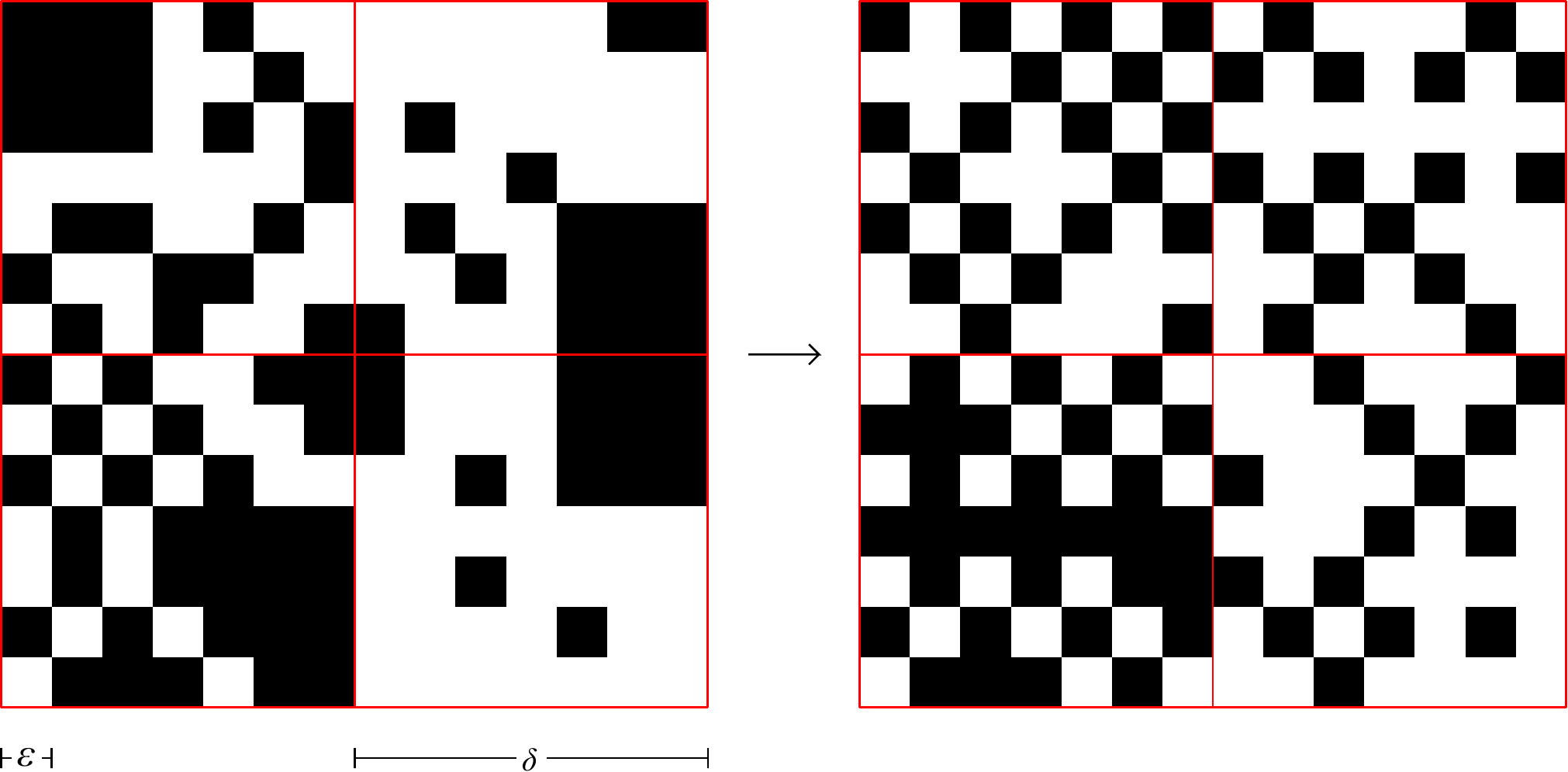}
\caption{On the left the set $E_\e$, on the right we exhibit a set $E_\e^\delta$ satisfying (i) and (ii).}
\label{fig-thm5}
\end{center}
\end{figure}
We now provide a rearrangement of the centers of $E_\e$ which is energy decreasing.
Let $\delta>0$ be fixed.
We consider the lattice $\delta\Z^2$ and sets $E_\e^\delta\in\D_\e$ satisfying $\#(Z_\e(E_\e^\delta)\cap q_\delta(\iv))=\#(Z_\e(E_\e)\cap q_\delta(\iv))$ and with the following properties:

(i) if $\e^2 \#(Z_\e(E_\e)\cap q_\delta(\iv))\le\delta^2/2$ then $Z_\e(E_\e^\delta\cap q_\delta(\iv))\subset\e\Z^2_e$,

(ii) if $\e^2 \#(Z_\e(E_\e)\cap q_\delta(\iv))>\delta^2/2$ then $Z_\e(E_\e^\delta\cap q_\delta(\iv))\supset\e\Z^2_e\cap q_\delta(\iv)$,

\noindent
for every $\iv\in\delta\Z^2$ (see Figure \ref{fig-thm5}).

Now, for every $E\in\D_\e$ and $F\in\mathcal{X}$ we define
$$
F_\e(E;F) = F_\e(E\cap E(\e\Z^2\cap F))\,,
$$
and analogously $G_\e(E;F)$.
In both cases (i) and (ii) we have $F_\e(E_\e;q_\delta(\iv))\ge F_\e(E_\e^\delta;q_\delta(\iv))$.
Since the contribution of the interaction between two adjacent $\delta$-squares $q_\delta(\iv)$ and $q_\delta(\jv)$ is less than $2\delta\e$ and the number of $\delta$-squares whose intersection with $\supp(u)\not=\emptyset$ is proportional to $1/\delta^2$, we get
$$G_\e(E_\e) \ge G_\e(E_\e^\delta)-C\frac{\e}{\delta}$$
for some positive constant $C$.
Now, from the convergence of $\chi_{E_\e}$ to $u$, for every $\iv\in\delta\Z^2$ we get
\begin{align}
&G_\e(E_\e^\delta;q_\delta(\iv))+O(\e)=-4|E_\e\cap q_\delta(\iv)|+O(\e)=-4\int_{q_\delta(\iv)}u(\x)\,\mathrm{d}\x:=u_\delta(\iv)\,, \label{Chapter3:(1)}\\
&G_\e(E_\e^\delta;q_\delta(\iv))+O(\e)=-4|q_\delta(\iv)\setminus E_\e|+O(\e)=-4\int_{q_\delta(\iv)}(1-u(\x))\,\mathrm{d}\x:=u_\delta(\iv)\,\label{Chapter3:(2)}
\end{align}
in cases (i) and (ii), respectively. After identifying $u_\delta$ with its piecewise-constant interpolation, taking the limit as $\e\to0$ first, we get
$$\liminf_{\e\to0} G_\e(E_\e) \ge \int_{\R^2}u_\delta(\x)\,\mathrm{d}\x\,,$$
and then taking the limit as $\delta\to0$ we obtain the liminf inequality.

The construction of a recovery sequence follows an analogous argument. Let $u\in L^\infty(\R^2;[0,1])$ have a compact support.
Consider the lattice $\sqrt{\e}\Z^2$
and define
$$u_\e(\iv) = \frac{1}{\e}\int_{q_{\sqrt{\e}}(\iv)} u(\x)\, \mathrm{d}\x\, , \quad \text{for every } \iv\in\sqrt{\e}\Z^2.$$
As a recovery sequence we will choose $E_\e$ having the same mean (unless a small error) of $u$ in every $\sqrt{\e}$-square with maximal perimeter term.
Indeed, we can take a set $E_\e\in\D_\e$ satisfying $\#\big(Z_\e(E_\e)\cap q_{\sqrt{\e}}(\iv)\big) = \lceil u_\e(\iv)/\e\rceil$ and such that:
\smallskip

(i) if $u_\e(\iv)\le 1/2$, then $Z_\e(E_\e)\cap q_{\sqrt{\e}}(\iv)\subset\e\Z^2_e$;

(ii) if $u_\e(\iv)>1/2$, then $Z_\e(E_\e)\cap q_{\sqrt{\e}}(\iv)\supset\e\Z^2_e\cap q_{\sqrt{\e}}(\iv)$.

\noindent Then, $\chi_{E_\e}$ weakly-$^*$ converge to $u$ and
$$
G_\e(E_\e;q_{\sqrt{\e}}(\iv)) + O(\e) = \begin{cases}\displaystyle
-4\int_{q_{\sqrt{\e}}(\iv)}u(\x) \mathrm{d}\x & \text{if } u_\e(\iv)\le \frac{1}{2} 
\\
\displaystyle
-4\int_{q_{\sqrt{\e}}(\iv)}(1-u(\x)) \mathrm{d}\x & \text{if } u_\e(\iv)> \frac{1}{2}
\end{cases}$$
for every $\iv\in\sqrt{\e}\Z^2$,
which proves that $\chi_{E_\e}$ is a recovery sequence and concludes the proof.
\end{proof}

\begin{remark}\label{recseq-remk}
Note that, in the proof of Theorem \ref{gamma-limit-per} we have exhibited a recovery sequence whose supports $E_\e$ also converges to $E=\supp(u)$ in the Hausdorff sense.
This remark allows us to reduce the computation of the $\Gamma$-limit of $G_\e$ to functions weakly-$^*$ converging to $u$ having supports in $\D_\e$ converging to $E$ with respect to the Hausdorff distance.
\end{remark}

\begin{remark}[$\Gamma$-limit on characteristic functions]\label{remk-haus-top}
An immediate consequence of Theorem \ref{gamma-limit-per} is that, among all the functions having the same support $E$, the ground state of the energy $G$ is achieved by the simple function $1/2\,\chi_E$.
In particular, since any family of sets $\{E_\e\}$ converging in the Hausdorff sense to $E$ are such that $\chi_{E_\e}$ is weakly-$^*$ compact, from Theorem \ref{gamma-limit-per} we infer that
$$\Gamma(d_{\mathcal{H}})\text{-}\lim_{\e\to0} G_\e(E)=-2|E|,$$
once noted that the recovery sequences are $\e$-checkerboard sets.
\end{remark}

\subsection{Convergence of the minimizing-movement scheme}
We prove that when $\e/\tau\to0$ every minimizing movement of scheme \eqref{MM-scheme-unc} may be seen as the solution of a continuum problem having a gradient-flow structure with respect to the limit energy.
In this regime the monotonicity constraint is not needed to obtain a completely characterized limit motion.
A straightforward consequence is that the solution of the unconstrained problem corresponds to that of the monotone scheme \eqref{MM-scheme}.

\begin{theorem}\label{fast-conv-thm}
Let $F_\e, D_\e^\varphi$ and $\mathcal{F}_{\e,\tau}^\varphi$ be as in \eqref{discrper}, \eqref{dissipation-def} and \eqref{energy}, respectively.
Then there exists a unique minimizing movement of the unconstrained scheme \eqref{MM-scheme-unc}
at regime $\e/\tau\to0$ and it satisfies
$$E(t)=B_{4t}^\varphi, \quad t\ge0.$$
Moreover, for every discrete solution $E_{\e,\tau}$ of \eqref{MM-scheme-unc} we have
$\chi_{E_{\e,\tau}(t)}\weakcs\frac{1}{2}\chi_{B_{4t}^\varphi}$ for all $ t\ge0$ as $\e\to0$.
\end{theorem}
\begin{proof}
The first claim is a direct consequence of Lemma \ref{compactness-lem}.
Indeed, $d_{\mathcal{H}}(E_{\e,\tau}(t),B_{4\lfloor t/\tau\rfloor})<C\lfloor t/\tau\rfloor \e$, which goes to zero locally uniformly at regimes $\e/\tau\to0$.
%
In an analogous way as for \eqref{discrper-gen}, we further generalize the dissipations in Remark \ref{integral-diss-remk} as functionals $D_\e^\varphi:L^\infty(\R^2)\times \mathcal{X}\to[0,+\infty]$ defined by
$$D_\e^\varphi(u,E')=\begin{cases}
D_\e^\varphi(E,E') & u=\chi_E,\, E\in\D_\e \\
+\infty & \text{otherwise.}
\end{cases}$$
Accordingly, we write $\mathcal{F}_{\e,\tau}^\varphi(u,E')=\e F_\e(u)+{1\over\tau} \, D_\e^\varphi(u,E')$ for every $u\in L^\infty(\R^2)$ with $F_\e$ as in \eqref{discrper-gen}.
Since for every sequence $\{E_\e\}\subset\D_\e$
such that $\chi_{E_\e}$ weakly-$^*$ converge to $u$ we have
$$D_\e^\varphi(E_\e,q_\e)\to\int_{\R^2} u(\x)\varphi(\x)\,\mathrm{d}\x,$$
then Theorem \ref{gamma-limit-per} yields that $\mathcal{F}_{\e,\tau}^\varphi$ $\Gamma$\hbox{-}converge, as $\e\to0$, to the functional $\mathcal{F}^\varphi_\tau$ given by
\begin{equation}\label{glim-induction0}
\mathcal{F}^\varphi_\tau(u) := \int_{\R^2} \Big(|4u(\x)-2|-2+\frac{1}{\tau}u(\x)\varphi(\x)\Big)\,\mathrm{d}\x,
\end{equation}
with respect to the weak-$^*$ topology.
Energy $\mathcal{F}^\varphi$ has a unique minimizer in $L^\infty(\R^2;[0,1])$, given by $u=1/2\,\chi_{B_{4\tau}^\varphi}$. 
Indeed
\begin{multline*}
\int_{\R^2} \Big(|4u(\x)-2|-2+\frac{1}{\tau}u(\x)\varphi(\x)\Big)\,\mathrm{d}\x = \int_{\{u\le1/2\}} \Big(\frac{\varphi(\x)}{\tau}-4\Big)u(\x)\,\mathrm{d}\x \\
+\int_{\{u>1/2\}} \Big(4u(\x)-4+\frac{\varphi(\x)}{\tau}u(\x)\Big)\,\mathrm{d}\x.
\end{multline*}
Both integrands are positive for almost every $\varphi(\x)>4\tau$ and are minimized when $u\equiv 1/2$.
Then, since $\Gamma$-convergence implies the convergence of minimum problems (see for instance \cite[Theorem~1.21]{BraGamma}) and the minimum is unique, we get that $\chi_{E_{\e,\tau}^1}$ weakly-$^*$ converges to $u_\tau^1=1/2 \chi_{B_{4\tau}^\varphi}$ as $\e\to0$.
Note also that, by virtue of Lemma~\ref{compactness-lem}, $E_{\e,\tau}^1\to B_{4\tau}^\varphi$ in the Hausdorff sense and moreover by the minimality of $E_{\e,\tau}^1$ and Remark \ref{integral-diss-remk} follows that
\begin{equation}\label{rec-seq}
\begin{aligned}
\e F_\e(E_{\e,\tau}^1) &\le \e F_\e(E_\e) + \frac{1}{\tau} \big( D_\e^\varphi(E_\e,q_\e) - D_\e^\varphi(E_{\e,\tau}^1,q_\e) \big) \\
&\le \e F_\e(E_\e) + \frac{1}{\tau}\int_{\R^2} \big(\chi_{E_\e}(x)-\chi_{E_{\e,\tau}^1}(\x)\big) \big(\varphi(\x)+\e\big) \mathrm{d}\x \le \e F_\e(E_\e)+o(1),
\end{aligned}
\end{equation}
for every $\chi_{E_\e}$ weakly$^*$ converging to $u_\tau^1$.

Now we show the $\Gamma$-convergence of $\mathcal{F}_{\e,\tau}(\cdot,E_{\e,\tau}^1)$, which will allow us to deduce the convergence of the whole scheme by an inductive procedure.
Consider $E_\e\in\D_\e$ such that $\chi_{E_\e}$ are converging weakly-$^*$ to some $u\in L^\infty(\R^2)$.
Mimicking the arguments of the proof of Theorem \ref{gamma-limit-per}, we consider $E_\e'\in\D_\e$ satisfying $\#(Z_\e(E_\e')\cap q_{\sqrt{\e}}(\iv))=\#(Z_\e(E_\e)\cap q_{\sqrt{\e}}(\iv))$ and such that:

(i) if $\#(Z_\e(E_\e)\cap q_{\sqrt{\e}}(\iv))\le \#\big(Z_\e(E_{\e,\tau}^1)\cap q_{\sqrt{\e}}(\iv)\big)$ then $Z_\e(E_\e'\cap q_{\sqrt{\e}}(\iv))\subset Z_\e(E_{\e,\tau}^1)$,

(ii)
 if $\#(Z_\e(E_\e)\cap q_{\sqrt{\e}}(\iv))>\#\big(Z_\e(E_{\e,\tau}^1)\cap q_{\sqrt{\e}}(\iv)\big)$ then $Z_\e(E_\e'\cap q_{\sqrt{\e}}(\iv))\supset Z_\e(E_{\e,\tau}^1)\cap q_{\sqrt{\e}}(\iv)$,

\noindent
for every $\iv\in\sqrt{\e}\Z^2\cap B_{4\tau}^\varphi$, and $Z_\e(E_\e')\setminus B_{4\tau}^\varphi=Z_\e(E_\e)\setminus B_{4\tau}^\varphi$.
Reasoning as in the proof of Theorem \ref{gamma-limit-per} and from \eqref{rec-seq}, $\chi_{E_\e'}$ still weakly-$^*$ converges to $u$ and $\e F_\e(E_\e)+o(1)\ge \e F_\e(E_\e')$.
Then we get
\begin{align*}
\mathcal{F}_{\e,\tau}^\varphi(E_\e,E^1_{\e,\tau})+o(1) &\ge \mathcal{F}_{\e,\tau}^\varphi(E'_\e,E^1_{\e,\tau}) \\
&= \e F_\e(E_\e')+\frac{1}{\tau}D_\e^\varphi(E_\e\setminus B_{4\tau}^\varphi,E_{\e,\tau}^1)+\frac{1}{\tau}\sum_{i\in\sqrt{\e}\Z^2\cap B_{4\tau}^\varphi}D_\e^\varphi(E_\e'\cap q_{\sqrt{\e}}(\iv),E_{\e,\tau}^1).
\end{align*}
Since $D_\e^\varphi(E_\e'\cap q_{\sqrt{\e}}(\iv),E_{\e,\tau}^1)=C\e^3|\#Z_\e(E_\e)-\#Z_\e(E_{\e,\tau}^1)|+O(\e^2)$ and $d_\e^\varphi(\x,\partial E_{\e,\tau}^1)$ converge uniformly to $d^\varphi(\x,B_{4\tau}^\varphi)$ for every $\x\not\in E_\tau^1$, we get that
\begin{equation}\label{G-limit-scheme}
\Gamma\text{-}\lim_{\e\to0}\mathcal{F}_{\e,\tau}^\varphi(u,E^1_{\e,\tau}) = \int_{\R^2}\Big(|4u(\x)-2|-2+\frac{1}{\tau}u(\x)d^\varphi(\x,B_{4\tau}^\varphi)\Big)\mathrm{d}\x,
\end{equation}
since the same argument applies to every recovery sequence $E_\e$.
By arguing as above, we get $\chi_{E_{\e,\tau}^2}$ converge to $1/2\,\chi_{B_{8\tau}^\varphi}$ and by induction the result follows.
\end{proof}

Arguing as in the proof of Theorem \ref{fast-conv-thm} we obtain the following result.

\begin{corollary}\label{fast-conv-cor}
Let $F_\e, D_\e^\varphi$ and $\mathcal{F}_{\e,\tau}^\varphi$ be defined as in \eqref{discrper}--\eqref{energy}.
Then there exists a unique minimizing movement of scheme \eqref{MM-scheme} at regime $\e/\tau\to0$ and it satisfies
$E(t)=B_{4t}^\varphi$ for  $t\ge0$.
Moreover, for every discrete solution $E_{\e,\tau}$ of \eqref{MM-scheme} we have
$\chi_{E_{\e,\tau}(t)}\weakcs\frac{1}{2}\chi_{B_{4t}^\varphi}$ for $t\ge0$ as $\e\to0$.
\end{corollary}

\begin{remark}\label{e-fast-remk}
Arguing as in Remark \ref{remk-haus-top}, for any $E'\in\mathcal{X}$ and every $E_\e'$ converging to $E'$ in $d_\mathcal{H}$ such that $\e F(E_\e')\to-2|E'|$ we get, from \eqref{G-limit-scheme}, that
$$
\Gamma(d_\mathcal{H})\text{-}\lim_{\e\to0} \mathcal{F}_{\e,\tau}^\varphi(E,E'_\e) = \mathcal{F}_\tau^\varphi(E,E') := -2|E|+\frac{1}{2\tau}\int_{E\triangle E'} d^\varphi(x,E') dx.
$$
Note that the minima of $\mathcal{F}_\tau^\varphi(\cdot,E')$ are solutions of
$$
\Big(-2+\frac{1}{2\tau} d^\varphi(x,E')\Big)\nu_E(x)\mathcal{H}^1\mres\partial E = 0;
$$
that is, $E\in\mathcal{X}$ such that $d^\varphi(x,E')\equiv 4\tau$ for $\mathcal{H}^1$-almost every $x\in\partial E$.
This gives that the limit scheme
\begin{equation}\label{scheme-remk-fast}
\begin{cases}
E^0_\tau=\{(0,0)\}, \\
E^{k+1}_\tau\in\argmin{E\in\mathcal{X}}\mathcal{F}_\tau^\varphi(E,E^k_\tau).
\end{cases}
\end{equation}
is solved by $E^k_\tau = B_{4k\tau}^\varphi$.
Hence, by Theorem \ref{fast-conv-thm} and Corollary \ref{fast-conv-cor} the minimizing movements of schemes \eqref{MM-scheme} and \eqref{MM-scheme-unc} at regimes $\e/\tau\to0$ are solutions of limit scheme \eqref{scheme-remk-fast}.
\end{remark}

\section{The critical regime: a microscopic checkerboard structure}\label{sec:scaledcheckbd}
So far, we have shown that scheme \eqref{MM-scheme} is completely characterized in the regimes $\tau/\e\to0$ (Remark \ref{tau-fast-remk}) and $\e/\tau\to0$ (Remark \ref{e-fast-remk}).
Throughout this section we will study the regimes where $\e/\tau$ has a non-zero finite limit, which turn out to be richer of features than the others.

Without loss of generality we consider only the case $\e=\alpha\tau$, where $\alpha>0$ is a positive constant.
The main goal is to determine any solution to the iterative variational scheme \eqref{MM-scheme}.
Within this regime, instead of solving a family of schemes depending on $\e$, by a rescaling argument we can solve one minimization scheme in the unique environment $\Z^2$.
Indeed, for every $E,F\in\mathcal{D}_\e$, the energies defined in \eqref{energy} can be rewritten as
\begin{align*}
\mathcal{F}_{\e,\tau}^\varphi(E,F) &= -\e\mathcal{H}^1(\partial E) + \frac{1}{\tau} D_\e^\varphi(E,F) = -\e\mathcal{H}^1(E) + \frac{\e^2}{\tau} \sum_{i\in Z_\e(E)\triangle Z_\e(F)}d_\e^\varphi(\iv,\partial F) \\
&= \e \Big( -\mathcal{H}^1(\partial E) + \alpha \sum_{\iv\in Z_\e(E)\triangle Z_\e(F)}d_\e^\varphi(\iv,\partial F) \Big) = \e^2 \mathcal{F}_\alpha^\varphi\Big(\frac{1}{\e}E,\frac{1}{\e}F\Big),
\end{align*}
where we have defined $\mathcal{F}_\alpha^\varphi:\D\times\D\to\R$ as
\begin{equation}\label{scaled}
\mathcal{F}_\alpha^\varphi(E',F')=-\mathcal{H}^1(\partial E')+\alpha \sum_{\iv\in Z(E')\triangle Z(F')} d^\varphi(\iv,\partial F').
\end{equation}
Thus, the solutions of \eqref{MM-scheme} are $E_{\e,\tau}^k=\e E_\alpha^k$ for every $\e>0$, $k\in\N$, where $\{E_\alpha^k\}$ solves the scaled scheme
\begin{equation}\label{MM-scheme2}
\begin{cases}\displaystyle
E_\alpha^0=q, \\
E_\alpha^{k+1}\in\argmin{E\in\D,\, E\supset E_\alpha^k}\mathcal{F}_\alpha(E,E_\alpha^k), & k\ge0.
\end{cases}
\end{equation}

We will prove that scheme \eqref{MM-scheme2} has a unique solution $\{E_\alpha^k\}$ whenever $\alpha$ is outside a countable set (see Remark \ref{non-unique} below).
If $\alpha$ is greater than a threshold value $\tilde\alpha>0$, the corresponding solution is trivially $E_\alpha^k\equiv q$, and we will say that the motion is \emph{pinned}.
If instead $\alpha$ is below the \emph{pinning threshold} (see Definition~\ref{pin-def}) the solutions $\{E_\alpha^k\}$ have a checkerboard structure; that is, $E_\alpha^k\in\A^e$ for every $k\in\N$, and they are obtained by the iterative formula
$$
Z(E_\alpha^{k+1})=Z(E_\alpha^k)+Z(E_\alpha^1), \quad \text{for every } k\in\N,\, k\ge1.
$$
We call this process \emph{nucleation from the origin}, and the lattice set $Z(E_\alpha^1)$, which we call the \emph{nucleus} of the process, completely characterizes the motion.
The limit evolution will be a motion of expanding polygons with constant velocity; both the velocity and the shape of the limit sets will be a discretization (depending on $\alpha$) of those of the minimizing movement of \eqref{MM-scheme} at regime $\e/\tau\to0$ studied in Section \ref{sec:fastconvergence}. This result will be proven under a technical assumption on the ``convexity'' of the nucleus $Z(E_\alpha^1)$ (cf. \eqref{monotone-edges}) which will allow us to use a localization method to solve any minimization problem of the scheme \eqref{MM-scheme2}.

The following result is a rereading of Lemma \ref{compactness-lem} in the scaled setting.
We note that, as for Lemma \ref{compactness-lem}, the following result holds for every norm.

\begin{lemma}\label{lemball}
Let $\mathcal{F}_\alpha^\varphi:\D\times\D\to\R$ be as in \eqref{scaled}, where $\D$ is defined as in \eqref{unions} with $\e=1$.
Then, for any given $E'\in\D$ it holds that
\begin{equation}\nonumber
\mathcal{F}_\alpha^\varphi\big(E(\I),E'\big)\le \mathcal{F}_\alpha^\varphi(E,E'), \quad \text{where } \I=\Big\{\iv\in Z(E) \,:\, d^\varphi(\iv,\partial E')\le \frac{4}{\alpha}\Big\}
\end{equation}
for every $E\in\D$.
In particular, for every $\{E_\alpha^k\}$ discrete solution of the scheme \eqref{MM-scheme2}, there holds
\begin{equation}\label{disc-sol-bound}
Z(E^{k+1}_\alpha)\subset\Big\{\iv\in\Z^2 \,:\, d^\varphi(\iv,\partial E^{k}_\alpha)\le \frac{4}{\alpha}\Big\}, \quad \text{for every } k\in\N.
\end{equation}
\end{lemma}
\begin{proof}
The result immediately follows from the fact that for every $E'\in\D$ the variation of adding an isolated square to any $E\in\D$ is $\mathcal{F}_\alpha^\varphi(E\cup q(\iv),E')-\mathcal{F}_\alpha^\varphi(E,E')=-4+\alpha d^\varphi(\iv,\partial E').$
\end{proof}

\begin{remark}[non-uniqueness]\label{non-unique}
Note that for every $\iv\in\Z^2$ such that $d^\varphi(\iv,\partial E^{k}_\alpha)=\frac{4}{\alpha}$ (if any), the energy contribution of the square $q(\iv)$ is zero; that is,
$$\mathcal{F}_\alpha^\varphi(E^{k+1}_\alpha\cup q(\iv), E^k_\alpha)=\mathcal{F}_\alpha^\varphi(E^{k+1}_\alpha\setminus q(\iv),E^k_\alpha).$$
Therefore, in this case, there is non-uniqueness of solutions for the problem \eqref{MM-scheme2}.
Note that, if $\varphi(\x)=\frac{4}{\alpha}$ has no integer solutions then, by the periodicity of $\Z^2$, the same holds true for equation $d^\varphi(\x,\partial E)=\frac{4}{\alpha}$ for every $E\in\D$.
This in particular implies that the $k$-th minimization problem of the scheme \eqref{MM-scheme2} has non-unique solution if and only if the first minimization problem has non-unique solution.\end{remark}

With the previous remark in mind, we define the \emph{singular set} $\Lambda^\varphi$ as
\begin{equation}\label{singularset}
\Lambda^\varphi:=\Big\{\frac{4}{\varphi(\iv)} \,:\, \iv\in\Ze\setminus\{(0,0)\}\Big\}\,.
\end{equation}
Note that the set $\Lambda^\varphi$ is countable and has a unique accumulation point in $0$.

\begin{example}\label{bif-ex}
We take $\varphi$ as the $\ell^\infty$-norm and choose $\alpha=4$, so that $\alpha\in\Lambda^\varphi$ as defined in \eqref{singularset}.
In this case, the set of lattice points having zero energy is $\{\iv\in\Z^2 \,:\, \|i\|_\infty=1\}$.
This yields that $\mathcal{F}_\alpha^\varphi(q,q)=\mathcal{F}_\alpha^\varphi(E,q)=-4$ for every admissible set $E\subset q\cup \{q(\iv) \,:\, |i_1|=|i_2|=1\}$ which implies that the minimum of the first step of \eqref{MM-scheme2} is not unique.
As already noted in Remark~\ref{non-unique}, the same situation arises at each minimization step of the scheme \eqref{MM-scheme2}.

Without entering into the details, we may check that every parametrized family $E:[0,+\infty)\to\mathcal{X}$ of connected sets satisfying
\begin{equation}\label{eq:normalvelocity}
E(0)=\{(0,0)\},
\quad
E(t)\subset E(s) \text{ for every } t<s,
\quad
\|v_{\perp}(t)\|_\infty\le 4 \text{ for every } t\ge0,
\end{equation}
is a minimizing movement, where $v_\perp$ denotes the normal velocity of $\partial E(t)$.
Indeed, for every fixed $t>0$, from \eqref{eq:normalvelocity} we have $E(t)\subseteq [-4t,4t]^2$, since $E(t)$ is connected.
Then, for any $\tau>0$ define
$$E^k_{\e,\tau}:=E(E(k\tau)\cap\e\Ze).$$
Since $E(k\tau)\subseteq [-4k\tau,4k\tau]^2=[-k\e,k\e]^2$, $E^k_{\tau,\e}$ can be obtained by solving the first $k$ steps of \eqref{MM-scheme}.
The corresponding discrete solutions $E_{\e,\tau}(t)$ converge to $E(t)$ as $\e,\tau\to0$ in the Hausdorff sense for every $t>0$, whence $E(t)$ is a minimizing movement.
%
\end{example}

\begin{figure}[h]
\center\includegraphics[scale=.4]{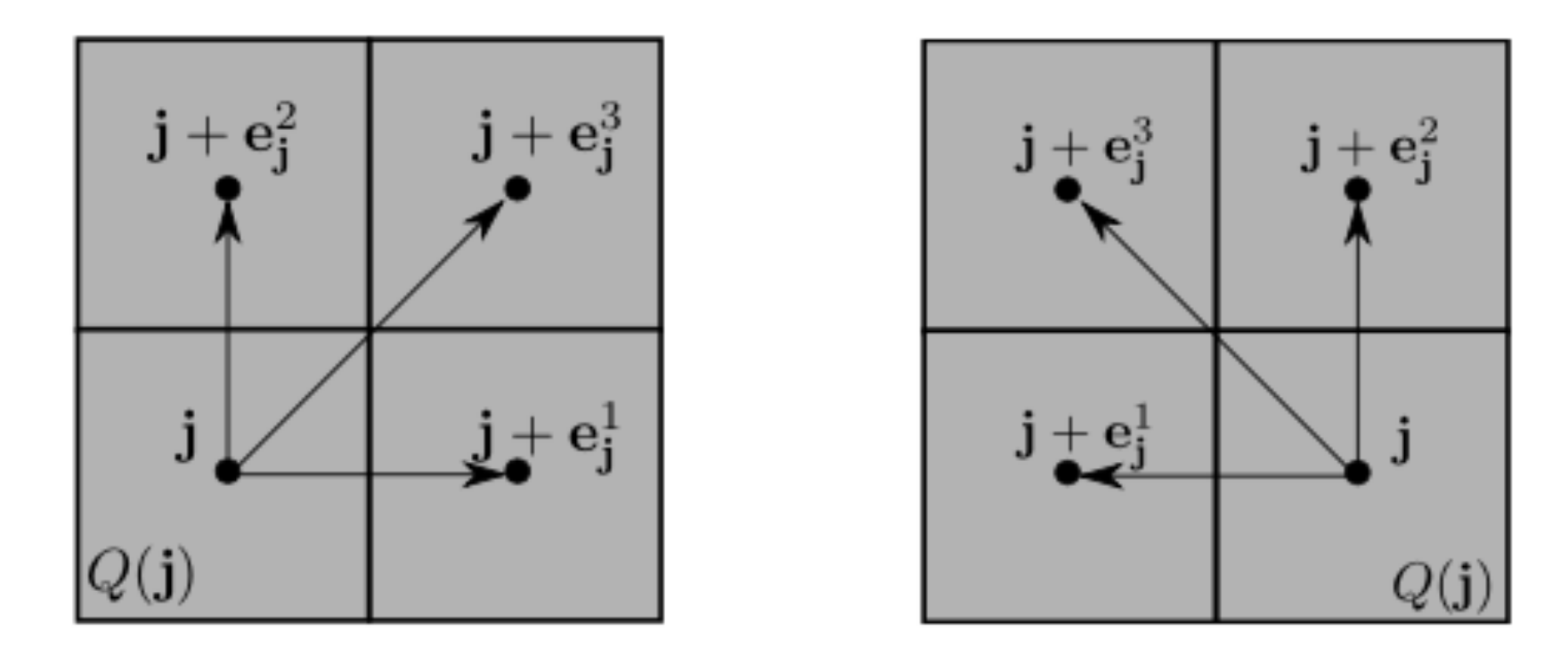}
\caption{Examples of $2\times2$ squares of the covering.
On the left the case $j_1,j_2>0$, on the right $j_1>0$, $j_2<0$.}
\label{fig:quadratone}
\end{figure}
\subsection{A localization argument: the $2\times2$-square tiling} \label{sec:bigsquares}

In order to determine the optimal structure of a minimizer, we will argue locally by defining the following covering of admissible sets.

\begin{definition}[$2\times2$-square coverings]\label{coverings}
For every $\jv=(j_1,j_2)\in\Z^2$, we define the vectors $\ev_\jv^1=(\sgn(j_1),0)$, $\ev_\jv^2=(0,\sgn(j_2))$, $\ev_\jv^3=\ev_\jv^1+\ev_\jv^2$ and, correspondingly, the $2\times2$ square (see Fig.~\ref{fig:quadratone})
\begin{equation}
Q(\jv):=q(\jv)\cup\bigcup_{k=1}^3q(\jv+\ev_\jv^k)\,.
\label{quadratone}
\end{equation}

\begin{figure}[htp]
\centering
\def\svgwidth{150pt}
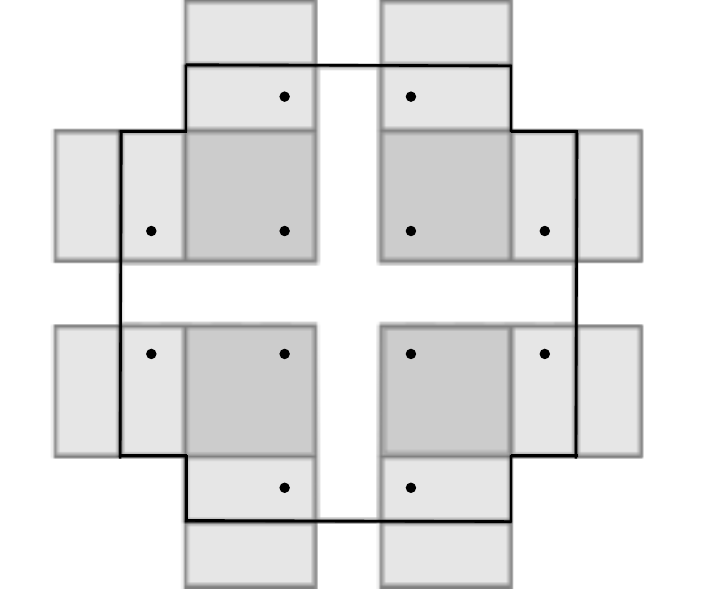
\caption{\small The picture clarifies the $2\times2$-square covering for a set $E$, whose boundary is marked by a bold black line. The darker $2\times2$ squares are in $\mathcal{S}_e^c(E)$, the lighter ones in $\mathcal{S}_e^b(E)$. The areas in white are those left uncovered.
}
\label{bsc}
\end{figure}
Let $E\in\D$ be an admissible set. Then, we define the family of sets
\begin{equation}
\mathcal{S}_e(E):=\big\{Q(\jv):\, \jv\in\Ze\mbox{ with } j_1, j_2 \mbox{ odd},\, Q(\jv)\cap E\neq\emptyset\big\},
\label{family1}
\end{equation}
which is a covering of non-overlapping squares of $E\setminus\mathcal{C}_0$, where $\mathcal{C}_0:=\bigcup\{q(\iv)\,|\,i_1i_2=0\}$ (see Fig.~\ref{bsc}).
We can subdivide the squares of $\mathcal{S}_e(E)$ in those contained in $E$ and those that are not, defining the partition $\mathcal{S}_e(E)=\mathcal{S}^b_e(E)\cup\mathcal{S}^c_e(E)$ where $\mathcal{S}^c_e(E)=\{Q(\jv)\in\mathcal{S}_e(E)\,|\,Q(\jv)\subseteq E\}$ and $\mathcal{S}^b_e(E)=\{Q(\jv)\in\mathcal{S}_e(E)\,|\,Q(\jv)\cap E^c\neq\emptyset\}$.
\end{definition}

\subsection{Choice of the dissipation term}\label{sec:norm}

We restrict our analysis to dissipations \eqref{dissipation-def} induced by an absolute norm $\varphi$; \emph{i.e.}, $\varphi(\x)$ depends only on $|x_1|$ and $|x_2|$, with the additional assumptions 
\begin{enumerate}[font={\normalfont},label={(H\arabic*)}]
\item $\varphi$ is symmetric (or permutation invariant); that is, $\varphi(x_1,x_2)=\varphi(x_2,x_1)$ for every $\x\in\R^2$;\label{ass-H1}
\item $\varphi$ complies with the normalization condition $\varphi(1,0)=\varphi(0,1)=1$. \label{ass-H2}
\end{enumerate}
We refer to an absolute norm with these properties as a \emph{symmetric absolute normalized norm}. The $\ell^p$-norms, $1\leq p\leq\infty$, are examples of such norms. 
This choice is of course motivated by the symmetry properties of the corresponding unit balls, which simplify the computations and the arguments of the proofs.
Moreover, as remarked in Section~\ref{sec:submodularity} an absolute norm is a submodular function on $\R^2$, a property that will be crucial in the sequel as it will
allow to reduce the main minimization problem to a finite number of local minimization problems, taking into account four-point interactions. Indeed, we can infer from \eqref{submodularity} a submodularity-type inequality involving only the norms of the four lattice points contained in any of the $2\times2$ squares of the coverings defined above. Namely, 


\begin{equation}
\varphi(\iv)+\varphi(\iv+\ev_\iv^3)\leq\varphi(\iv+\ev_\iv^1)+\varphi(\iv+\ev_\iv^2),
\label{normestimate}
\end{equation}
for every $\iv\in\Z^2$.

\subsection{The first step of the evolution: checkerboards nucleating from a point} \label{sec:coveringstep1}

With the covering argument of Section~\ref{sec:bigsquares} and the key norm inequality \eqref{normestimate} at hand, we are now in position to give the explicit characterization of the first step $E^1_\alpha$ of the discrete evolution, showing that it is an even checkerboard. 
A local analysis by means of the $2\times2$-square tilings will allow us to prove, with the following Proposition~\ref{firststep}, that the set of centers of $E^1_\alpha$ coincides with the discretization of the ball $B_{4\over\alpha}$ 
on the even lattice $\Z^2_e$.
We stress the generality of the following result, which only requires $\varphi$ to be an absolute norm without any additional assumption, in particular we do not assume {\ref{ass-H1}} and {\ref{ass-H2}}.
\begin{proposition}\label{firststep}
Let $\varphi$ be an absolute norm, let $\alpha>0$ be such that $\alpha\not\in\Lambda^\varphi$ and let $\mathcal{F}_\alpha^\varphi$ be as in \eqref{scaled}.
Then the first minimization problem of scheme \eqref{MM-scheme2} has a unique solution
$$E_\alpha^1 = \argmin{E\in\D,\, E\supset q}\mathcal{F}_\alpha^\varphi(E,q)$$
and it satisfies
\begin{equation}\label{firststep-eq}
E^1_\alpha=E(\Ze\cap B_{4\over\alpha}^\varphi)\,.
\end{equation}
In particular, $E^1_\alpha\in\A^e_{\conv}$.
\end{proposition}
\begin{proof}
The argument does not require the normalization assumption \ref{ass-H2}; we then set 
\begin{equation*}\label{constant-norm}
\varphi_{\max} := \max\{\varphi(1,0), \varphi(0,1)\}, \quad \varphi_{\min} := \min\{\varphi(1,0), \varphi(0,1)\},
\end{equation*}
and we assume, without loss of generality, that $\varphi_{\max}=\varphi(1,0)$.
Note that $\frac{4}{\varphi_{\rm min}}, \frac{4}{\varphi_{\rm max}}\in\Lambda^\varphi$.

If $\alpha>\frac{4}{\varphi_{\min}}$ we get $E^1_\alpha=q$ since $\mathcal{F}_\alpha^\varphi(q(\iv),q)>0$ for every $\iv\in\Z^2\setminus\{(0,0)\}$ and \eqref{firststep-eq} trivially holds.
If $\frac{4}{\varphi_{\max}}<\alpha<\frac{4}{\varphi_{\min}}$, we get that for any $\iv=(i_1,i_2)$ with $i_1\not=0$ there holds $\mathcal{F}_\alpha^\varphi(q(\iv),q)>0$, thus $Z(E_\alpha^1)\subset\{0\}\times\Z$.
\begin{figure}[h]
\centering
\includegraphics{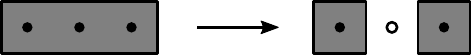}
\qquad\qquad
\includegraphics{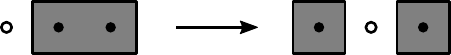}
\caption{Clusters of two or three lattice points are ``locally'' not energetically convenient.}
\label{fig:remov}
\end{figure}\\
Let $E\in\D$ be a competitor such that $Z(E)\subset\{0\}\times\Z$.
If $\iv\in Z(E)\setminus\{(0,0)\}$ has two nearest-neighbors, removing $q(\iv)$ leaves the total perimeter unchanged but decreases the dissipation (see Figure \ref{fig:remov}).
If instead $\iv$ has only one nearest-neighbor $\iv'\neq(0,0)$, if $|i_2|<|i_2'|$ then shifting $q(\iv)$ towards the origin does not decrease the perimeter but reduces the dissipation; if instead $|i_2|>|i_2'|$ the same holds shifting $q(\iv')$ (see Figure \ref{fig:remov}).
Hence, we may restrict our analysis to the two configurations $E\big(\Ze\cap B_\frac{4}{\alpha}^\varphi\big)$ and $E\big(\Zo\cap B_\frac{4}{\alpha}^\varphi\big)\cup q$.
A comparison between the two energy contributions yields that the variation from the odd checkerboard to the even one is less then $0$, thus $E^1_\alpha=E\big(\Ze\cap B_\frac{4}{\alpha}^\varphi\big)$.

Now let $\alpha<\frac{4}{\varphi_{\max}}$.
We consider the covering described in Definition~\ref{coverings}.
First, we note that the energy of every admissible set $E$ complies with the estimate
\begin{equation}\label{bound}
\mathcal{F}_\alpha^\varphi(E,q)\ge\sum_{Q(\jv)\in\mathcal{S}_e(\R^2)}\mathcal{F}_\alpha^\varphi(E\cap Q(\jv),q)+\mathcal{F}_\alpha^\varphi(E\cap \mathcal{C}_0,q)\,,
\end{equation}
the equality holding if and only if $\{E\cap Q(\jv)\}$ and $E\cap\mathcal{C}_0$ are non-overlapping; this is the case of sets $E$ having a checkerboard structure.
Inequality \eqref{bound} corresponds to localizing the energy, neglecting interactions between neighboring squares.

From Lemma \ref{lemball}, we can reduce our analysis to admissible sets contained in $E_{\alpha,\varphi}:=\Z^2\cap B_{4\over\alpha}^\varphi$ and inequality \eqref{bound} holds restricting the sum to every $Q(\jv)\in\mathcal{S}_e(E_{\alpha,\varphi})$ since $\mathcal{F}_\alpha^\varphi(q(\jv),q)>0$ for every $\varphi(\jv)>\frac{4}{\alpha}$.
We will prove that
\begin{equation}
\begin{aligned}
\min_{E\in\mathcal{D},\, E\supset q}\mathcal{F}_\alpha^\varphi(E\cap Q(\jv),q) &=\mathcal{F}_\alpha^\varphi(E(\Ze\cap B_\frac{4}{\alpha}^\varphi))\cap Q(\jv),q)\\
\min_{E\in\mathcal{D},\, E\supset q}\mathcal{F}_\alpha^\varphi(E\cap\mathcal{C}_0,q) &=\mathcal{F}_\alpha^\varphi(E(\Ze\cap B_\frac{4}{\alpha}^\varphi)\cap\mathcal{C}_0,q)
\end{aligned}
\label{minresult}
\end{equation}
for every $Q(\jv)\in\mathcal{S}_e(E_{\alpha,\varphi})$; that is, the optimal structure is an even checkerboard set in each of the following cases: (a) inside $Q(\jv)\in\mathcal{S}_e^c(E_{\alpha,\varphi})$; (b) inside $Q(\jv)\in\mathcal{S}_e^b(E_{\alpha,\varphi})$; (c) on $E_{\alpha,\varphi}\cap\mathcal{C}_0$.
In the sequel, $E$ will denote a general competitor $E\in\D$, $E\subset E_{\alpha\varphi}$.

\begin{figure}[h]
\centering
\def\svgwidth{150pt}
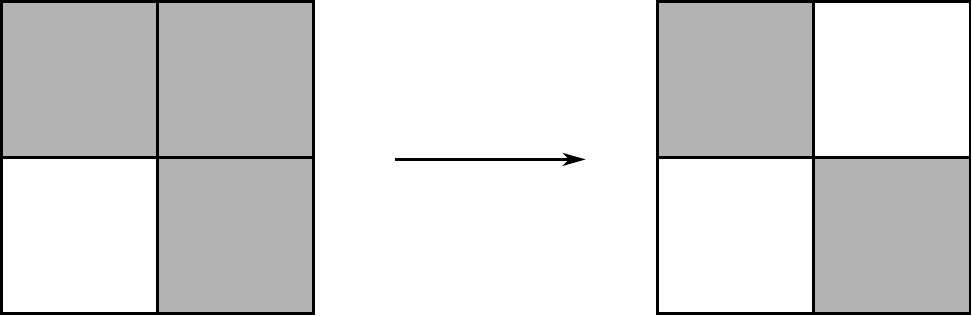
\caption{It is convenient to remove $q(\iv)$ if $\iv$ has two nearest-neighbors.}
\label{fig:step1}
\end{figure}

(a) Consider $Q(\jv)\in\mathcal{S}_e^c(E_{\alpha,\varphi})$ and let $q(\iv)\subset Q(\jv)\cap E$.
Note that the class $\mathcal{S}_e^c(E)$ is not empty if and only if $\alpha<\frac{\varphi_{\max}}{2}$.
Moreover, since adding an isolated square in $Q(\jv)$ is always energetically convenient, we can restrict to configurations of $Q(\jv)\cap E$ consisting of exactly two squares $q(\iv')$ and $q(\iv'')$ (see Fig.~\ref{fig:step1}).
Now, if $q(\iv')\cup q(\iv'')$ has no checkerboard structure; that is, $\iv'$ and $\iv''$ are nearest-neighbors, both the checkerboard configurations $E'$ and $E''$, containing $q(\iv')$ and $q(\iv'')$ respectively, decrease the energy.
Indeed, the corresponding variation of the energy is given by
\begin{equation*}
\mathcal{F}^\varphi_\alpha(E',q)-\mathcal{F}^\varphi_\alpha(q(\iv')\cup q(\iv''),q)\le-2+\alpha\varphi_{\max},
\quad
\mathcal{F}^\varphi_\alpha(E'',q)-\mathcal{F}^\varphi_\alpha(q(\iv')\cup q(\iv''),q)\le-2+\alpha\varphi_{\max}.
\end{equation*}
\begin{figure}[h]
\centering
\def\svgwidth{200pt}
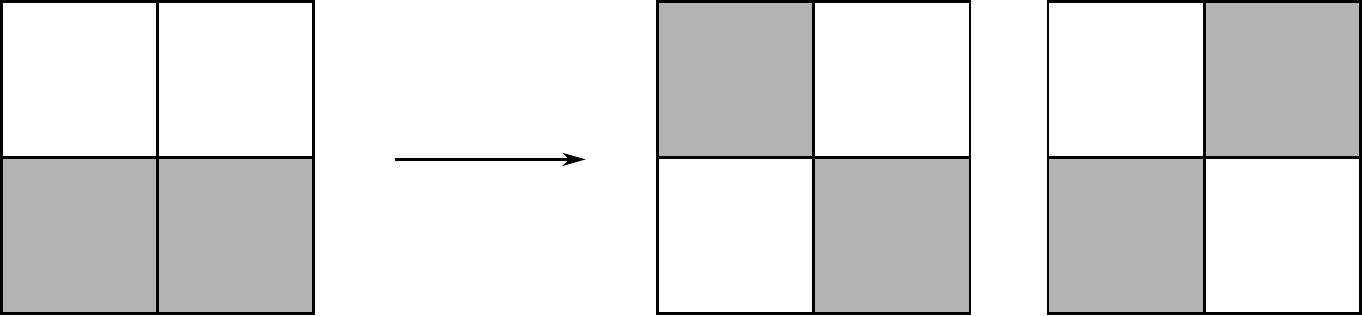
\caption{Any checkerboard configuration inside $Q(\jv)$ is a competitor with less energy.}
\label{fig:step2}
\end{figure}
This variation is never positive, since when $\alpha>\frac{2}{\varphi_{\max}}$ the class $\mathcal{S}_e(E_{\alpha,\varphi})$ is empty. 
Thus, any checkerboard configuration inside $Q(\jv)$ is a competitor with less energy then $E$ (see Fig.~\ref{fig:step2}).
Now we should compare the energies of the two possible checkerboard configurations inside $Q(\jv)$.
For this, we note that the variation of the energy in order to pass from the odd checkerboard configuration $q(\jv+\ev_\jv^1)\cup q(\jv+\ev_\jv^2)$ to the even one $q(\jv)\cup q(\jv+\ev_\jv^3)$ is
\begin{multline*}
\mathcal{F}_\alpha^\varphi(q(\jv)\cup q(\jv+\ev_\jv^3),q)-\mathcal{F}_\alpha^\varphi(q(\jv+\ev_\jv^1)\cup q(\jv+\ev_\jv^2),q) \\
=\alpha\big(\varphi(\jv)+\varphi(\jv+\ev_\jv^3)-\varphi(\jv+\ev_\jv^1)-\varphi(\jv+\ev_\jv^2)\big),
\end{multline*}
which is non-positive by \eqref{normestimate}.

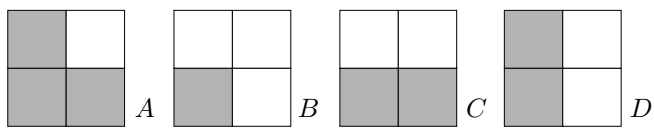
\begin{figure}[h]
\centering
\def\svgwidth{250pt}
\begingroup%
  \makeatletter%
  \providecommand\color[2][]{%
    \errmessage{(Inkscape) Color is used for the text in Inkscape, but the package 'color.sty' is not loaded}%
    \renewcommand\color[2][]{}%
  }%
  \providecommand\transparent[1]{%
    \errmessage{(Inkscape) Transparency is used (non-zero) for the text in Inkscape, but the package 'transparent.sty' is not loaded}%
    \renewcommand\transparent[1]{}%
  }%
  \providecommand\rotatebox[2]{#2}%
  \newcommand*\fsize{\dimexpr\f@size pt\relax}%
  \newcommand*\lineheight[1]{\fontsize{\fsize}{#1\fsize}\selectfont}%
  \ifx\svgwidth\undefined%
    \setlength{\unitlength}{514.45298082bp}%
    \ifx\svgscale\undefined%
      \relax%
    \else%
      \setlength{\unitlength}{\unitlength * \real{\svgscale}}%
    \fi%
  \else%
    \setlength{\unitlength}{\svgwidth}%
  \fi%
  \global\let\svgwidth\undefined%
  \global\let\svgscale\undefined%
  \makeatother%
  \begin{picture}(1,0.17640334)%
    \lineheight{1}%
    \setlength\tabcolsep{0pt}%
    \put(0,0){\includegraphics[width=\unitlength,page=1]{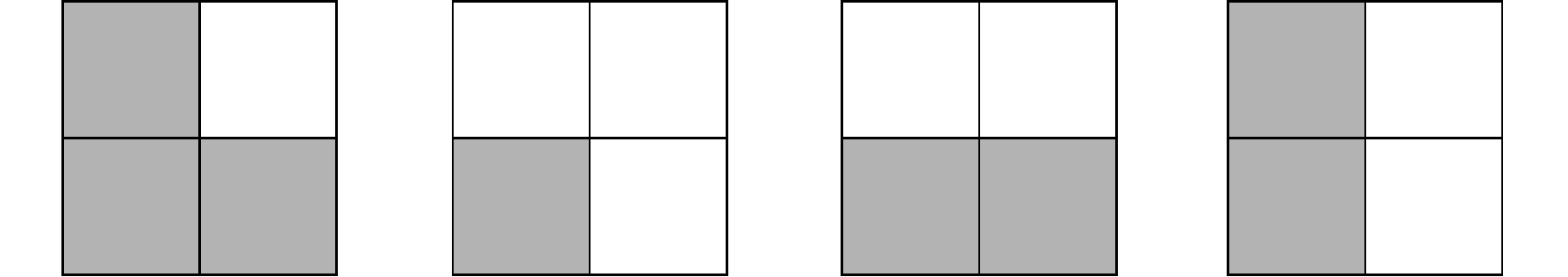}}%
    \put(0.23009327,0.01457859){\color[rgb]{0,0,0}\makebox(0,0)[lt]{\lineheight{1.25}\smash{\begin{tabular}[t]{l}$A$\end{tabular}}}}%
    \put(0.47402641,0.01457859){\color[rgb]{0,0,0}\makebox(0,0)[lt]{\lineheight{1.25}\smash{\begin{tabular}[t]{l}$B$\end{tabular}}}}%
    \put(0.72638626,0.01457859){\color[rgb]{0,0,0}\makebox(0,0)[lt]{\lineheight{1.25}\smash{\begin{tabular}[t]{l}$C$\end{tabular}}}}%
    \put(0.97158202,0.01457859){\color[rgb]{0,0,0}\makebox(0,0)[lt]{\lineheight{1.25}\smash{\begin{tabular}[t]{l}$D$\end{tabular}}}}%
  \end{picture}%
\endgroup%

\caption{The possible cases of $Q(\jv)\cap E_{\alpha,\varphi}$.}
\label{bsb}
\end{figure}
(b) Now, let $Q(\jv)\in\mathcal{S}_e^b(E_{\alpha,\varphi})$. Without loss of generality, we may assume that $j_1,j_2>0$, the situation being completely symmetric in the other cases.
Inside such a $2\times2$ square, we have four possible cases for $Q(\jv)\cap E_{\alpha,\varphi}$, as pictured in Fig.~\ref{bsb}.
We claim that the configuration with minimal energy inside $Q(\jv)$ is a checkerboard set.
Consider first $\alpha>\frac{2}{\varphi_{\max}}$, then $\iv\in B_\frac{4}{\alpha}^\varphi$ if and only if $|i_1|\le1$, thus the only possible cases for $Q(\jv)\cap E_{\alpha,\varphi}$ are those labeled by $B$ and $D$ in Fig.~\ref{bsb}.
Since
$$
\mathcal{F}_\alpha^\varphi(q(\jv),q)<\mathcal{F}_\alpha^\varphi(\ev_\jv^2,q),
\quad
\mathcal{F}_\alpha^\varphi(q(\jv),q)-\mathcal{F}_\alpha^\varphi(q(\jv)\cup q(\jv+\ev_\jv^2),q) = 2-\alpha\varphi_{\max} < 0
$$
in both cases the optimal configuration is $q(\jv)$.
Consider now $\alpha<\frac{2}{\varphi_{\max}}$.
Reasoning as before we can assume $Q(\jv)\cap E=q(\iv')\cup q(\iv'')$.
In cases $B$, $C$ and $D$, if $\iv'$ and $\iv''$ were nearest neighbors, with, \emph{e.g.}, $\varphi(\iv')>\varphi(\iv'')$, then removing $q(\iv')$ would produce a negative variation; that is,
$$\mathcal{F}^\varphi_\alpha(q(\iv''),q)-\mathcal{F}^\varphi_\alpha(q(\iv')\cup q(\iv''),q)\le 2-\alpha\varphi(\iv')<2-\alpha\Big(\frac{4}{\alpha}-\varphi_{\max}\Big)\le-2+\alpha\varphi_{\max}.$$
Thus the minimal configuration is the even checkerboard.
For what concerns the case $A$, since $\varphi(\jv+\ev_\jv^3)>\frac{4}{\alpha}$ and by (\ref{normestimate}) we have that
$$\mathcal{F}_\alpha^\varphi(q(\jv),q)<\mathcal{F}_\alpha^\varphi(q(\jv)\cup q(\jv+\ev_\jv^3),q)\le\mathcal{F}_\alpha^\varphi(q(\jv+\ev_\jv^1)\cup q(\jv+\ev_\jv^2),q)$$
which again leads to the result.

\begin{figure}[h]
\begin{center}
\includegraphics{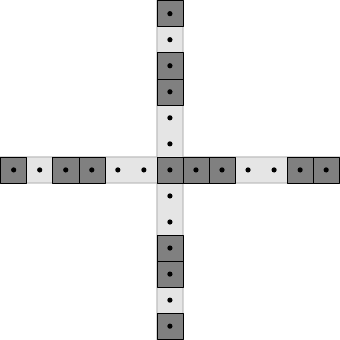}\qquad
\includegraphics{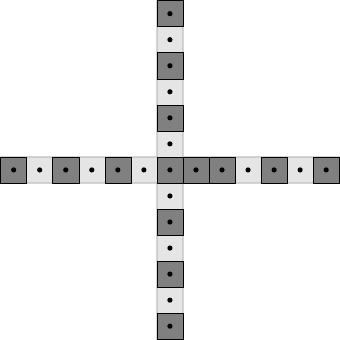}\qquad
\includegraphics{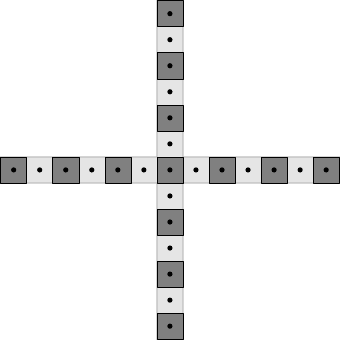}
\end{center}
\caption{Optimal configuration for $E_{\alpha,\varphi}\cap\mathcal{C}_0$.}
\label{fig:step3}
\end{figure}
(c) Finally, we consider $E_{\alpha,\varphi}\cap\mathcal{C}_0$.
Reasoning as in the case $\frac{4}{\varphi_{\max}}<\alpha<\frac{4}{\varphi_{\min}}$, we can restrict our analysis to competitors having a checkerboard structure union $q$ on the coordinate axes.
A comparison between the two energy contributions on each axis yields that the variation from the odd checkerboard to the even one is less then $0$, and equals $0$ if and only if $\alpha\in\Lambda$.
Thus, the minimal configuration is the even checkerboard (see Figure \ref{fig:step3}).
With \eqref{bound} and the finite superadditivity of the infimum, this implies that
\begin{equation}\label{subaddmin}
\begin{aligned}
\min_{E\in\mathcal{D},\, E\supset q}\mathcal{F}_\alpha^\varphi(E,q) &\ge \min_{E\in\mathcal{D},\, E\supset q}\Big(\sum_{Q(\jv)\in\mathcal{S}_e(E_{\alpha,\varphi})}\mathcal{F}_\alpha^\varphi(E\cap Q(\jv),q)+\mathcal{F}_\alpha^\varphi(E\cap\mathcal{C}_0,q)\Big)\\
&\ge \sum_{Q(\jv)\in\mathcal{S}_e(E_{\alpha,\varphi})}\min_{E\in\mathcal{D},\, E\supset q}\mathcal{F}_\alpha^\varphi(E\cap Q(\jv),q)+\min_{E\in\mathcal{D},\, E\supset q}\mathcal{F}_\alpha^\varphi(E\cap\mathcal{C}_0,q) \\
&= \sum_{Q(\jv)\in\mathcal{S}_e(E_{\alpha,\varphi})}\mathcal{F}_\alpha^\varphi(E(\Ze\cap B_\frac{4}{\alpha}^\varphi)\cap Q(\jv),q)+\mathcal{F}_\alpha^\varphi(E(\Ze\cap B_\frac{4}{\alpha}^\varphi)\cap\mathcal{C}_0,q) \\
&=\mathcal{F}_\alpha^\varphi(E(\Ze\cap B_\frac{4}{\alpha}^\varphi),q)\,,
\end{aligned}
\end{equation}
whence the equality follows, thus concluding the proof. Uniqueness comes from step (c).
%
\end{proof}

Note that the local minimum problems studied in points (a) and (b) in the proof above might be satisfied also by the odd checkerboard if \eqref{normestimate} reduces to an equality (\emph{e.g.} when $\varphi=\|\cdot\|_1$).
Nevertheless, for odd checkerboards the equality in \eqref{subaddmin} no longer holds and this implies that $E_{\alpha,\varphi}$ is the unique minimum.

\begin{definition}\label{nucleus-def}
For every $\alpha>0$, $\alpha\not\in\Lambda^\varphi$, we define the \emph{nucleus} of the motion given by the scheme \eqref{MM-scheme2} as the lattice set $$\mN_\alpha^\varphi:=Z(E_\alpha^1)$$ where 
 $E_\alpha^1 = \argmin{E\in\D,\, E\supset q}\mathcal{F}_\alpha^\varphi(E,q)$,
which is well defined by Proposition \ref{firststep}.
\end{definition}

We stress that the assumption on $\varphi$ to be an absolute norm is crucial in order to obtain the previous structure result of Proposition~\ref{firststep}. Indeed, if not fulfilled, the set $E_\alpha^1$ may not be a checkerboard as shown by the following simple example.
\begin{figure}[h]
\begin{center}
\includegraphics{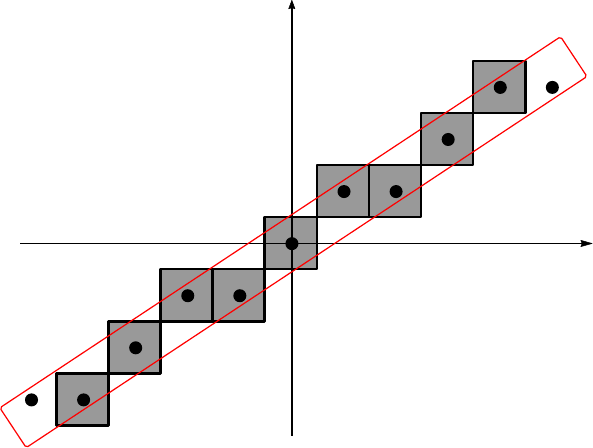}
\end{center}
\caption{The black dots represent the lattice set $\Z^2\cap B_{4\over\alpha}^\varphi$, while the set $E_\alpha^1$ is pictured in gray.}
\label{nocheck:fig}
\end{figure}

\begin{example}[non-checkerboard minimizers]\label{ex:noncheckmin}
We consider the norm
$$
\varphi(\x)=\max\Big\{\frac{|3x_1+2x_2|}{10},|3x_2-2x_1|\Big\}\,,
$$
and we assume that $\alpha\in (\frac{20}{13},\frac{40}{21})$.
In this case, for every such $\alpha$, the set $ B_{4\over\alpha}^\varphi$ is a rectangle and
\begin{equation*}
\I^{\varphi,\alpha}:= B_{4\over\alpha}^\varphi\cap\Z^2= \{(0,0),\pm(1,1),\pm(2,1),\pm(3,2),\pm(4,3),\pm(5,3)\}
\end{equation*}
(see Fig.~\ref{nocheck:fig}).
We show that the first step of \eqref{MM-scheme2} $E_\alpha^1$ is not a checkerboard set. 
First note that the points $(0,0)$ and $\pm(3,2)$ are isolated in $\I^{\varphi,\alpha}$ so their contribution is $-4+\alpha\varphi({\bf i})$ which is always negative, thus $Z(E_\alpha^1)$ contains these points.
Hence, we are reduced to study the minimal configurations of the pairs of nearest-neighbours $\{(1,1),(2,1)\}$ and $\{(4,3),(5,3)\}$:
$$
\mathcal{F}_\alpha^\varphi\big(q(1,1)\cup q(2,1),q\big)=-6+2\alpha<-4+\alpha=\mathcal{F}_\alpha^\varphi\big(q(1,1),q\big)=\mathcal{F}_\alpha^\varphi\big(q(2,1),q\big)
$$
and
\begin{align*}
\mathcal{F}_\alpha^\varphi\big(q(4,3),q\big) =-4+\frac{9}{5}\alpha &<-6+\alpha\Big(\frac{9}{5}+\frac{21}{10}\Big)=\mathcal{F}_\alpha^\varphi\big(q(4,3)\cup q(5,3),q\big) \\
&<-4+\alpha\frac{21}{10}=\mathcal{F}_\alpha^\varphi\big(q(5,3),q\big)\,.
\end{align*}
The same holds for $\{(-1,-1),(-2,-1)\}$ and $\{(-4,-3),(-5,-3)\}$, and this gives that
$$
E_\alpha^1=E(\I^{\varphi,\alpha}\backslash\{\pm(5,3)\})
$$
which is not a checkerboard (see Fig.~\ref{nocheck:fig}).
\end{example}


We conclude noting that, if we renounce to the monotonicity constraint $E\supset q$, the minimization problem above may admit, for suitable values of $\alpha$, also a checkerboard solution $E_\alpha^1$ of odd parity. In order not to distract the reader's attention from the monotone case, we prefer to postpone this generalization of Proposition~\ref{firststep} to Subsection~\ref{sec:conjectures} (see Proposition~\ref{firststep-gen}).

\subsection{The structure result for non-trivial initial datum}\label{sec:step2}

Proposition~\ref{firststep} shows that the first step $E_\alpha^1$ of discrete scheme \eqref{MM-scheme2} is a checkerboard set and that $Z(E^1_\alpha)$ is a $\Ze$-convex set (see Definition \ref{lattice-convex-def}).
Our aim now is to prove that an analogous structure result can be obtained for minimizers of the energy $\mathcal{F}_\alpha^\varphi(\cdot,E)$, where $\varphi$ is a symmetric absolute normalized norm (see Section~\ref{sec:norm}), also for a general $E\in\A_{\rm conv}$ fulfilling suitable assumptions (see \eqref{monotone-edges} below), and then to iteratively apply it to $E=E_\alpha^{k-1}$ for $k\geq1$. The proof of this stability result will rely on a localization argument only reminiscent of that used in the proof of Proposition~\ref{firststep}. Indeed, we have to face a technical issue: since the dissipation term $D^\varphi(\cdot,E)$ does not satisfy a submodularity inequality analogous to \eqref{normestimate}, the $2\times2$-square covering no longer works.
We will then define suitable coverings ``outside'' every discrete edge (see Definition \ref{disc-edge-def}) of $E$ which mimick the $2\times2$-square covering, and then match them altogether.
For this, we need the following ``convexity'' conditions:

\smallskip
(i) on the norm, we assume that
\begin{enumerate}[font={\normalfont},label={(H3)}]
\item $\varphi(h,h+1)-\varphi(h,h)\ge\frac{1}{2}, \quad \text{for every } h\in\N$;\label{norm-derivative}
\end{enumerate}

(ii) on the structure of $\partial^{\rm eff}E$, we require that
\begin{equation}\label{monotone-edges}
\theta(\bm\nu(\ell'),\bm\nu(\ell))<0 \quad \text{for every } \ell,\ell'\in\mathcal{E}(E) \text{ such that } \ell \text{ precedes (clockwise) } \ell'\,,
\end{equation}
where $\theta$ is introduced in Definition \ref{angle}.

The $\ell^p$-norms, $1\leq p\leq\infty$, are a class of norms complying with {\ref{ass-H1}}--{\ref{norm-derivative}}.
We also note that assumption {\ref{norm-derivative}} will play a role only in Step 5 of the proof of Proposition \ref{steps}.

In order to avoid some (interesting) pathological phenomena (as a one-dimensional motion, see Example \ref{1D-ex}), we assume non-degeneracy conditions on the sets $E$ and on the minimizer of $\mathcal{F}_\alpha^\varphi(\cdot,E)$; namely, {\ref{ass-H2}} and \eqref{non-degeneracy}.
Finally, to simplify the exposition, we assume that
\begin{equation}
\mbox{$E$ is symmetric with respect to the axes and the lines $x_2=\pm x_1$.}
\label{E-symm}
\end{equation}

We now state the main result of this section.

\begin{proposition}\label{steps}
Let $\varphi$ be a symmetric absolute normalized norm complying with {\rm\ref{norm-derivative}} 
and let $\alpha>0$ be such that $\alpha\not\in\Lambda^\varphi$.
Let $E\in\A^e_{\rm conv}$ be a set 
satisfying \eqref{non-degeneracy}, \eqref{monotone-edges} and \eqref{E-symm}.
Then there exists a unique solution of the minimization problem
\begin{equation}\label{min-prob-steps}
E_\alpha=\argmin{E'\supset E\, E'\in\D}\mathcal{F}_\alpha^\varphi(E',E)
\end{equation}
and it satisfies
\begin{equation}\label{even-steps}
Z(E_\alpha)=\Big\{\iv\in\Ze \,:\, d^\varphi(\iv,E)<\frac{4}{\alpha}\Big\}.
\end{equation}
In particular, $E_\alpha\in\A^e_{\rm conv}$.
\end{proposition}

Before entering in the details of the proof, we premise some remarks.

\begin{figure}[htp]
\centering
\includegraphics{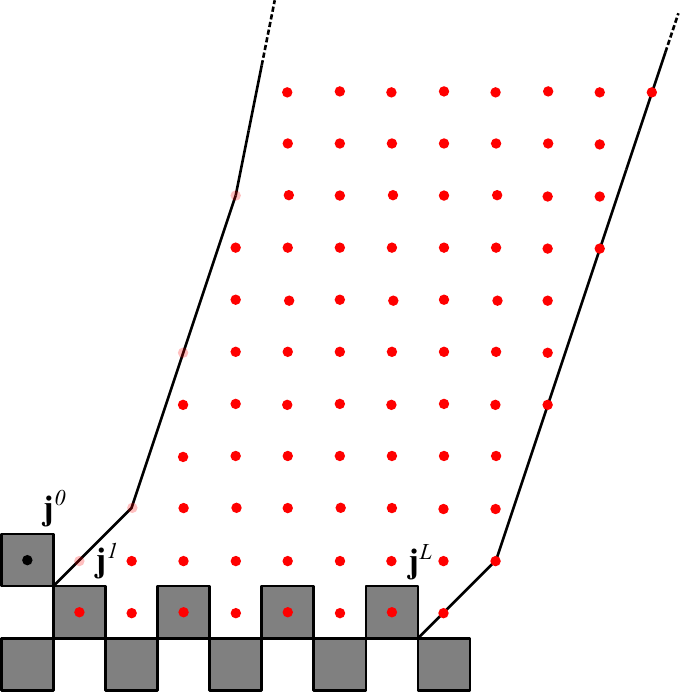}
\quad
\includegraphics{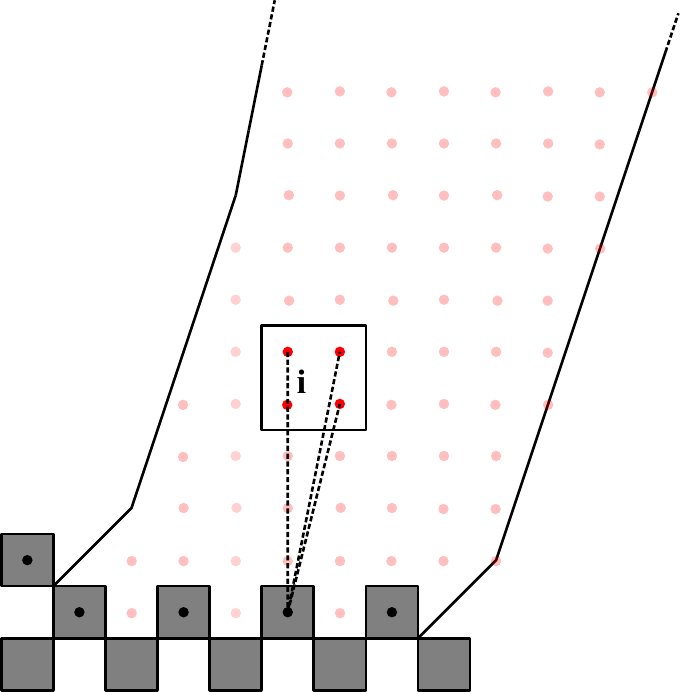}
\caption{In red an example of $A(\ell)$ for $\ell$ as in (ii) of Remark \ref{flat-slant-edge}.
On the left, lighter dots are outside $A(\ell)$.
On the right, the projection of the centers of a $2\times2$-square on a common point of $\ell$.}
\label{fig:area-proj}
\end{figure}

\begin{remark}[Projection of a $2\times2$ square]\label{BS-proj}
Let $E$ be given as in the statement of Proposition \ref{steps}.
We partition the lattice points of the region of the plane ``outside'' $E$ into sets $A(\ell)$ according to the discrete edge $\ell\in\mathcal{E}(E)$ they project onto.
We follow the classification of discrete edges given in Remark \ref{flat-slant-edge}, and we start with case (ii); that is, $\ell\subset\{\x\in\R^2 \,:\, x_2>0\}$ and $s(\ell)\in(0,\frac{1}{3}]$.
For such edges we define the set
\begin{equation}\label{ext-edge-flat}
A(\ell) := \big\{\iv\in\Z^2\,:\, i_1\ge j_1^1,\, i_2\ge j_2^1,\,\pi_E^\varphi(\iv)\subset\{\jv^l\}_{l=1}^L \text{ or } \pi_E^\varphi(\iv)\ni\jv^L\big\}\,,
\end{equation}
consisting of all the lattice points that project on $\ell\setminus\{\jv^0\}$ (Fig. \ref{fig:area-proj}).
The choice of excluding the points projecting also on $\jv^0$, although arbitrary, will simplify the definition of the covering in the proof of Proposition~\ref{steps}; moreover, thanks to this choice, if $\ell$ and $\ell'$ are two consecutive edges then $A(\ell)$ and $A(\ell')$ are disjoint.

We can assume, up to translations and for the sake of simplicity, that $\ell:=\{\jv^l\}_{l=0}^L=\{(1,1)\}\cup\{(2l,0)\}_{l=1}^L$.
From the fact that $\varphi$ is monotonic, for every $\iv\in A(\ell)$ it holds that
$$
\pi_E^\varphi(\iv)\ni
\begin{cases}
\jv^l & 2l-1\le i_1\le 2l+1 \text{ with } 0<l<L\\
\jv^L & i_1\ge 2L-1.
\end{cases}
$$
This yields that for every $\iv\in A(\ell)$ such that $Z\big(Q(\iv)\big)\subset A(\ell)$ there holds
\begin{equation}\label{eq:intersection-flat}
\Big(\bigcap_{\jv\in Z(Q(\iv))}\pi_E^\varphi(\jv)\Big)\cap\big(\ell\setminus\{\jv^0\}\big)\neq\emptyset\,.
\end{equation}
This means that the four lattice points inside $Q(\iv)$ project onto a common point of $\ell$, see Fig.~\ref{fig:area-proj}.
An analogous result holds in case (i) of Remark \ref{flat-slant-edge}, when $s(\ell)=0$.

\begin{figure}[htp]
\centering
\includegraphics{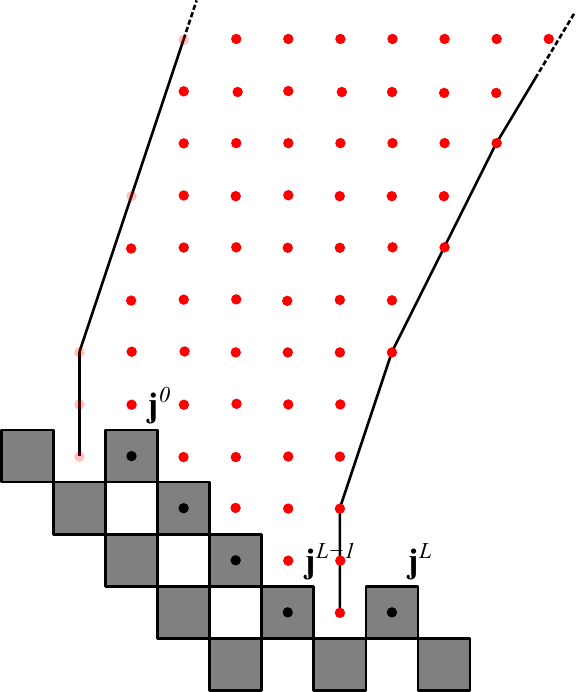}
\quad
\includegraphics{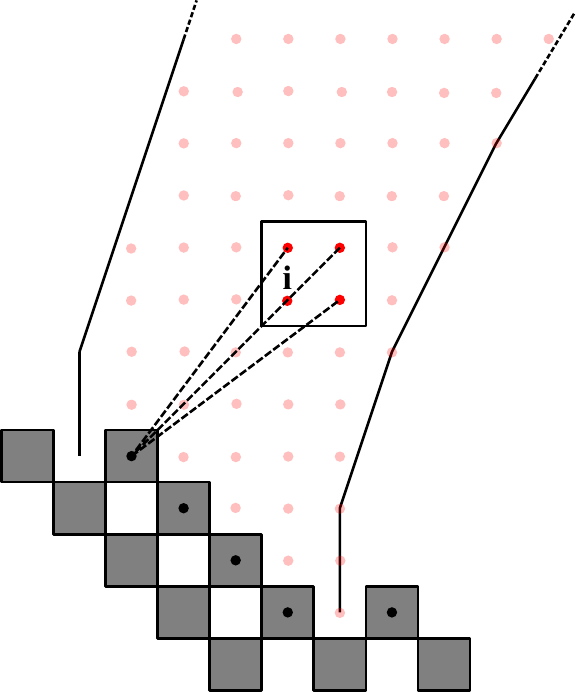}
\caption{In red an example of $A(\ell)$ for $\ell$ as in (iii) of Remark \ref{flat-slant-edge}.
On the left, lighter dots are outside $A(\ell)$.
On the right, the projection of the centers of a $2\times2$-square on a common point of $\ell$.}
\label{fig:area-proj2}
\end{figure}
Now consider $\ell\in\mathcal{E}(E)$ complying with case (iii) of Remark \ref{flat-slant-edge}; that is, $\ell\subset\{\x\in\R^2 \,:\, x_2>0\}$ and $s(\ell)\in(\frac{1}{3},1)$.
In this case the sets of lattice points that project on $\ell\setminus\{\jv^L\}$ is defined as
\begin{equation}\label{ext-edge-slant}
A(\ell) := \big\{\iv\in\Z^2\,:\, \|\iv-(j_1^0,j_2^{L-1})\|_1\ge|j_2^0-j_2^{L-1}|,\, i_2\ge j_2^{L-1},\, \pi_E^\varphi(\iv)\subset\{\jv^l\}_{l=0}^{L-1} \text{ or } \pi_E^\varphi(\iv)\ni\jv^{L-1}\big\}\,,
\end{equation}
see Fig.~\ref{fig:area-proj2}.
For simplicity we can assume, up to translations, that $\ell=\{\jv^l\}_{l=0}^L=\{(l,-l)\}_{l=0}^{L-1}\cup\{(L+1,-L+1)\}$.
From the symmetry assumption {\ref{ass-H1}} there holds
$$
\pi_E^\varphi(\iv)\ni
\begin{cases}
\jv^0 & i_1-i_2\le1 \\
\jv^l & 2l-1\le i_1-i_2\le 2l+1 \text{ with } 0<l<L.
\end{cases}
$$

This can be seen by characterizing the projection of points $\iv\in A(\ell)$ of coordinates $\iv=(h,h)$ and $(h+1,h)$ with $h\in\N$, since the other cases reduce to this situation from the translation invariance of the distance.
Thus, assume by contradiction that there exist $h$ and $0<l<L$ such that $\varphi(\iv-\jv^l)=\varphi(h-l,h+l)<\varphi(h,h)=\varphi(\iv-\jv^0)$.
We reduce to $l\le h$ from the fact that $\varphi$ is monotonic.
Then, by {\ref{ass-H1}} and convexity we get
$$
\varphi(h,h)\le\frac{1}{2}\varphi(h+l,h-l)+\frac{1}{2}\varphi(h-l,h+l)=\varphi(h-l,h+l),
$$
leading to a contradiction.
As for the case $\iv=(h+1,h)$, assuming that $\varphi(\iv-\jv^l)<\varphi(\iv-\jv^0)$ again by {\ref{ass-H1}} and convexity we get
$$
\varphi(h+1,h)\le\frac{h+1}{2h+1}\varphi(h+1-l,h+l)+\frac{h}{2h+1}\varphi(h+l,h+1-l)=\varphi(h+1-l,h+l)
$$
and we obtain a contradiction.
Hence, for every $\iv\in A(\ell)$ such that $Z\big(Q(\iv)\big)\subset A(\ell)$ there holds
\begin{equation}\label{eq:intersection-slant}
\Big(\bigcap_{\jv\in Z(Q(\iv))}\pi_E^\varphi(\jv)\Big)\cap\big(\ell\setminus\{\jv^L\}\big)\neq\emptyset\,;
\end{equation}
again, as for \eqref{eq:intersection-flat}, \eqref{eq:intersection-slant} means that the lattice points inside $Q(\iv)$ project onto a common point of $\ell$, see Fig.~\ref{fig:area-proj2}.
An analog of \eqref{eq:intersection-slant} holds in the case (iv) of Remark \ref{flat-slant-edge}.
\end{remark}

\begin{remark}\label{rectremk}
In order to compare the energies of checkerboard configurations with different parities inside certain rectangular tiles, it will be useful to establish some inequalities involving the dissipation term.

\begin{figure}[h]
\centering
\def\svgwidth{250pt}
\begingroup%
  \makeatletter%
  \providecommand\color[2][]{%
    \errmessage{(Inkscape) Color is used for the text in Inkscape, but the package 'color.sty' is not loaded}%
    \renewcommand\color[2][]{}%
  }%
  \providecommand\transparent[1]{%
    \errmessage{(Inkscape) Transparency is used (non-zero) for the text in Inkscape, but the package 'transparent.sty' is not loaded}%
    \renewcommand\transparent[1]{}%
  }%
  \providecommand\rotatebox[2]{#2}%
  \newcommand*\fsize{\dimexpr\f@size pt\relax}%
  \newcommand*\lineheight[1]{\fontsize{\fsize}{#1\fsize}\selectfont}%
  \ifx\svgwidth\undefined%
    \setlength{\unitlength}{314.75866843bp}%
    \ifx\svgscale\undefined%
      \relax%
    \else%
      \setlength{\unitlength}{\unitlength * \real{\svgscale}}%
    \fi%
  \else%
    \setlength{\unitlength}{\svgwidth}%
  \fi%
  \global\let\svgwidth\undefined%
  \global\let\svgscale\undefined%
  \makeatother%
  \begin{picture}(1,0.28944415)%
    \lineheight{1}%
    \setlength\tabcolsep{0pt}%
    \put(0,0){\includegraphics[width=\unitlength,page=1]{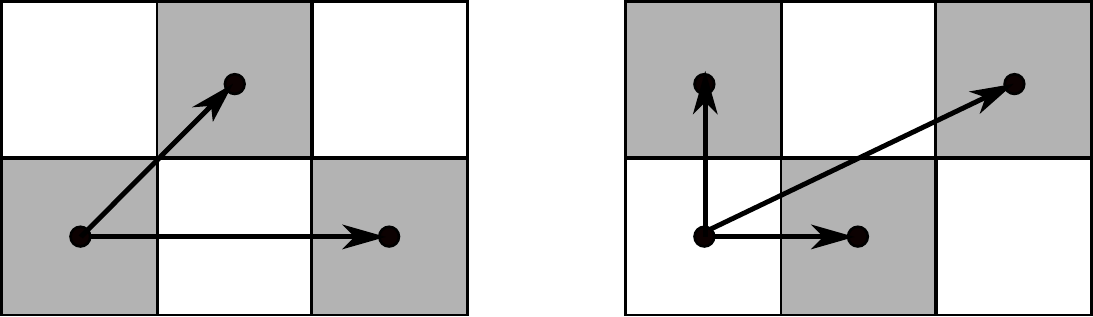}}%
    \put(0.05815638,0.02884907){\color[rgb]{0,0,0}\makebox(0,0)[lt]{\lineheight{1.25}\smash{\begin{tabular}[t]{l}$\iv$\end{tabular}}}}%
    \put(0.18002303,0.23947135){\color[rgb]{0,0,0}\makebox(0,0)[lt]{\lineheight{1.25}\smash{\begin{tabular}[t]{l}$\iv+(1,1)$\end{tabular}}}}%
    \put(0.30791751,0.02884907){\color[rgb]{0,0,0}\makebox(0,0)[lt]{\lineheight{1.25}\smash{\begin{tabular}[t]{l}$\iv+(2,0)$\end{tabular}}}}%
    \put(0.61002302,0.23947135){\color[rgb]{0,0,0}\makebox(0,0)[lt]{\lineheight{1.25}\smash{\begin{tabular}[t]{l}$\iv+(0,1)$\end{tabular}}}}%
    \put(0.7568175,0.02884907){\color[rgb]{0,0,0}\makebox(0,0)[lt]{\lineheight{1.25}\smash{\begin{tabular}[t]{l}$\iv+(1,0)$\end{tabular}}}}%
    \put(0.87702256,0.23947135){\color[rgb]{0,0,0}\makebox(0,0)[lt]{\lineheight{1.25}\smash{\begin{tabular}[t]{l}$\iv+(2,1)$\end{tabular}}}}%
  \end{picture}%
\endgroup%

\caption{The triples of points involved in \eqref{rect}.}\label{fig:remark6}
\end{figure}
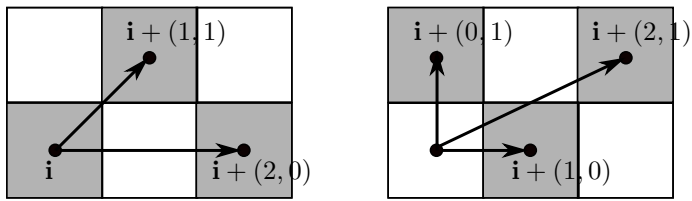
Consider $E$ as in the statement of Proposition \ref{steps} and $\ell\in\mathcal{E}(E)$ such that $\ell\subset\{\x\in\R^2\,:\, x_2>0\}$ and $s(\ell)\in[0,1]$.
For the sake of simplicity we can assume (up to a translation) that $\jv^L=(0,0)$ where $\ell=\{\jv^l\}_{l=0}^L$.
If $s(\ell)\in[0,\frac{1}{3}]$,
for every $\iv\in A(\ell)$ with $i_1\in2\Z$ such that \eqref{eq:intersection-flat} holds, from \eqref{normestimate} and the properties of $\varphi$ one can infer (see Fig.~\ref{fig:remark6}) the inequality
\begin{multline}\label{rect}
d^\varphi(\iv,E)+d^\varphi(\iv+(1,1),E)+d^\varphi(\iv+(2,0),E) \\
\le d^\varphi(\iv+(0,1),E)+d^\varphi(\iv+(1,0),E)+d^\varphi(\iv+(2,1),E).
\end{multline}
The same inequality holds if $s(\ell)\in(\frac{1}{3},1]$,
for every $\iv\in A(\ell)$ with $\iv\in\Ze$ such that \eqref{eq:intersection-slant} holds.

Indeed, \eqref{eq:intersection-flat} and \eqref{eq:intersection-slant} ensure the existence of some $\jv'\in \ell$ such that $d^\varphi(\jv,E)=\varphi(\jv-\jv')$ for every $\jv\in Z\big(Q(\iv)\big)$.
Hence \eqref{normestimate} reads
\begin{equation}
d^\varphi(\iv,E)+d^\varphi(\iv+(1,1),E) \le d^\varphi(\iv+(0,1),E)+d^\varphi(\iv+(1,0),E)\,.
\label{rect_1}
\end{equation}
Now, from the fact that $i_2\ge j_2'$ (see Remark \ref{BS-proj}) and the monotonicity of the norm $\varphi$, we have $\varphi(\iv+(2,0)-\jv')\le\varphi(\iv+(2,1)-\jv')$, whence we get
\begin{equation} 
d^\varphi(\iv+(2,0),E)\le d^\varphi(\iv+(2,1),E)\,.
\label{rect_2}
\end{equation}
Inequality \eqref{rect} then follows by adding term by term \eqref{rect_1} and \eqref{rect_2}.
\end{remark}

\begin{remark}\label{remk:monotone-edges}
As a last preparatory remark to the proof of Proposition~\ref{steps}, we analyze and motivate assumption \eqref{monotone-edges} on the sets that intervene in minimization problem \eqref{min-prob-steps}.
Assumption \eqref{monotone-edges} ensures that for every discrete edge there are infinitely many $2\times2$-squares whose centers project onto it.
This property is crucial to define a well-posed covering argument (see Section \ref{sec:proofsteps}).
Specifically, let $E$ be as in the statement of Proposition \ref{steps} and $\ell\in\mathcal{E}(E)$ be such that $\ell\subset\{\x\in\R^2 \,:\, x_2>0\}$ and
\begin{equation}\label{slope-cond}
s(\ell)\in[0,1]\,.
\end{equation}
We claim that, for every such $\ell$ the following property holds:
\begin{equation}\label{tech-ass}
\text{for every } h\in\N \text{ there exists } \iv\in A(\ell)\cap\big(\Z\times\{h\}\big)\cap\Ze \,:\, Z(Q(\iv))\subset A(\ell),
\end{equation}
where we have set $\ell=\{\jv^l\}_{l=0}^L$ and $\jv^L=(0,0)$ for simplicity.
This claim is proved inductively (on parameter labeling clockwise consecutive discrete edges) by showing that for any triple of consecutive edges of $E$, say $\ell^-,\ell,\ell^+$ with $\ell^-$ satisfying \eqref{tech-ass}, we can find a point $\iv\in(\Z\times\{h\}\big)\cap\Ze$ for which, thanks to \eqref{monotone-edges} and the translation invariance of the distance, there holds $d^\varphi(\jv,\ell)\le d^\varphi(\jv,\ell^-\cup\ell^+)$ for every $\jv\in Z\big(Q(\iv)\big)$ and every $h\ge0$.

\begin{figure}[htp]
\centering
\includegraphics{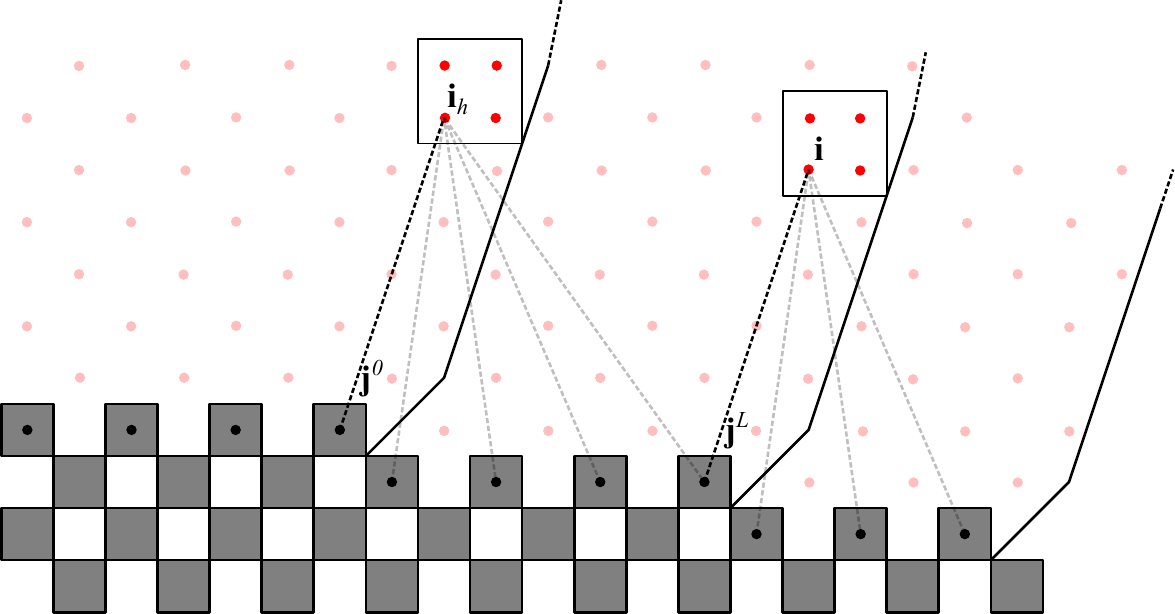}
\caption{An example of the situation described in Remark \ref{remk:monotone-edges} in the case $s(\ell)\le\frac{1}{3}$.
The lighter dots represent the points of $\Ze$ inside $A(\ell^-)$, $A(\ell)$ and $A(\ell^+)$.}
\label{fig:remk-tech-ass}
\end{figure}
Let $\ell_0=\{\jv_0^l\}_{l=0}^{L_0}$ be the first (clockwise-ordered) edge such that $s(\ell)\geq0$ 
and set
$$
\ell_0':=\begin{cases}
\ell_0 & \text{if } \bm\nu(\ell_0)=(0,1) \\
\{\jv_0^0\} & \text{otherwise.}
\end{cases}
$$
It is straightforward that \eqref{tech-ass} is satisfied for $\ell=\ell_0'$ (reasoning as in Remark \ref{BS-proj}) where we have set
\begin{equation}\label{ext-edge-top0}
A(\jv^0_0)=\big\{\iv\in \Z^2 \,:\, \pi_E^\varphi(\iv)\ni\jv_0^0\big\}.
\end{equation}
Consider $\ell,\ell^-,\ell^+\in\mathcal{E}(E)$ satisfying \eqref{slope-cond}, with $\ell^-$ preceding $\ell$, $\ell$ preceding $\ell^+$.
Write $\ell^-=\{\jv^{-,l}\}_{l=0}^{L^-}$ and $\ell^+=\{\jv^{+,l}\}_{l=0}^{L^+}$.
We point out that if $\ell^-$ coincide with $\ell_0=\{\jv_0^0\}$ then $\ell=\ell_0$.

Assume that $\ell^-$ satisfies \eqref{tech-ass}.
Consider first the case $s(\ell)\leq \frac{1}{3}$ (see Fig.~\ref{fig:remk-tech-ass}).
For any fixed $h\in\N$, we set
$\iv_h=\argmax{}\{i_1 \,:\, \iv\in\Ze,\, Z\big(Q(\iv)\big)\subset A(\ell^-),\, i_2=h+1\}$,
which is well defined since we have assumed that $\ell^-$ satisfies \eqref{tech-ass}.
By Remark \ref{BS-proj} and by definition of $A(\ell^-)$ \eqref{ext-edge-flat}, there holds
\begin{equation}\label{remk-tech-ass-i}
d^\varphi(\jv,E) = \varphi(\jv-\jv_0) < d^\varphi(\jv,\ell) \quad \text{for every } \jv\in Z\big(Q(\iv_h)\big).
\end{equation}
Set $\iv:=\iv_h+(2L-1,-1)$ and note that $\iv\in\Ze$ and $i_2=h$.
Note also that the definition of $\iv_h$ yields $Z\big(Q(\iv)\big)\cap A(\ell^-)=\emptyset$.
Then, by \eqref{remk-tech-ass-i} and the translation invariance of the distance, since $\jv_0+(2L-1,-1)=\jv_L$ we get
$$
d^\varphi(\jv,\ell) = \varphi(\jv-\jv_L) < d^\varphi(\jv,\ell+(2L-1,-1)), \quad \text{for every } \jv\in Z\big(Q(\iv)\big).
$$
Now, \eqref{monotone-edges} yields $d^\varphi(\jv,\ell+(2L-1,-1))\le d^\varphi(\jv,\ell^+)$.
Indeed, if $s(\ell^+)\leq\frac{1}{3}$
then $L^+\le L$ by \eqref{monotone-edges}, thus $\ell^+\subset\ell+(2L-1,-1)$.
If instead $s(\ell^+)>\frac{1}{3}$
then by the monotonicity of $\varphi$ we get $d^\varphi(\jv,\ell^+)\ge \varphi(\jv-\jv^L)$.
\begin{figure}[htp]
\centering
\includegraphics{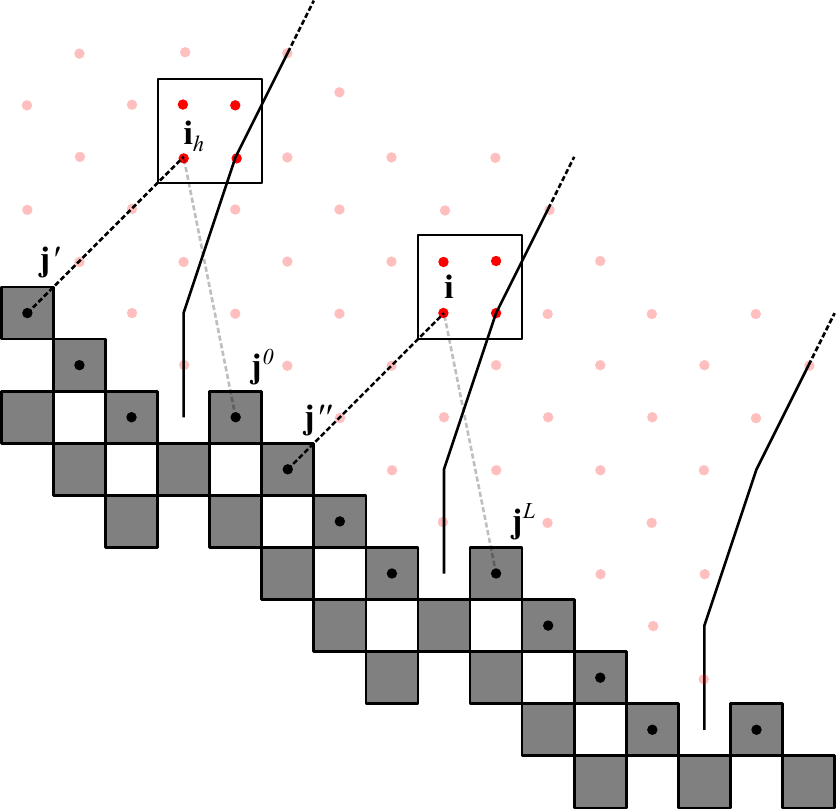}
\caption{An example of the situation described in Remark \ref{remk:monotone-edges} in the case $s(\ell)<\frac{1}{3}$. The lighter dots represent those points of $\Ze$ lying inside $A(\ell^-)$, $A(\ell)$ and $A(\ell^+)$.}
\label{fig:remk-tech-ass2}
\end{figure}

Now consider the case $s(\ell)>\frac{1}{3}$ (Fig.~\ref{fig:remk-tech-ass2}).
For any fixed $h\in\N$, we set
$$
\iv_h=\argmax{}\{i_1 \,:\, \iv\in\Ze,\, Z\big(Q(\iv)\big)\subset A(\ell^-),\, i_2=h+L-1\},
$$
which is well defined as above.
By Remark \ref{BS-proj} and by definition of $A(\ell^-)$ \eqref{ext-edge-slant} there holds
\begin{equation}\label{remk-tech-ass-ii}
d^\varphi(\jv,E) = \varphi(\jv-\jv') \le \varphi(\jv-\jv^0) = d^\varphi(\jv,\ell) \quad \text{for every } \jv\in Z\big(Q(\iv_h)\big),
\end{equation}
for some $\jv'\in\{j^{-,l}\}_{l=0}^{L^--1}$.
Again, set $\iv:=\iv_h+(L-1,-L+1)$ and note that $\iv\in\Ze$ and $i_2=h$.
Reasoning as above we have $Z\big(Q(\iv)\big)\cap A(\ell^-)=\emptyset$.
By \eqref{remk-tech-ass-ii}, the translation invariance of the distance and since $\jv'':=\jv'+(L^-+1,-L^-+1)\in\{\jv_l\}_{l=0}^{L^--1}\subset\ell$ we get
$$
d^\varphi(\jv,\ell) = \varphi(\jv-\jv'') \le \varphi\big(\jv-\jv^L), \quad \text{for every } \jv\in Z\big(Q(\iv)\big).
$$
From \eqref{monotone-edges} we have $s(\ell)>\frac{1}{3}$
and $L^+\ge L$, thus $\varphi\big(\jv-\jv^L)=d^\varphi(\jv,\ell^+)$, arguing as in Remark \ref{BS-proj}.
\end{remark}

\subsection{Proof of Proposition~\ref{steps}} \label{sec:proofsteps}

We are now ready to prove the main result on the structure of the minimizer of $\mathcal{F}_\alpha^\varphi(\cdot,E)$.
For the covering argument that we will introduce, the $2\times2$-squares are not sufficient.
Therefore, we define a new class of tiles for the covering.

\begin{figure}[htbp]
\centering
\def\svgwidth{300pt}
\begingroup%
  \makeatletter%
  \providecommand\color[2][]{%
    \errmessage{(Inkscape) Color is used for the text in Inkscape, but the package 'color.sty' is not loaded}%
    \renewcommand\color[2][]{}%
  }%
  \providecommand\transparent[1]{%
    \errmessage{(Inkscape) Transparency is used (non-zero) for the text in Inkscape, but the package 'transparent.sty' is not loaded}%
    \renewcommand\transparent[1]{}%
  }%
  \providecommand\rotatebox[2]{#2}%
  \newcommand*\fsize{\dimexpr\f@size pt\relax}%
  \newcommand*\lineheight[1]{\fontsize{\fsize}{#1\fsize}\selectfont}%
  \ifx\svgwidth\undefined%
    \setlength{\unitlength}{681.06583723bp}%
    \ifx\svgscale\undefined%
      \relax%
    \else%
      \setlength{\unitlength}{\unitlength * \real{\svgscale}}%
    \fi%
  \else%
    \setlength{\unitlength}{\svgwidth}%
  \fi%
  \global\let\svgwidth\undefined%
  \global\let\svgscale\undefined%
  \makeatother%
  \begin{picture}(1,0.22750753)%
    \lineheight{1}%
    \setlength\tabcolsep{0pt}%
    \put(0,0){\includegraphics[width=\unitlength,page=1]{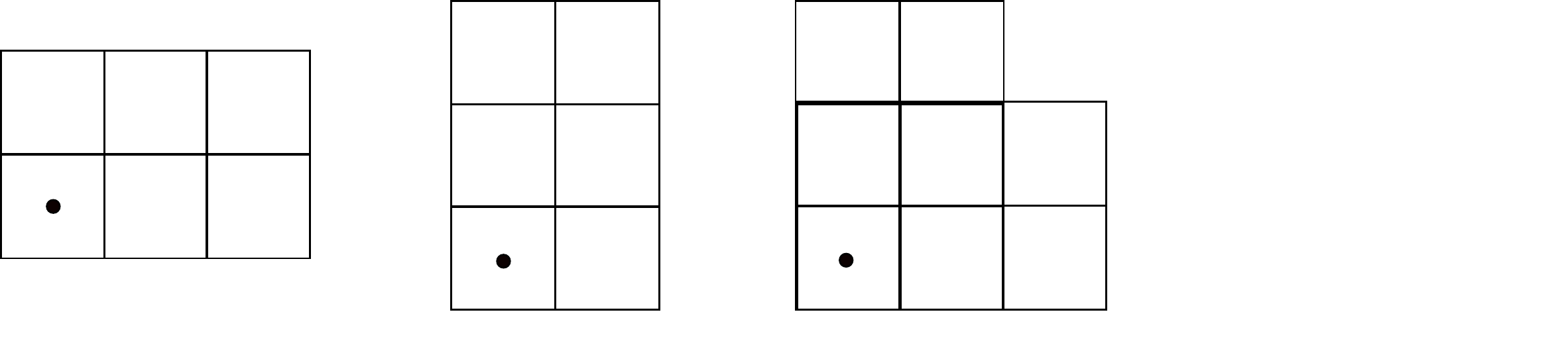}}%
    \put(0.0487898,0.02652753){\color[rgb]{0,0,0}\makebox(0,0)[lt]{\lineheight{1.25}\smash{\begin{tabular}[t]{l}$R^{\rm hor}({\bf i})$\end{tabular}}}}%
    \put(0.30640705,-0.00215512){\color[rgb]{0,0,0}\makebox(0,0)[lt]{\lineheight{1.25}\smash{\begin{tabular}[t]{l}$R^{\rm ver}({\bf i})$\end{tabular}}}}%
    \put(0.60484754,-0.00204328){\color[rgb]{0,0,0}\makebox(0,0)[lt]{\lineheight{1.25}\smash{\begin{tabular}[t]{l}$R^+({\bf i})$\end{tabular}}}}%
    \put(0,0){\includegraphics[width=\unitlength,page=2]{rettangoli.pdf}}%
    \put(0.90077019,-0.00204328){\color[rgb]{0,0,0}\makebox(0,0)[lt]{\lineheight{1.25}\smash{\begin{tabular}[t]{l}$R^-({\bf i})$\end{tabular}}}}%
  \end{picture}%
\endgroup%

\caption{The different tiles of the covering.}\label{fig:rectangles}
\end{figure}
\begin{definition}\label{def:newtiles}
For every $\iv\in\Z^2$ we set
\begin{eqnarray*}
&R^{\rm hor}(\iv) := Q(\iv)\cup q(\iv+(2,0))\cup q(\iv+(2,1))\,,
&R^+(\iv) := R^{\rm ver}(\iv)\cup R^{\rm hor}(\iv)\,,\\
&R^{\rm ver}(\iv) := Q(\iv)\cup q(\iv+(0,2))\cup q(\iv+(1,2))\,,
&R^-(\iv) := R^{\rm ver}(\iv+(1,-1))\cup R^{\rm hor}(\iv)\,,
\end{eqnarray*}
where $Q(\iv)$ is defined as in \eqref{quadratone} (see Fig.~\ref{fig:rectangles}).
\end{definition}

For every discrete edge $\ell\in\mathcal{E}(E)$, we will define a covering of the region outside $E$ projecting onto $\ell$. We warn the reader that the choice of the tiles will depend both on the slope $s(\ell)$ and the neighboring edges. Heuristically, where the discrete projection $\pi_E^\varphi$ behaves as in the case of the distance from a point, we will still use the tiles $Q(\iv)$, as in the proof of Proposition \ref{firststep}.
In order to match the coverings of the regions projecting onto adjacent edges, we will need tiles $R^{\rm hor}(\iv)$ and $R^{\rm ver}(\iv)$ (see Steps 2 and 3 of the proof), in which the even checkerboard is the minimizer by virtue of Remark \ref{rectremk}.
Moreover, we will take into account that the effective boundary $\partial^{\rm eff}E$ may present some irregularities due to the discrete nature of the problem (see Steps 4 and 5).
In that case, where needed, we will use the ``siding tiles'' $R^+(\iv)$ and $R^-(\iv)$ which are compatible with the rest of the covering and still favor the even configurations in the local minimum problems therein.

\begin{proof}[Proof of Proposition {\rm\ref{steps}}]
According to the discussion in Remark \ref{flat-slant-edge}, we reduce the description of the covering corresponding to the discrete edges of $E$ contained in $\{\x\in\R^2\,:\,x_2\ge0\}$ complying with
\begin{equation}\label{edge-simpl}
0\leq s(\ell)\leq 1\,,
\end{equation}
as the covering for the remaining edges can be obtained symmetrically.
We divide the proof into several steps.

{\bf Step 1: ordering of the discrete edges.}
We label in clockwise order the set of discrete edges of $E$; namely, $\{\ell_m\}_{m=1}^{m_1}\subset\mathcal{E}(E)$.
For our convenience, writing $\ell_1=\{\jv^l\}_{l=0}^{L}$, with a slight abuse of notation, in the case that $s(\ell_1)=0$ we write (without relabelling) $\ell_1=\{\jv_1^l\}_{l=\lfloor\frac{L}{2}\rfloor}^L$.
If $s(\ell_1)>0$ we set $\ell_0:=\{\jv^0_1\}$.
Now we set
$$
m_0 := \max\Big\{0\le m\le M \,:\, s(\ell_m)\le\frac{1}{3}\Big\}.
$$

\begin{figure}[htbp]
\centering
\includegraphics[width=0.4\textwidth]{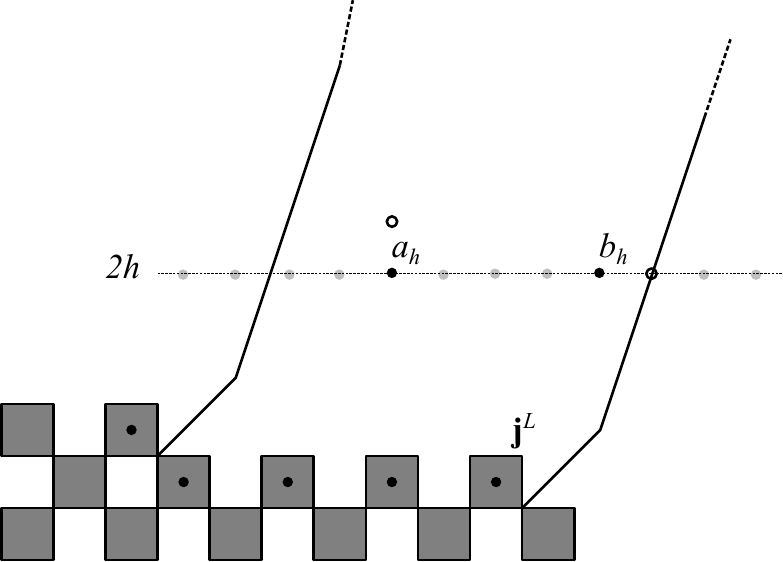}
\caption{The point $(a_h,2h+1)\in A(\ell)$ (black circle on the left).
The point $(b_h+1,2h)\in A(\ell_{m+1})$ (black circle on the right).
The gray points represent $\Z\times\{2h\}$.}
\label{fig:a-b-def}
\end{figure}

{\bf Step 2: covering of the region outside ${E}$ projecting onto $\ell_m$ with $0< m<m_0$.}
We set $\ell_m:=\{\jv^l\}_{l=0}^{L}$ assuming, without loss of generality, that $\jv^L=(0,0)$.
We also define
\begin{equation}\label{extrema}
\begin{aligned}
a_h &:= \min\{h'\in2\Z\,:\,(h',2h+1)\in A(\ell_m)\}\,, \\
b_h &:= \max\{h'\in2\Z\,:\,(h'+1,2h)\in A(\ell_m)\}\,,
\end{aligned}
\end{equation}
for every $h\in\N$, where $A(\ell)$ is defined in \eqref{ext-edge-flat} (see Fig.~\ref{fig:a-b-def}).
In the case $m=1$ and $s(\ell_1)=0$, the set $A(\ell_1)$ is still as in \eqref{ext-edge-flat} with $\{\jv^l\}_{l=\lfloor\frac{L}{2}\rfloor}^L$ in place of $\{\jv^l\}_{l=1}^L$.
Note that, by Remark \ref{remk:monotone-edges}, assumption \eqref{monotone-edges} yields that $a_h$ and $b_h$ are well-defined for every $h\in\N$.
We then introduce the set
\begin{equation}\label{indices}
\I(\ell_m):=\bigcup_{h\ge0}\{(h',2h)\,:\,h'\in2\Z,\,a_h\le h'\le b_h\}\,.
\end{equation}

\begin{figure}[htbp]
\centering
\includegraphics[width=0.4\textwidth]{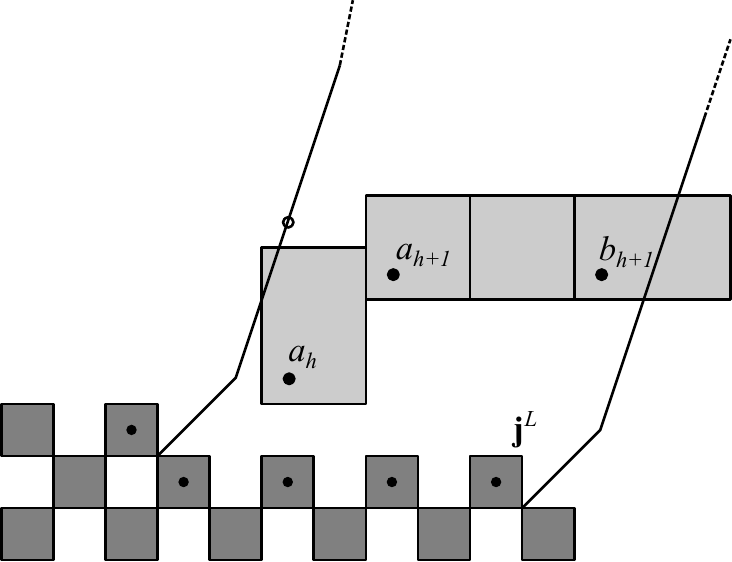}
\quad
\includegraphics[width=0.3\textwidth]{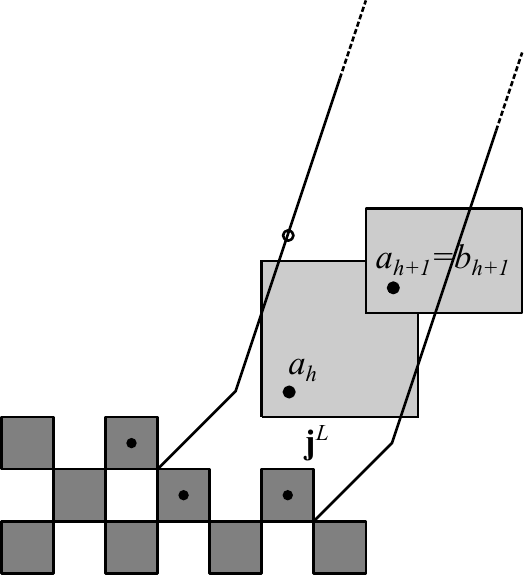}
\caption{Examples of coverings defined in \eqref{covering-0}. The black circle represents the point $(a_{h-1},2h+3)$.}
\label{fig:ricoprimenti}
\end{figure}
Correspondingly, for every $h\in\N$ we define the following covering (see Fig.~\ref{fig:ricoprimenti}):
\begin{equation}\label{covering-0}
\begin{aligned}
C(\iv) &:=\begin{cases}
R^{\rm ver}(\iv) &\iv=(a_h,2h) \text{ and } (a_h,2h+3)\not\in A(\ell_m)\\
Q(\iv) &\iv=(a_h,2h) \text{ and } (a_h,2h+3)\in A(\ell_m) \\
Q(\iv) &i_1\in2\Z,\, a_h<i_1<b_h \\
R^{\rm hor}(\iv) &\iv=(b_h,2h)
\end{cases}, \quad \text{if } a_h<b_h\,, \\
C(\iv) &:= \begin{cases}
R^+(\iv), &\iv=(a_h,2h) \mbox{ and } (a_h,2h+3)\not\in A(\ell_m) \\
R^{\rm hor}(\iv), &\iv=(a_h,2h) \mbox{ and } (a_h,2h+3)\in A(\ell_m)
\end{cases}, \quad \text{if } a_h=b_h\,,
\end{aligned}
\end{equation}
\begin{figure}[htbp]
\centering
\includegraphics[width=0.4\textwidth]{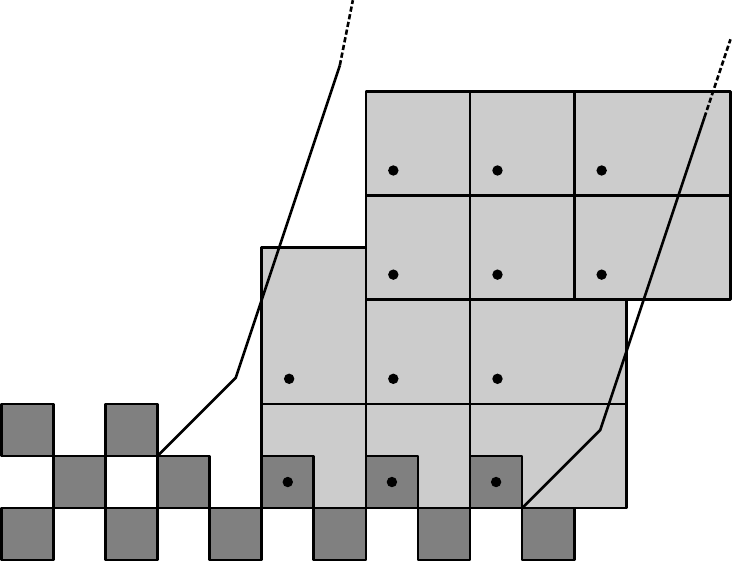}
\quad
\includegraphics[width=0.3\textwidth]{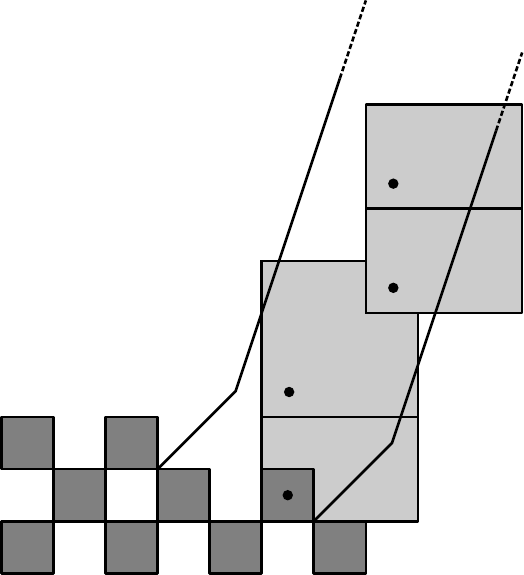}
\caption{Examples of whole coverings defined in formula \eqref{covering-0}.
The black dots represent the points of $\I(\ell_m)$.}
\label{fig:ricoprimenti-compl}
\end{figure}
(see Fig.~\ref{fig:ricoprimenti-compl} for an example of $\{C(\iv)\,:\, \iv\in \I(\ell_m)\}$).

\begin{figure}[htbp]
\centering
\includegraphics[width=0.3\textwidth]{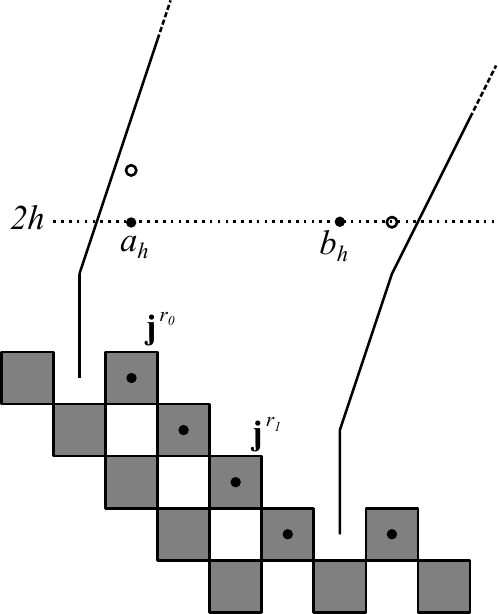}
\caption{The point $(a_h,2h+1)\in A(\ell_m)$ (black circle on the left).
The point $(b_h+1,2h)\in A(\ell_m)$ (black circle on the right).}
\label{fig:a-b-slant}
\end{figure}

{\bf Step 3: covering of the region outside ${E}$ projecting onto $\ell_m$ with $m_0\le m\le m_1-1$.}
As before, we label clockwise the set of points
$\{\jv^r\}_{r\geq0}=\bigcup_{m=m_0+1}^{m_1}\ell_m\setminus\{\ell_0\}$.
For every $m_0+1\le m\le m_1-1$, writing $\ell_m=\{\jv^l\}_{l=0}^L$ we define 
\begin{equation}\label{order-points}
r_0:=\min\{r\in2\Z:\,\jv^r\in\ell_m\} \,\mbox{ and }\, r_1:=\max\{r\in2\Z:\,\jv^r\in\ell_m\setminus\{\jv^L\}\}\,.
\end{equation}
\begin{figure}[htbp]
\centering
\includegraphics[width=0.3\textwidth]{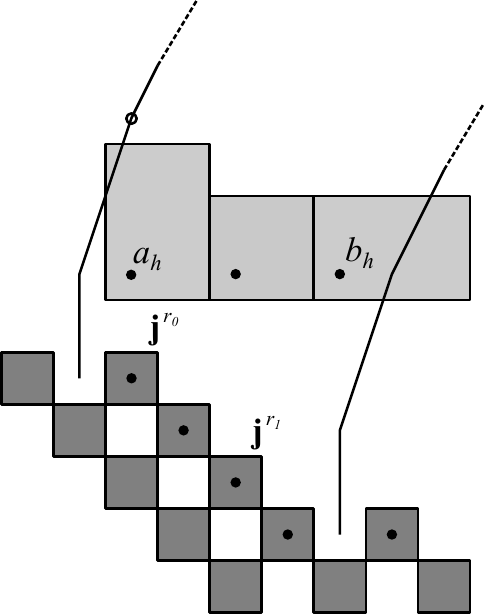}
\quad
\includegraphics[width=0.3\textwidth]{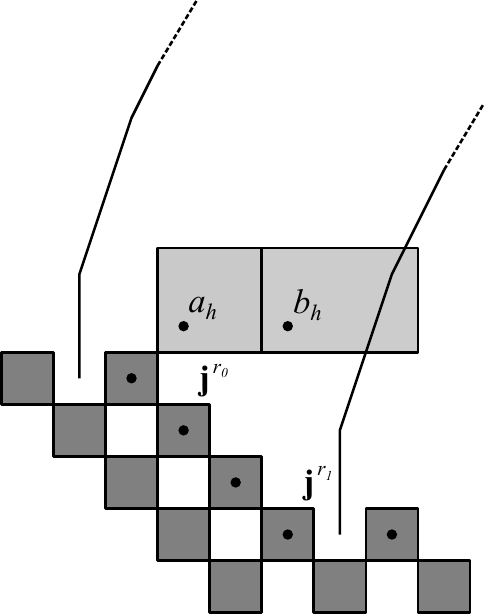}
\caption{Examples of coverings of the two possible parities defined in formula \eqref{covering-0}. The black circle represents the point $(a_h,2h+3)$.}
\label{fig:ricoprimenti-slant}
\end{figure}

\begin{figure}[htbp]
\centering
\includegraphics[width=0.3\textwidth]{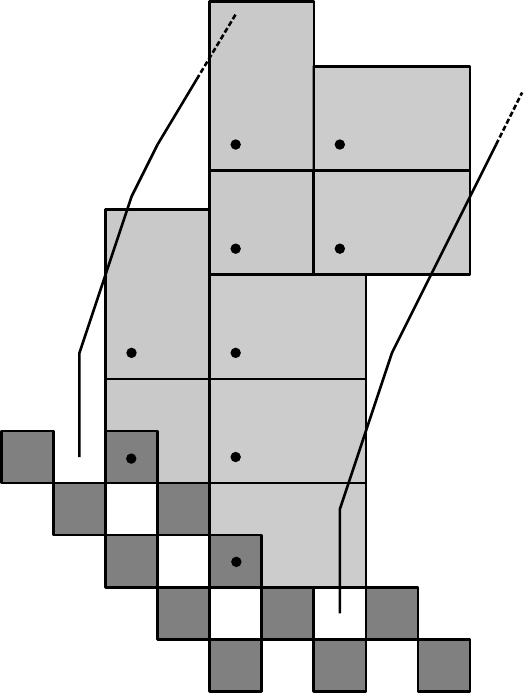}
\quad
\includegraphics[width=0.3\textwidth]{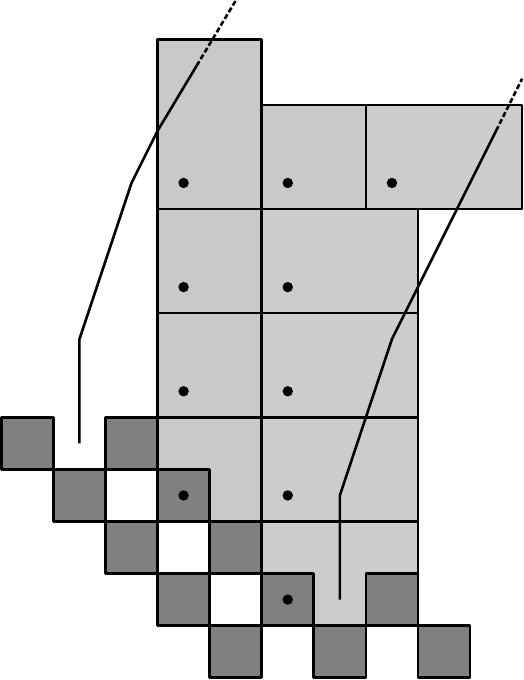}
\caption{Examples of whole coverings defined in formula \eqref{covering-0} of two different parities. The black dots represent the points of $\I(\ell_m)$.}
\label{fig:ricoprimenti-slant-compl}
\end{figure}

Fix first $\ell_m$ with $m_0+1< m\le m_1-1$ and assume, without loss of generality, that $\jv^{r_1}=(0,0)$.
Now, for every $h\in\N$ we determine the integers $a_h$ and $b_h$ as in \eqref{extrema} (see Fig.~\ref{fig:a-b-slant}), which are well defined by Remark \ref{remk:monotone-edges}, where $A(\ell)$ is as in \eqref{ext-edge-slant}.
Correspondingly, we define the sets $\I(\ell_m)$ as in \eqref{indices} and $C(\iv)$ for every $\iv\in\I(\ell_m)$ as in \eqref{covering-0}, respectively (see Figg.~\ref{fig:ricoprimenti-slant} and \ref{fig:ricoprimenti-slant-compl}).

The covering outside the discrete edges $\ell_{m_0}$ and $\ell_{m_0+1}$ must be treated separately.
Let $r_0,r_1$ be as in \eqref{order-points} with $m=m_0+1$.
Again, we assume that $\jv^{r_1}=(0,0)$, define $a_h$, $b_h$ as in \eqref{extrema} for every $h\in\N$ with $A(\ell_{m_0})\cup A(\ell_{m_0+1})$ in place of $A(\ell_m)$
and the set $\I(\ell_{m_0}\cup\ell_{m_0+1})$ as in \eqref{indices}.
The sets $C(\iv)$ are defined, for every $\iv\in\I(\ell_{m_0}\cup\ell_{m_0+1})$, as in \eqref{covering-0} with $A(\ell_{m_0})\cup A(\ell_{m_0+1})$ in place of $A(\ell_m)$
(see Fig.~\ref{fig:ricoprimenti-m0}).
Note that, in this case $a_h\neq b_h$ for every $h\in\N$.
\begin{figure}[htbp]
\centering
\includegraphics[width=.48\linewidth]{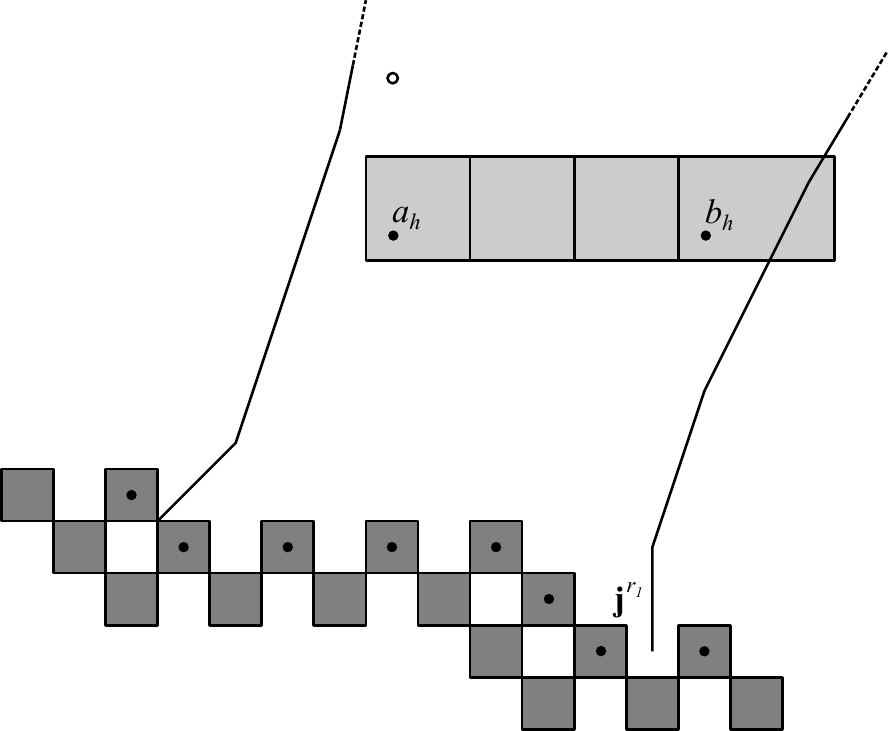}
\quad
\includegraphics[width=.48\linewidth]{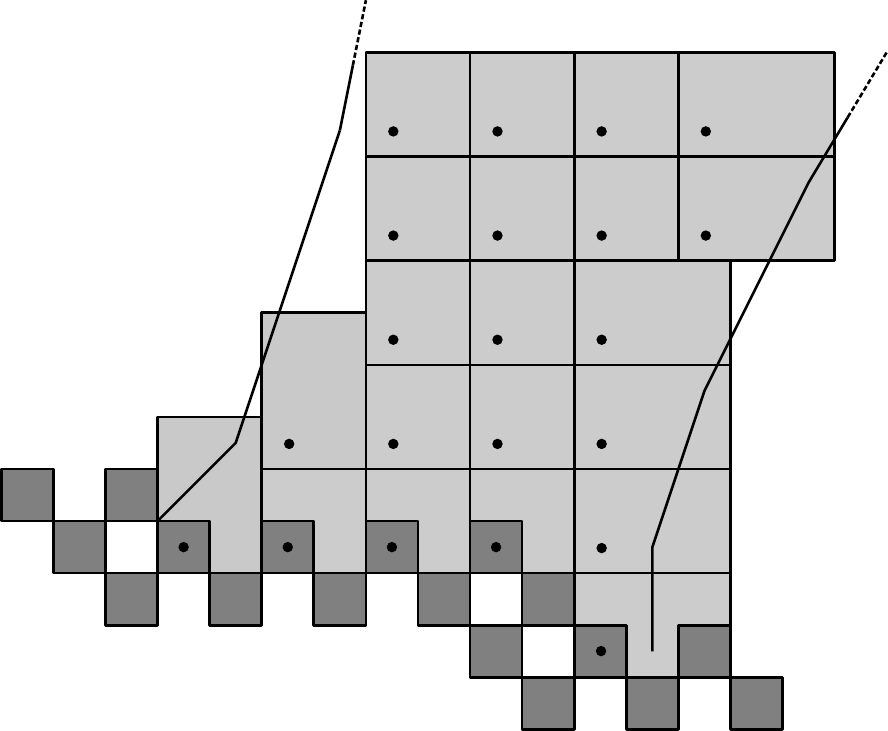}
\caption{On the left, an example of $a_h$ and $b_h$, and the black circle represents the point $(a_h,2h+3)$.
On the right the corresponding covering, where the black dots represent the points of $\I(\ell_{m_0-1}\cup\ell_{m_0})$.}
\label{fig:ricoprimenti-m0}
\end{figure}

\begin{figure}[htbp]
\centering
\includegraphics{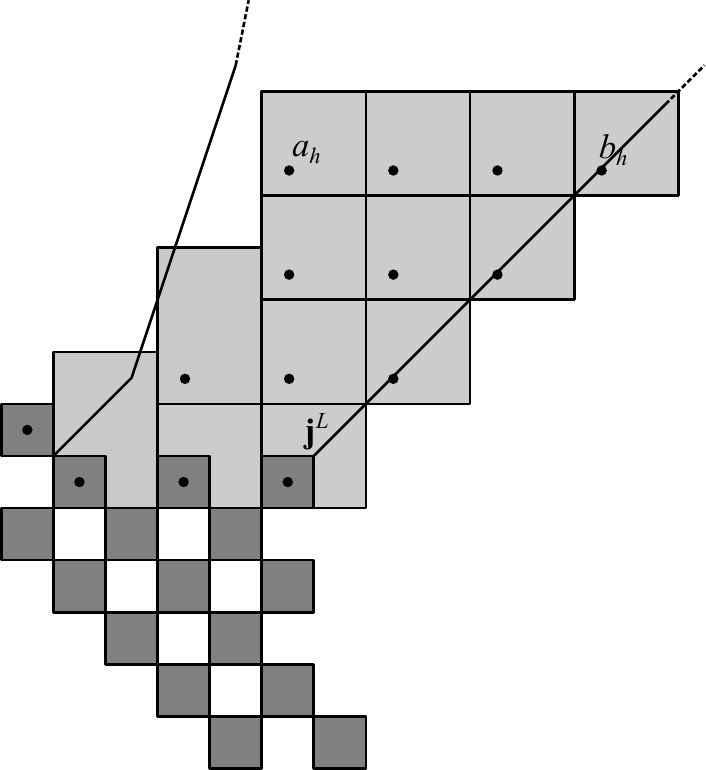}
\caption{Example of covering outside $\ell_{m_1}$ in the case (i).}
\label{fig:ricoprimenti-m1-i}
\end{figure}
{\bf Step 4: covering of the region outside ${E}$ projecting onto $\ell_{m_1}$.}
We set $\ell_{m_1}=\{\jv^l\}_{l=0}^L$ with $\jv^L=(0,0)$.
There are different possible cases depending on $\bm\nu(\ell_{m_1})$:

\begin{figure}[h]
\centering
\includegraphics{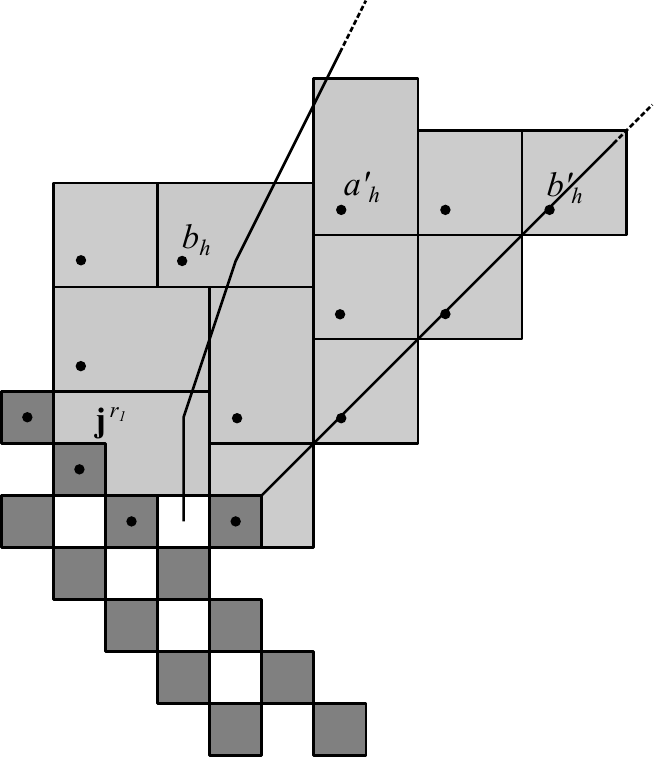}
\quad
\includegraphics{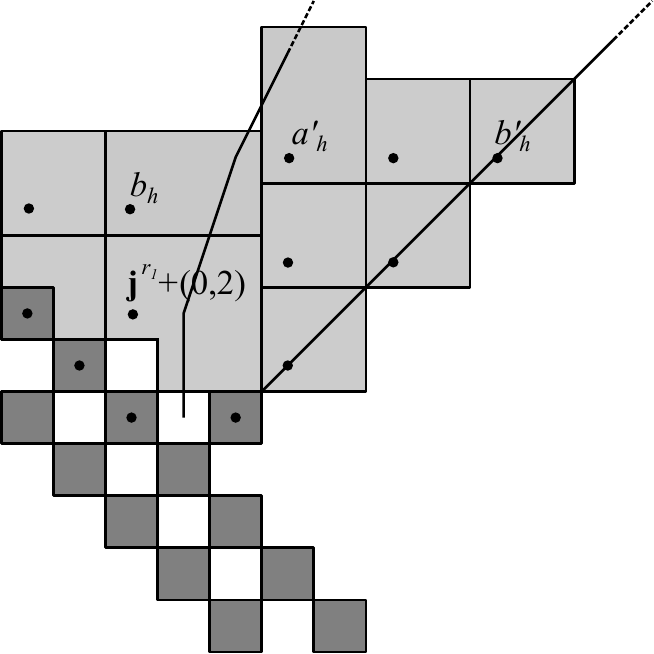}
\caption{The covering in the cases $\jv^{r_1}=\jv^{L-1}$ (on the left) and $\jv^{r_1}=\jv^L$ (on the right).}
\label{fig:ricoprimenti-m1-ii}
\end{figure}
\smallskip
(i) let $m_0=m_1$; \emph{i.e.}, $s(\ell_m)\le\frac{1}{3}$ for every $m$.
We set, for every $h\in\N$, $a_h$ as in \eqref{extrema}, $b_h=2h$ and $\I(\ell_{m_1})$ as in \eqref{indices}.
Then, $C(\iv)$ is defined as in \eqref{covering-0} for every $\iv\in\I(\ell_{m_1})\setminus\{(b_h,2h)\}_{h\in\N}$ and $C((b_h,2h))=Q((b_h,2h))$ for every $h\in\N$ (Fig.~\ref{fig:ricoprimenti-m1-i});

\begin{figure}[htbp]
\centering
\includegraphics{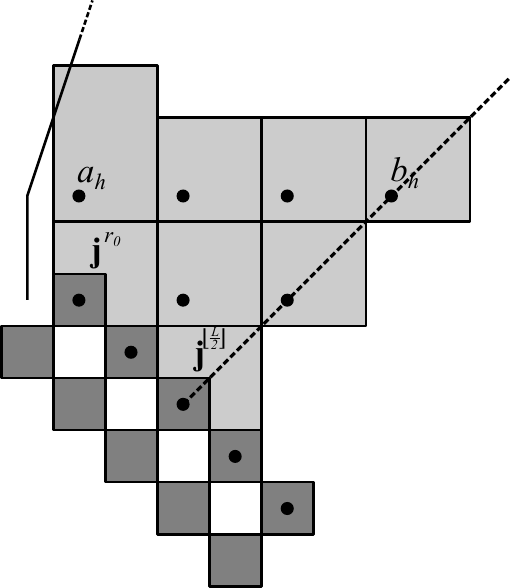}
\caption{The covering considered in (iii) in the case $L$ even.}
\label{fig:ricoprimenti-m1-iii-ev}
\end{figure}
\smallskip
(ii) let
$
\frac{1}{3}<s(\ell_{m_1})<1
$
and let $r_1$ be defined as in \eqref{order-points} with $m=m_1$.
Then $a_h$ and $b_h$ are as in \eqref{extrema} for every $h\in\N$ with $A(\ell_{m_1})$ or $A(\ell_{m_1})\cup A(\ell_{m_1-1})$ in place of $A(\ell_m)$ whether $m_1-1>m_0$ or $m_1-1=m_0$, respectively.
We define $\I(\ell_{m_1})$ as in \eqref{indices}.
If $\jv^{r_1}=\jv^{L-2}$, $C(\iv)$ is defined as in \eqref{covering-0}.
Whereas, if $\jv^{r_1}=\jv^{L-1}$, $C(\iv)$ is defined as in \eqref{covering-0} for every $\iv\in\I(\ell_{m_1})\setminus\{\jv^{r_1},\jv^{r_1}+(0,2)\}$ and
$$
C(\jv^{r_1})=\emptyset,
\quad
C(\jv^{r_1}+(0,2))=R^-(\jv^{r_1}+(0,2)).
$$
Then, setting $A(\jv^L)=\{\iv\in\Z^2 \,:\, i_1,i_2>0,\, \pi_E^\varphi(\iv)=\jv^L \}$, we introduce the integers
$$
\begin{aligned}
a_h' &=\begin{cases}
\min\{h'\in2\Z\,:\, (h',2h+1)\in A(\jv^L)\}&\text{if } \jv^{r_1}=\jv^{L-2}\\
\min\{h'\in2\Z+1\,:\, (h',2h+2)\in A(\jv^L)\}&\text{if } \jv^{r_1}=\jv^{L-1}
\end{cases}\\
b_h' &=\begin{cases}
2h&\text{if } \jv^{r_1}=\jv^{L-2}\\
2h+1&\text{if } \jv^{r_1}=\jv^{L-1}.
\end{cases}
\end{aligned}
$$
Now, we define $\I(\ell_{m_1}')$ as in \eqref{indices} with $a_h', b_h'$ in place of $a_h$ and $b_h$, and the tile
$C(\iv)$ as in \eqref{covering-0} for every $\iv\in\I(\jv^L)\setminus\{(b_h',b_h')\}_{h\in\N}$, and $C((b_h',b_h'))=Q((b_h',b_h'))$ (see Fig.~ \ref{fig:ricoprimenti-m1-ii});

\smallskip
(iii) consider now the case $s(\ell_{m_1})=1$.
Let $r_0$ be defined as in \eqref{order-points} with $m=m_1$.
Without relabeling, we set $\ell_{m_1}:=\{\jv^l\}_{l=0}^{\lfloor\frac{L}{2}\rfloor}$ and assume $\jv^{\lfloor\frac{L}{2}\rfloor}=(0,0)$.
Here the covering depends on the parity of $L$.
If $L$ is even, $a_h$ is defined as in \eqref{extrema} with $m=m_1$ and $b_h=2h$ for every $h\in\N$.
$\I(\ell_{m_1})$ is defined as in \eqref{indices}.
Then $C(\iv)$ is defined as in \eqref{covering-0} for every $\iv\in \I(\ell_{m_1})\setminus\{(b_h,2h)\}_{h\in\N}$ and $C((b_h,2h))=Q((b_h,2h))$ (Figure \ref{fig:ricoprimenti-m1-iii-ev}).
\begin{figure}[htbp]
\centering
\includegraphics{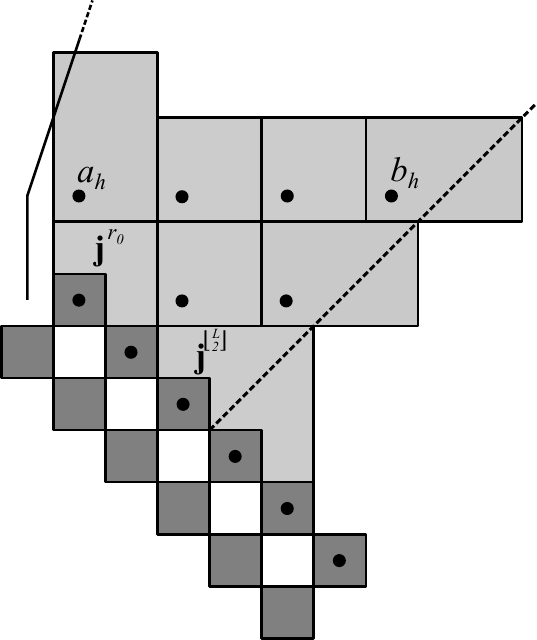}
\quad
\includegraphics{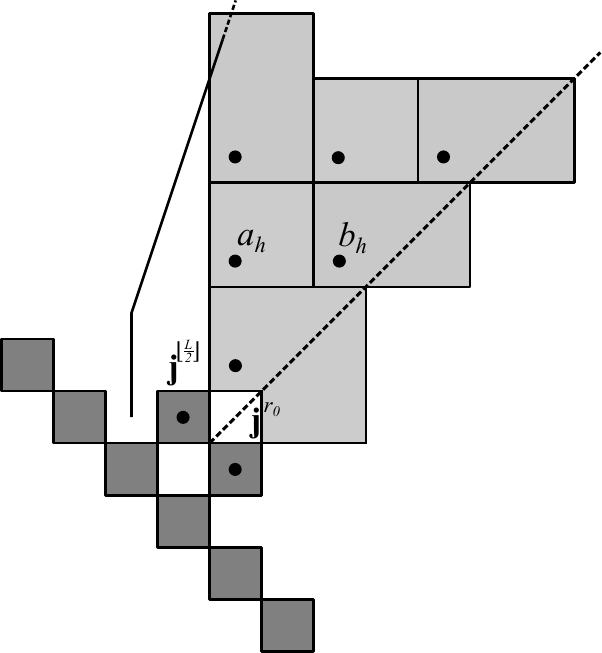}
\caption{The coverings defined in (iii), in the case $L$ odd, for $\jv^{r_0}\neq\jv^L$ (on the left) and $\jv^{r_0}=\jv^L$ (on the right).}
\label{fig:ricoprimenti-m1-iii-od}
\end{figure}
If $L$ is odd, analogously to what done in case (ii), for every $h\in\N$ we define
$$
\begin{aligned}
a_h &=\begin{cases}
\min\{h'\in2\Z\,:\, (h',2h+1)\in A(\ell_{m_1})\}&\text{if } \jv^{r_0}\neq\jv^L\\
\min\{h'\in2\Z+1\,:\, (h',2h+2)\in A(\jv^L)\}&\text{if } \jv^{r_0}=\jv^L
\end{cases}\\
b_h' &=\begin{cases}
2h&\text{if } \jv^{r_0}\neq\jv^L\\
2h+1&\text{if } \jv^{r_1}=\jv^L.
\end{cases}
\end{aligned}
$$
Then $\I(\ell_{m_1})$ is defined as in \eqref{indices} and $C(\iv)$ is defined as in \eqref{covering-0} for every $\iv\in\I(\ell_{m_1})\setminus\{(b_h,2h)\}_{h\in\N}$ and
$$
C_h((b_h,b_h)) = \begin{cases}
R^-((b_h,b_h)) & \text{if } h=0\\
R((b_h,b_h)) & \text{if } h>0,
\end{cases}
$$
see Fig.~\ref{fig:ricoprimenti-m1-iii-od}.

\begin{figure}[htbp]
\centering
\includegraphics[width=0.27\textwidth]{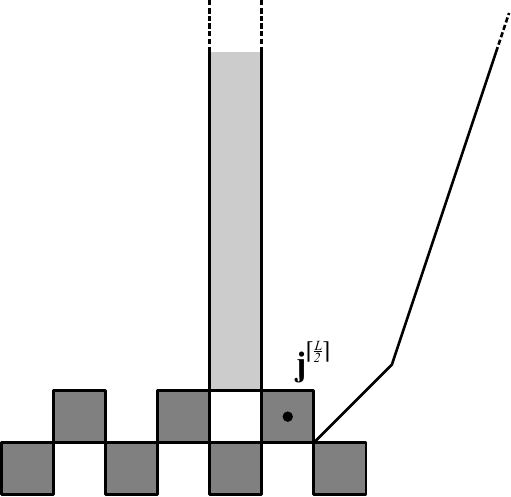}
\quad
\includegraphics[width=0.3\textwidth]{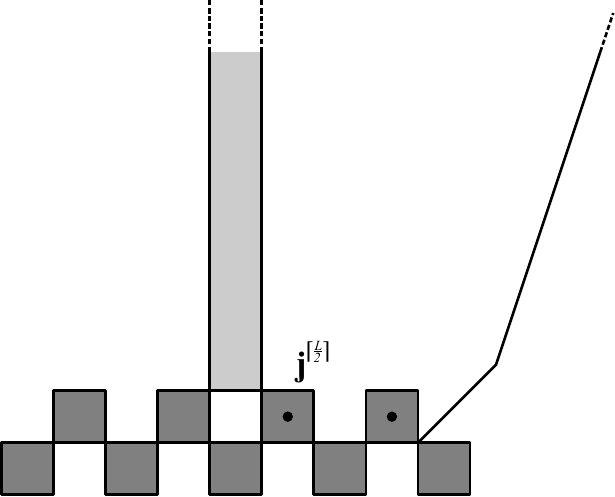}
\quad
\includegraphics[width=0.27\textwidth]{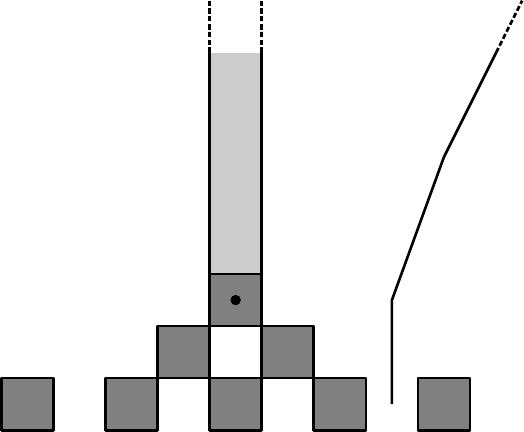}
\caption{The covering $\mathcal{S}_0$ in the cases listed in \eqref{axis-first-edge}.}
\label{fig:ricotrpimenti-top}
\end{figure}
{\bf Step 5: covering of the region outside $E$ projecting onto $\ell_0$.}
We define the set
\begin{equation}\label{axis-first-edge}
\mathcal{S}_0 = \begin{cases}
E(\{\iv\in\Z^2 \,:\, i_1=j_1^{\lceil\frac{L}{2}\rceil}-1,\, i_2\ge1\}) & \mbox{  if  }\,\, s(\ell_1)=0\,, \\
\emptyset & \displaystyle \mbox{  if  }\,\, 0<s(\ell_1)<\frac{1}{3}\,, \\
E(\{\iv\in\Z^2 \,:\, i_1=j_1^0,\, i_2\ge0\}) & \displaystyle \mbox{  if  }\,\, \frac{1}{3}<s(\ell_1)\le1\,.
\end{cases}
\end{equation}
(see Fig.~\ref{fig:ricotrpimenti-top}).

\begin{figure}[htbp]
\centering
\includegraphics{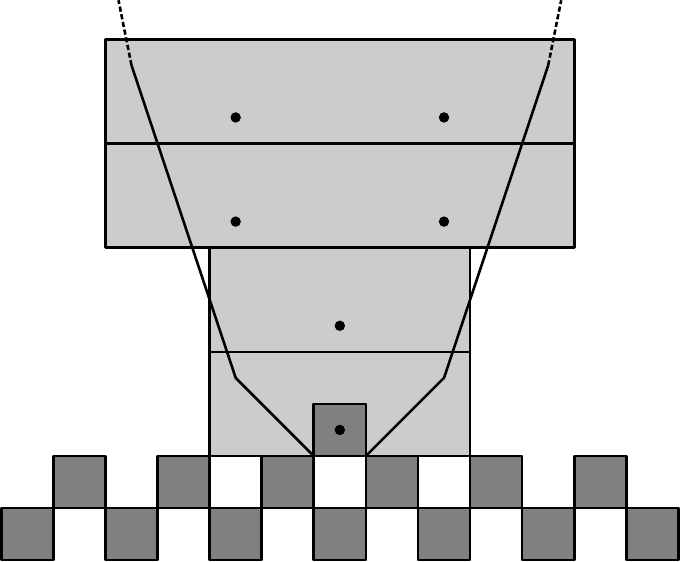}
\caption{Example of $C(\iv)$, $\iv\in\I(\ell_0)$.
The black dots represent the lattice points $(\pm b_h,2h).$}
\label{fig:ricotrpimenti-top-punta}
\end{figure}
If $\ell$ is such that $0<s(\ell)<\frac{1}{3}$,
we define $\I(\ell_0)=\{(0,2h)\}_{h\in\N}$ and for every $h\in\N$
$$
b_h = \max\{h'\in2\Z\,|\,(h'+1,2h)\in A(\jv^0)\},
$$
where $A(\ell_0)$ is as in \eqref{ext-edge-top0}.
Then, for every $\iv\in\I(\ell_0)$ we choose the tile 
\begin{equation}\label{axis-first-edge2}
C(\iv) = \bigcup\big\{q((k,2h))\cup q((k,2h+1)) \,:\, k\in\Z\,,\, -b_h-2\le k \le b_h+2\,,\, i_2=2h\big\},
\end{equation}
see Fig.~\ref{fig:ricotrpimenti-top-punta}.

{\bf Step 6: compatibility between different coverings.}
Here, we note that the family of sets $\{C(\iv)\,:\, \iv\in\I(\ell_m),\, 0\le m\le m_1,\, \iv\in\I(\ell_{m_1}')\}$ is a covering of $E(\{\iv\in\Z^2\,:\, \inf_{\jv\in Z(E)}\|\iv-\jv\|_1,\, i_2\ge i_1\})$, which is the region of plane ``outside'' the edges as in Step 1.
We point out that, if case (ii) of Step 4 does not hold, then $\I(\ell_{m_1}')=A(\ell_{m_1}')=\emptyset$.
Indeed, for every pair $\ell,\ell'\in\mathcal{E}(E)$ with $\ell'$ preceding $\ell$, the sets
$$\bigcup_{\iv\in\I(\ell')}C(\iv)
\quad \text{and} \quad
\bigcup_{\iv\in\I(\ell)}C(\iv)$$
are non-overlapping and their union does not leave uncovered regions.

\begin{figure}[htbp]
\centering
\includegraphics{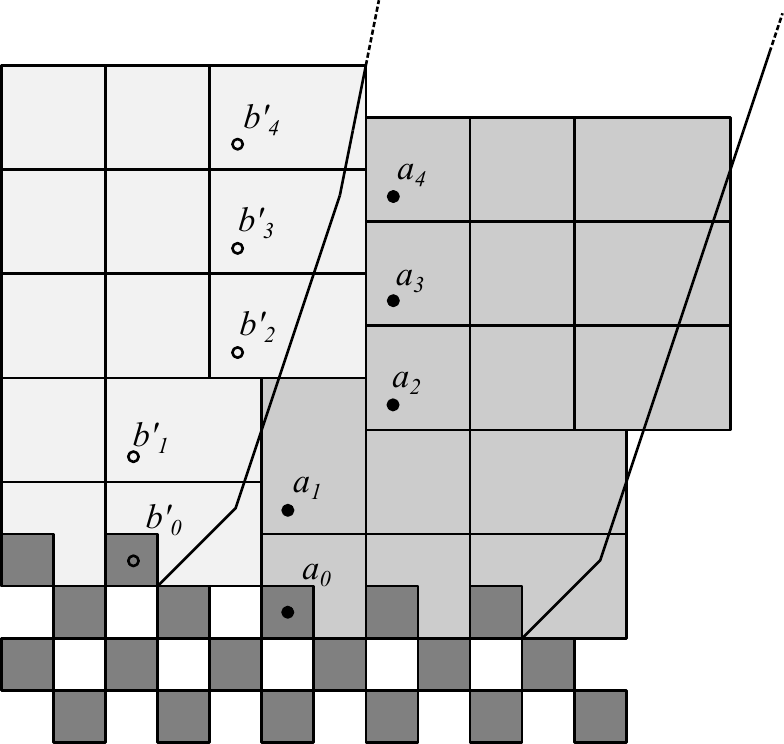}
\caption{Matching of the coverings outside a pair of adjacent discrete edges.}
\label{fig:compatibilita}
\end{figure}

We denote by $a_h, b_h$ and $a_h', b_h'$ the values defined in \eqref{extrema} corresponding to $\ell$ and $\ell'$, respectively.
We assume, for simplicity, that $\jv^L=(0,0)$.
Hence, every $\iv\in\I(\ell)$ and $\iv'\in\I(\ell')$ are such that $i_2=2h$ and $i'_2=2h+1+2h_0$, where $h_0=0$ if $0\le s(\ell)\le\frac{1}{3}$ and $2h_0=r_1-r_0$ otherwise, where $r_0$ and $r_1$ are defined in Step 3.
Therefore, in this coordinate system, the definition of $b_{h-h_0}'$ reads
$$
b_{h-h_0}'=\max\{h'\in2\Z+1 \,:\, (h'+1,2h+1)\in A(\ell')\}\,.
$$
Now, it is sufficient to note that, if $Q((a_h,2h))=R^{\rm ver}((a_h,2h))$ then $a_{h+1}=a_h+2$, while if $Q((a_h,2h))=Q((a_h,2h))$ then $a_{h+1}=a_h$, as it immediately follows from \eqref{covering-0} (see Fig.~\ref{fig:compatibilita}).

The covering of the regions projecting onto discrete edges $\ell\in\mathcal{E}(E)$ not fulfilling \eqref{edge-simpl} can be obtained symmetrically; we use the notation $\I(\ell)$ and $C(\iv)$ to denote the sets obtained symmetrically as in \eqref{indices} and \eqref{covering-0} respectively.
With $\mathcal{C}_0$ we denote the union of the set $\mathcal{S}_0$ defined in \eqref{axis-first-edge} and its symmetric analogs.

\begin{figure}[h]
\centering
\def\svgwidth{250pt}
\begingroup%
  \makeatletter%
  \providecommand\color[2][]{%
    \errmessage{(Inkscape) Color is used for the text in Inkscape, but the package 'color.sty' is not loaded}%
    \renewcommand\color[2][]{}%
  }%
  \providecommand\transparent[1]{%
    \errmessage{(Inkscape) Transparency is used (non-zero) for the text in Inkscape, but the package 'transparent.sty' is not loaded}%
    \renewcommand\transparent[1]{}%
  }%
  \providecommand\rotatebox[2]{#2}%
  \newcommand*\fsize{\dimexpr\f@size pt\relax}%
  \newcommand*\lineheight[1]{\fontsize{\fsize}{#1\fsize}\selectfont}%
  \ifx\svgwidth\undefined%
    \setlength{\unitlength}{400.93177135bp}%
    \ifx\svgscale\undefined%
      \relax%
    \else%
      \setlength{\unitlength}{\unitlength * \real{\svgscale}}%
    \fi%
  \else%
    \setlength{\unitlength}{\svgwidth}%
  \fi%
  \global\let\svgwidth\undefined%
  \global\let\svgscale\undefined%
  \makeatother%
  \begin{picture}(1,0.27416915)%
    \lineheight{1}%
    \setlength\tabcolsep{0pt}%
    \put(0,0){\includegraphics[width=\unitlength,page=1]{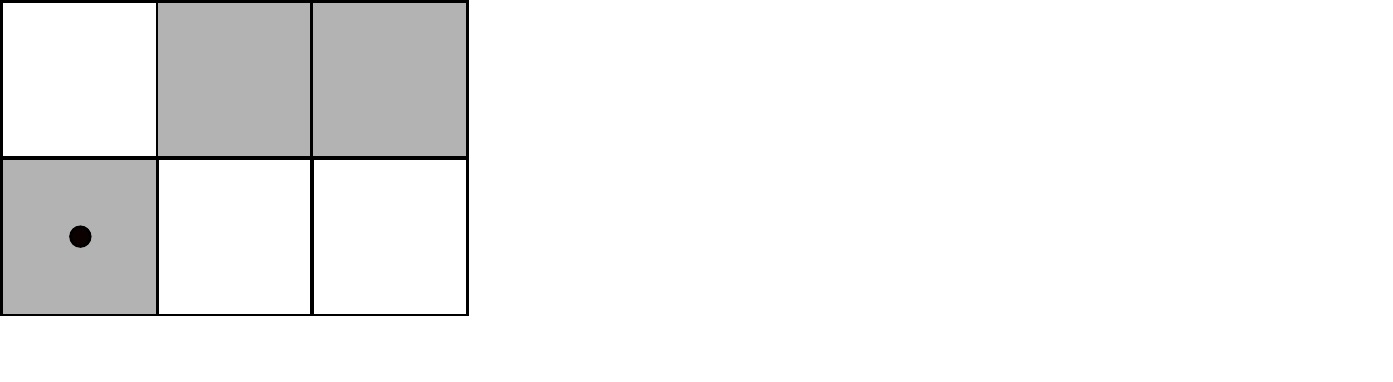}}%
    \put(0.08287961,0.00366089){\color[rgb]{0,0,0}\makebox(0,0)[lt]{\lineheight{1.25}\smash{\begin{tabular}[t]{l}$R^{\rm hor}(\iv)$\end{tabular}}}}%
    \put(0,0){\includegraphics[width=\unitlength,page=2]{retthorconf.pdf}}%
    \put(0.08306953,0.08528548){\color[rgb]{0,0,0}\makebox(0,0)[lt]{\lineheight{1.25}\smash{\begin{tabular}[t]{l}$\iv$\end{tabular}}}}%
  \end{picture}%
\endgroup%

\caption{The checkerboard configurations are energetically favorable inside each $R^{\rm hor}(\iv)$.}\label{fig:retthorconf}
\end{figure}
{\bf Step 7: local minimum problems on $C(\iv)$.}
As a next step, we prove that the configuration with minimal energy inside each tile $C(\iv)$ is the even checkerboard, for any $\iv\in\I(\ell)$, $\ell\in\mathcal{E}(E)$; \emph{i.e.},
\begin{equation}\label{local-min}
\mathcal{F}_\alpha^\varphi(E(\Ze)\cap C(\iv), E) \le \mathcal{F}_\alpha^\varphi(F\cap C(\iv), E)
\end{equation}
for every $F\in\mathcal{D}$, and the same for $\mathcal{C}_0$; \emph{i.e.},
\begin{equation}\label{local-min-croce}
\mathcal{F}_\alpha^\varphi(E(\Ze)\cap\mathcal{C}_0, E) \le \mathcal{F}_\alpha^\varphi(F\cap \mathcal{C}_0, E).
\end{equation}

Indeed, if $C(\iv)=Q(\iv)$ from Remark \ref{BS-proj} either \eqref{eq:intersection-flat} or \eqref{eq:intersection-slant} holds.
Hence, by arguing as in the proof of Proposition \ref{firststep}, from \eqref{normestimate} we get \eqref{local-min}.

If $C_h(\iv)=R^{\rm hor}(\iv)$, we can restrict the minimization in \eqref{local-min} to the checkerboard configurations.
Indeed, if $\jv\in Z(R^{\rm hor}(\iv))$ has a nearest neighbor $\jv'$ then by suitably shifting one of them towards an ``empty'' location the corresponding variation of the energy is at most $-2+\alpha<0$ (see Fig.~\ref{fig:retthorconf}); the case $\alpha>2$ is trivial.
Moreover, by the definition of $b_h$ we have that either \eqref{eq:intersection-flat} or \eqref{eq:intersection-slant} is satisfied, thence from Remark \ref{rectremk}, \eqref{rect} holds yielding \eqref{local-min}.
The cases of $C(\iv)=R^{\rm ver}(\iv)$, $C(\iv)=R^+(\iv)$ and $C(\iv)=R^-(\iv)$ can be treated analogously.

\begin{figure}[htbp]
\centering
\includegraphics{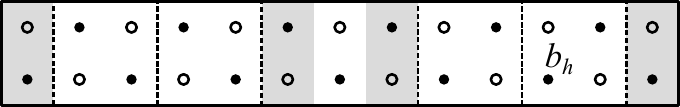}
\caption{The lattice points involved in \eqref{energy-top}.
The black dots are points of $\Ze$, circles are points of $\Zo$.
The energy contribution of the even checkerboard in the white regions is negative.}
\label{fig:energia-top}
\end{figure}
Now consider the case $C(\iv)$ as in \eqref{axis-first-edge2} with $\iv\in \I(\ell_0)$.
Reasoning as above, we can reduce minimum problem \eqref{local-min} to a comparison between the energies of the two checkerboards.
Then, for every $h\in\N$, $k\in2\Z$ with $0<|k|\le b_h$, the even checkerboard has minimum energy in $Q((k,2h))$, as above.
Hence \eqref{local-min} is proved if
\begin{multline*}
\varphi(0,2h)+2\varphi(1,2h+1)+2d^\varphi((b_h+2,2h),E) \\
\le \varphi(0,2h+1)+2\varphi(1,2h)+2d^\varphi((b_h+2,2h+1),E)\,;
\end{multline*}
that is,
\begin{equation}\label{energy-top}
\varphi(1,2h+1)+d^\varphi((b_h+2,2h),E)\le \frac{1}{2}+\varphi(1,2h)+d^\varphi((b_h+2,2h+1),E),
\end{equation}
see Fig.~\ref{fig:energia-top}.
If $(b_h+2,2h+1)\not\in A(\ell_0)$ the inequality above is trivial, since $d^\varphi((b_h+2,2h),E)\le\varphi(1,2h)$ and $d^\varphi((b_h+2,2h),E)\le\varphi(1,2h+1)$.
If, instead, $(b_h+2,2h+1)\in A(\ell_0)$ \eqref{energy-top} reduces to
$$
\varphi(1,2h+1)+\varphi(b_h+2,2h)\le \frac{1}{2}+\varphi(1,2h)+\varphi(b_h+2,2h+1),
$$
which holds from \eqref{normestimate} and {\ref{norm-derivative}}.

Reasoning as in point (c) of the proof of Proposition \ref{firststep} there holds
$$
\mathcal{F}_\alpha^\varphi(E(\Ze)\cap\mathcal{C}_0,E) \le \mathcal{F}_\alpha^\varphi(F\cap\mathcal{C}_0,E)\,.
$$

As a final remark, we note that for every $\iv\in E(\Ze)$ such that $d^\varphi(\iv,E)>\frac{4}{\alpha}$ the variation of removing $q(\iv)$ is negative, hence
$$
\argmin{E'\supset E}\mathcal{F}_\alpha^\varphi(E',E) \subset E\Big(\Big\{\iv\in\Z^2 \,:\, d^\varphi(\iv,E)<\frac{4}{\alpha}\Big\}\Big)\,.
$$

{\bf Step 8: conclusion.}
Set
$$
\I:=\Big(\bigcup_{\ell\in\mathcal{E}(E)}\I(\ell)\Big)\cap \Big\{\iv\in\Z^2 \,:\, d^\varphi(\iv,E)<\frac{4}{\alpha}\Big\}\,.
$$
An analogous argument as that in the proof of Proposition~\ref{firststep} (see \eqref{subaddmin}) shows that
$$
\mathcal{F}_\alpha^\varphi(E',E)\ge\sum_{\iv\in\I}\mathcal{F}_\alpha^\varphi(E'\cap C(\iv),E)+\mathcal{F}_\alpha^\varphi(E'\cap\mathcal{C}_0,E)
$$
for every $E'\supset E$, $E'\in\D$.
By virtue of Step 7 we get
$$
\min_{E'\supset E}\mathcal{F}_\alpha^\varphi(E',E)\ge\sum_{\iv\in\I}\mathcal{F}_\alpha^\varphi(E(\Ze)\cap C(\iv), E)+\mathcal{F}_\alpha^\varphi(E(\Ze)\cap\mathcal{C}_0,E)=\mathcal{F}_\alpha^\varphi(E(\I\cup Z(E)),E)
$$
which implies that the ground state of the energy is achieved by the even checkerboard configuration.
Lastly, the monotonicity constraint yields the uniqueness of the solution.
\end{proof}

%

We will apply Proposition~\ref{steps} iteratively to each $E=E_\alpha^k$, $k\geq1$  in order to characterize the solutions of the recursive scheme \eqref{MM-scheme2} (see Theorem~\ref{thm:nucleation}).
Indeed, as shown with Proposition~\ref{firststep}, the first step $E_\alpha^1$ coincides with $E(B_\frac{4}{\alpha}^\varphi\cap \Ze)$ which satisfies the symmetry conditions \eqref{E-symm} and, thanks to the following Lemma, the non-degeneracy condition \eqref{non-degeneracy}.

{
\begin{lemma}\label{non-deg-lem}
If $\varphi$ satisfies {\rm\ref{ass-H1}} and {\rm\ref{ass-H2}}, then for every $r>2$ the set $E=B_r^\varphi\cap\Ze$ satisfies \eqref{non-degeneracy}.
\end{lemma}
\begin{proof}
By the symmetric assumption {\ref{ass-H1}} we can restrict our analysis to points $\iv\in\partial^{\rm eff}E$ with $i_2\ge i_1\ge0$.
We subdivide the proof into two cases.
If $i_1=0$, then $i_2>0$ from {\ref{ass-H2}} and the condition $r>2$.
Since $(0,0)\in Z(E)$ we have that $(0,i_2-2)\in Z(E)$.
By {\ref{ass-H1}} the point $(\pm i_2,0)\in Z(E)$ then, by the $\Ze$-convexity of $E$ we get that $(\pm1,i_2-1)\in Z(E)$.
Since for every $\iv'$ with $i_2'>i_2$, $\iv'\not\in Z(E)$ thus $\iv$ is non-degenerate.
If, instead, $i_1>0$, for every $\jv\in Z(E)$ such that $\|\jv-\iv\|_1\le 2$, by the symmetry with respect to the coordinate axes of $\varphi$ we get that $(-j_1,j_2)$, $(j_1,-j_2)\in Z(E)$.
The $\Ze$-convexity and the fact that $(0,0)\in Z(E)$ yield that $(j_1-2,j_2)$, $(j_1,j_2-2)$, $(j_1+1,j_2-2)\in Z(E)$.
This implies that $\iv$ is non-degenerate.
\end{proof}
}

\begin{figure}[t]
\begin{minipage}{0.3\linewidth}
\begin{center}
\includegraphics[width=0.6\textwidth]{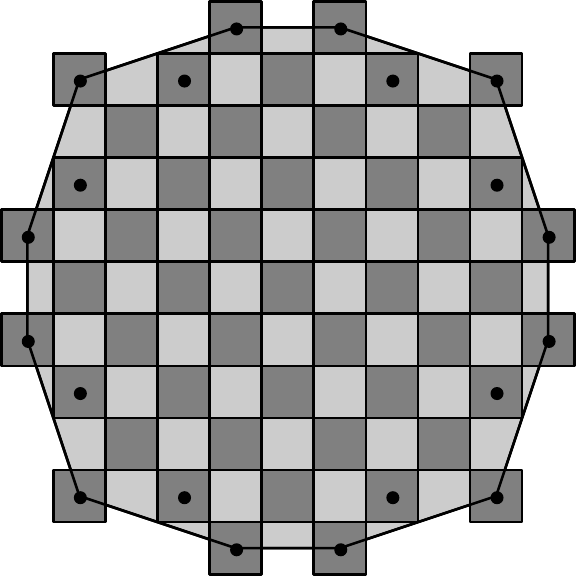}
\end{center}
\end{minipage}
\begin{minipage}{0.32\linewidth}
\begin{center}
\includegraphics[width=\textwidth]{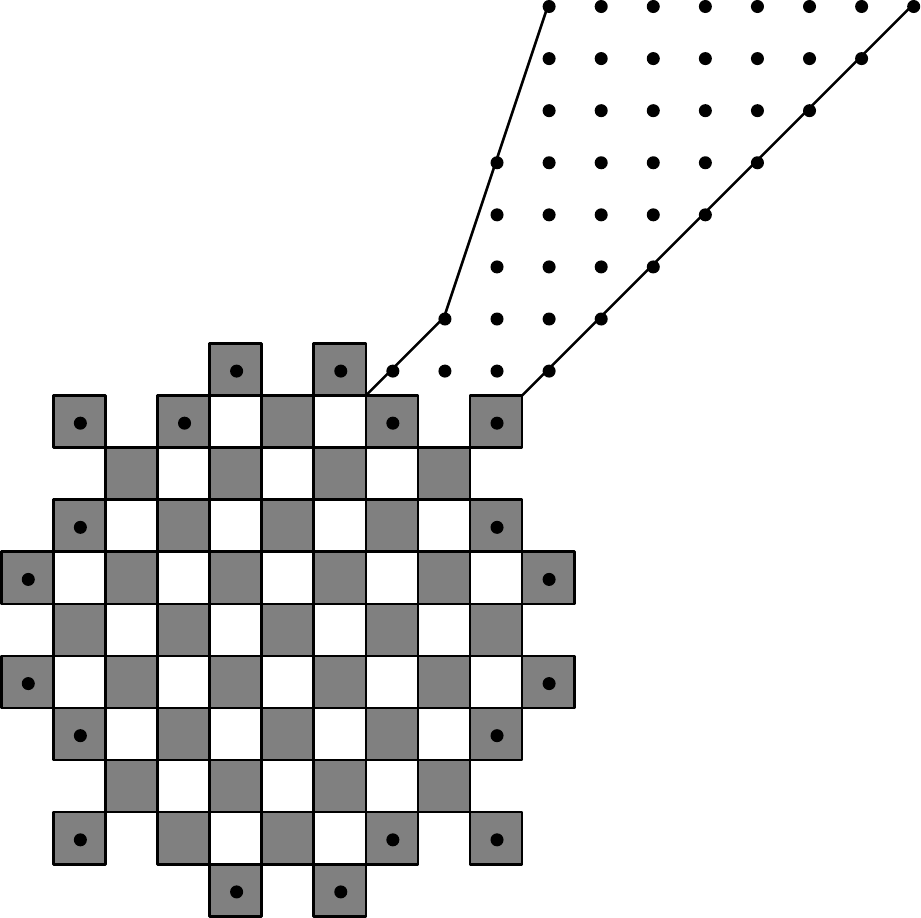}
\end{center}
\end{minipage}
\quad
\begin{minipage}{0.32\linewidth}
\begin{center}
\includegraphics[width=\textwidth]{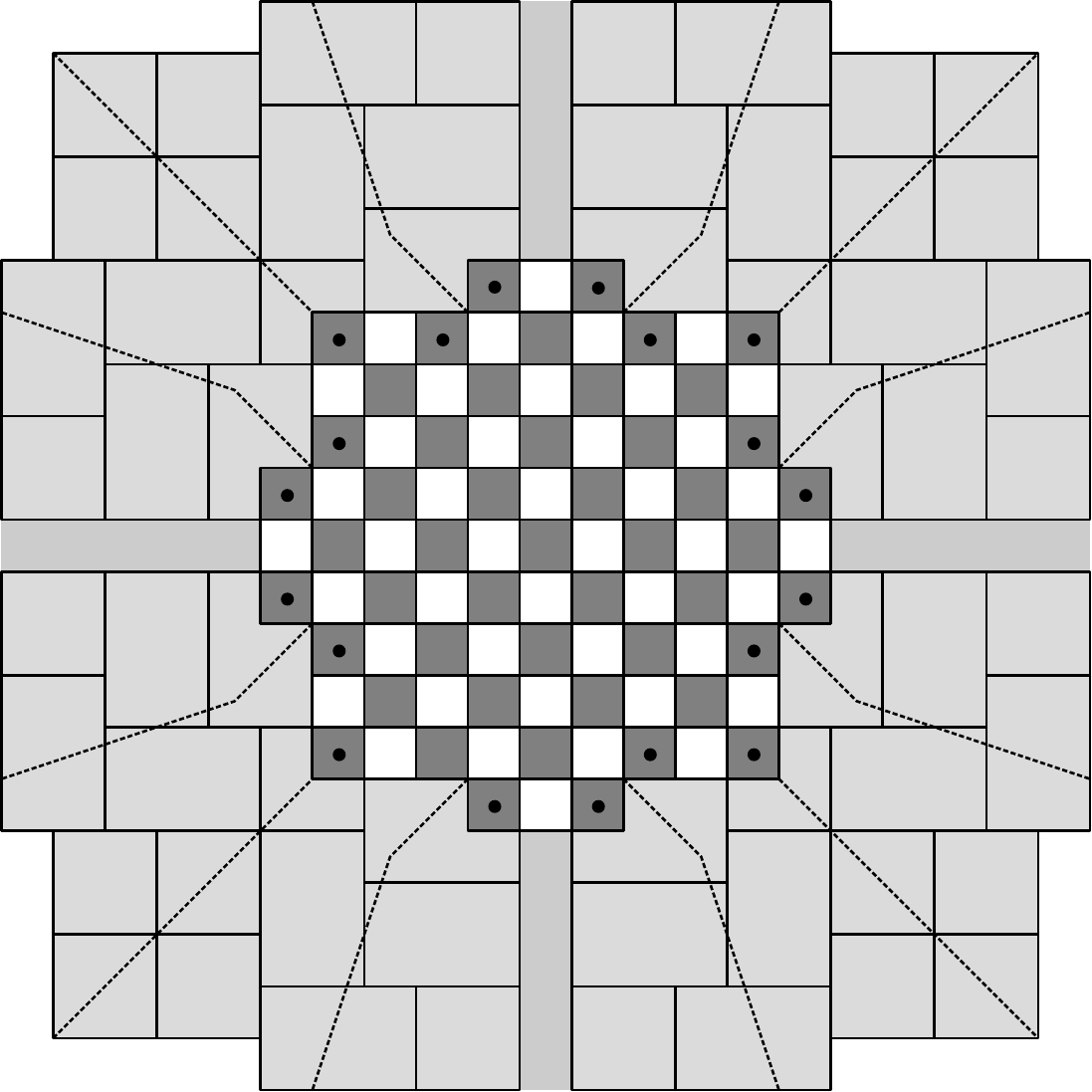}
\end{center}
\end{minipage}
\caption{From the left, the set $E^1_\alpha$ and the polygon $\conv(Z(E^1_\alpha))$,
the lattice sets $A(\ell)$, $\ell\in\mathcal{E}(E)$ and, lastly, the corresponding covering.}
\label{fig:ex-tech-ass}
\end{figure}

We conclude this section with some examples clarifying the role of compatibility assumption \eqref{monotone-edges} and non-degeneracy condition {\ref{ass-H2}}.

\begin{example}\label{ex:tech-ass}

Consider $\varphi$ the Euclidean norm and set $\alpha=0.7$. Then the resulting $E^1_\alpha$ complies with \eqref{monotone-edges} and the lattice sets $A(\ell)$ fulfill \eqref{tech-ass}, as it can be noted in Fig.~\ref{fig:ex-tech-ass}.

\begin{figure}[h]
\centering
\includegraphics[width=0.9\textwidth]{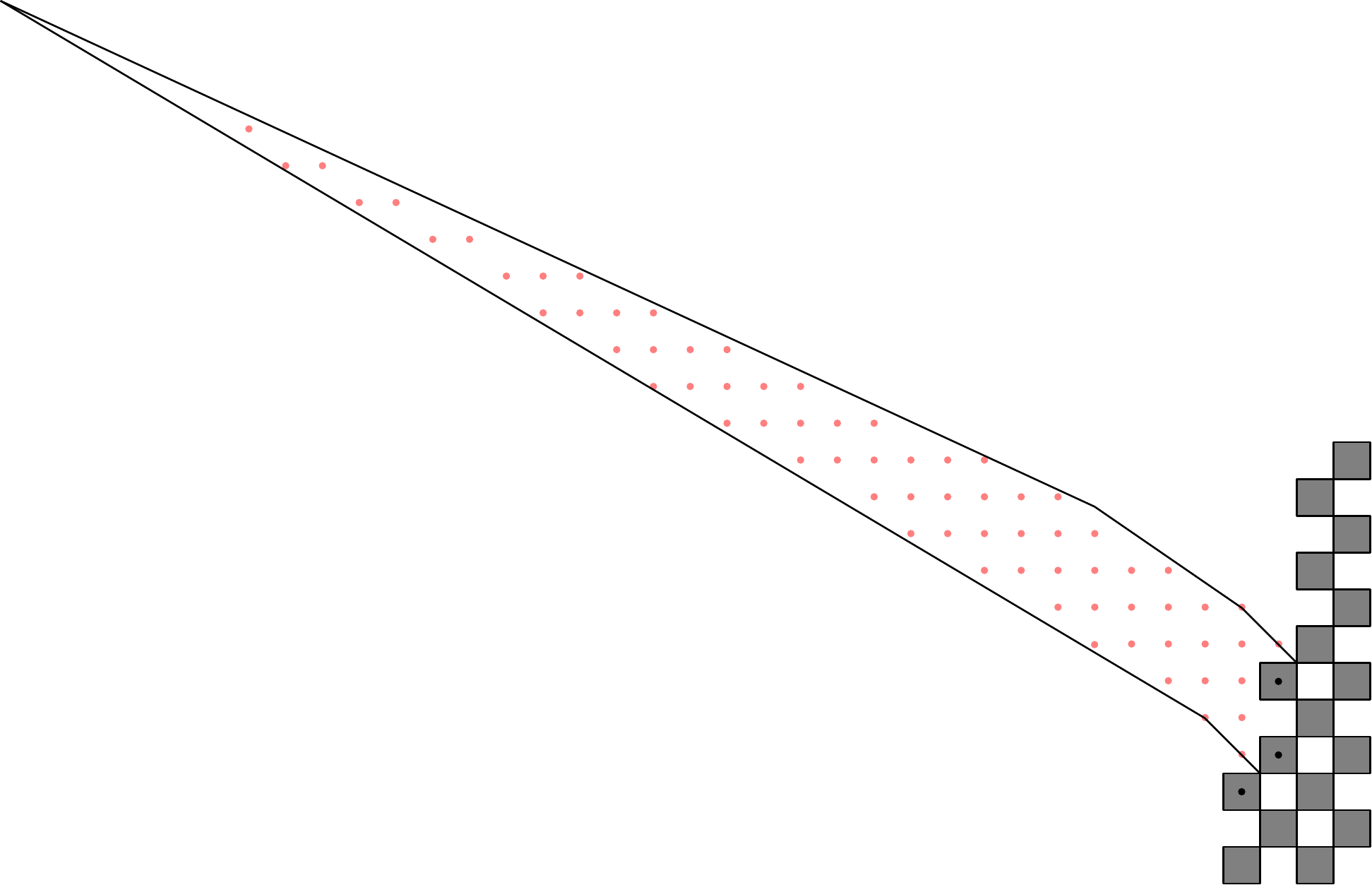}
\caption{The discrete edge $\ell$ represented with black dots does not satisfy the compatibility condition \eqref{monotone-edges}.
The red dots denote the points of $A(\ell)$ which does not comply with \eqref{tech-ass}.}
\label{fig:counterexample}
\end{figure}
If we choose instead $\varphi=\|\cdot\|_3$ and $\alpha=0.71$, the compatibility condition \eqref{monotone-edges} is violated for $E^1_\alpha$ as shown in Fig.~\ref{fig:counterexample}.
This also provides an example in which \eqref{tech-ass} is not satisfied, hence the indices $a_h$ and $b_h$ introduced in Step 2 of Proposition \ref{steps} are not well defined.
\end{example}

\begin{figure}[htp]
\begin{minipage}{0.4\linewidth}
\begin{center}
\includegraphics{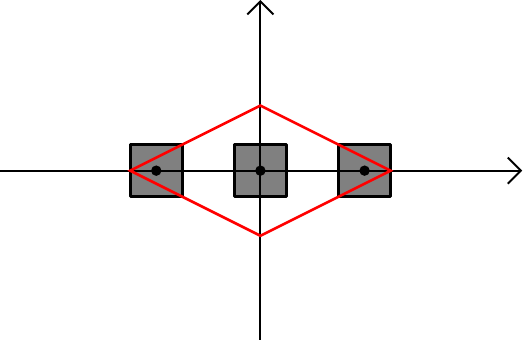}
\end{center}
\end{minipage}
\begin{minipage}{0.6\linewidth}
\begin{center}
\includegraphics{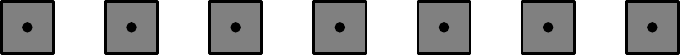}
\end{center}
\end{minipage}
\caption{On the left the set $E_\alpha^1$ and the boundary of $B_\frac{4}{\alpha}^\varphi$ in red.
On the right the discrete solution $E_\alpha^k$ after two steps.}
\label{fig:1D-ex}
\end{figure}

\begin{example}[one-dimensional motion]\label{1D-ex}
We consider an absolute norm which does not satisfy the normalization assumption $\varphi(1,0)=\varphi(0,1)=1$; that is,
$$
\varphi(\x)=|x_1|+2|x_2|, \quad \text{for every } \x\in\R^2,
$$
and take $\frac{4}{3}<\alpha<2$.
Then Proposition \ref{firststep} applies in this case and gives
$$
E_\alpha^1=q((-2,0))\cup q((0,0))\cup q((2,0))\,.
$$
Even though Proposition \ref{steps} cannot be applied on such set, it is straightforward to see in a direct way that the solution of \ref{MM-scheme2} $\{E_\alpha^k\}$ is given by
$$
E_\alpha^k=\bigcup_{h=0}^k \big(q((-2h,0))\cup q((2h,0))\big)\,,\,\,k\geq0\,,
$$
see Fig.~\ref{fig:1D-ex}.
The resulting minimizing movement will be the family of horizontal line segments
$$
E(t) =\lim_{\e\to0} E_{\e,\tau}(t) = \lim_{\e\to0} \e E_\alpha^{\lfloor\frac{t}{\tau}\rfloor} = [-2\alpha t,2\alpha t]\times\{0\}\,,\,\,t\geq0\,.
$$
\end{example}

\subsection{Nucleation and growth of a set}\label{sec:nucleation}
By virtue of Proposition~\ref{steps}, we can characterize the time-discrete flow $\{E_\alpha^k\}_{k\geq0}$ solution of \eqref{MM-scheme2}.
This evolution admits an alternative interpretation, based on a geometric iterative process that we will call \emph{nucleation of the initial set}.
Indeed, the set of centers of the $k$-th step of the discrete evolution $Z(E_\alpha^k)$ can be obtained from that of the previous step $Z(E_\alpha^{k-1})$ by adding (in the Minkowski sense) the nucleus $\mN_\alpha^\varphi$ (see Definition \ref{nucleus-def}); \emph{i.e.}, a lattice set that characterizes the motion.

\begin{theorem}\label{thm:nucleation}
Let $\varphi$ be a symmetric absolute normalized norm satisfying {\rm\ref{norm-derivative}}, and let $\alpha>0$ be such that $\alpha\not\in\Lambda^\varphi$.
If $E(\mN_\alpha^\varphi)$ satisfies assumption \eqref{monotone-edges} then there exists a unique discrete solution $\{E_\alpha^k\}$ of \eqref{MM-scheme2}  which is given, for any $k\ge1$, by
\begin{equation}\label{nucleation1}
Z(E^k_\alpha) = \underbrace{\mN_\alpha^\varphi + \dots + \mN_\alpha^\varphi}_{k\text{-times}} \,.
\end{equation}
In particular, $E^k_\alpha\in\A^e_{\rm conv}$ for every $k\ge1$.
\end{theorem}

\begin{proof}
We first note that, for a lattice set $\I$ such that $E(\I)$ belongs to $\mathcal{A}^e_{\rm conv}$ and satisfies \eqref{monotone-edges}
\begin{equation}
\mbox{$E(\underbrace{\I+\I+\dots+\I}_{m\text{-times}})\in\A^e_{\conv}$ still satisfies \eqref{monotone-edges}, for every $m\in\N$.}
\label{rem:minkowski}
\end{equation}
{It will suffice to show \eqref{rem:minkowski} for $m=2$, as the claim for $m\geq3$ will follow by an induction argument on the number $m$ of the summands. Setting $\mathcal{Q}:={\rm conv}(\I)$, property \eqref{iden} with $\Lambda=\Ze$ and $m=2$ reads as $(\mathcal{Q}\cap\Ze)+(\mathcal{Q}\cap\Ze)=2\mathcal{Q}\cap\Ze$, yielding that $E(\I+\I)\in\A^e_{\rm conv}$.
Moreover, a property equivalent to \eqref{monotone-edges} is that all the discrete vertices of $E(\I)$ belongs to $\partial\mathcal{Q}$.
This implies that the set of outward unit normal vectors of $\mathcal{Q}$ coincide with the set of (discrete) outward unit normal vectors of $E(\I)$.
In particular, every edge $l$ of $\mathcal{Q}$ identifies a finite chain of discrete edges of $E(\I)$ having the same unit normal vector $\nu(l)$.
This fact depends only on $\nu(l)$ and not on the length of $l$.
Proposition~\ref{min-sum-edges} with $A=B=\mathcal{Q}$ implies that the set of outward unit normal vectors of $2\mathcal{Q}$ coincide with that of $\mathcal{Q}$.
Hence, the edge $l+l$ of $2\mathcal{Q}$ corresponds to a chain of a finite number of discrete edges of $E(\I+\I)$ having the same unit normal vector $\nu(l)$.
The $\Ze$-convexity of $E(\I+\I)$ implies \eqref{monotone-edges}.}

Going back to the proof of \eqref{nucleation1}, we argue by induction on the step $k$.
By Proposition \ref{firststep} and Lemma \ref{non-deg-lem} $Z(E^1_\alpha)=\mN_\alpha^\varphi$ complies with all the assumptions on $E$ of Proposition \ref{steps}. 
Now, let $k\geq2$ and assume that 
$$
Z(E^{k-1}_\alpha) = \underbrace{\mN_\alpha^\varphi + \dots + \mN_\alpha^\varphi}_{(k-1)\text{-times}} \,.
$$
For what remarked in \eqref{rem:minkowski}, all the hypotheses of Proposition \ref{steps} are satisfied.
Then, taking into account \eqref{even-steps} with $E=E_\alpha^{k-1}$, we have that
\begin{equation}
Z(E^k_\alpha)=Z(E^{k-1}_\alpha)+\mN_\alpha^\varphi\,.
\label{eq:minkowskik}
\end{equation}
Indeed, setting $\I_k:=Z(E^{k-1}_\alpha)+\mN_\alpha^\varphi$, we have
\begin{equation*}
\max\{d^\varphi(\iv,\jv) \,:\, \iv\in\I_k,\, \jv\in Z(E^{k-1}_\alpha)\}\leq\max_{\iv\in\mN_\alpha^\varphi}\varphi(\iv)<\frac{4}{\alpha},
\end{equation*}
and this shows that $\mathcal{I}_k\subseteq Z(E^k_\alpha)$.
On the other hand, if $\iv\in Z(E^k_\alpha)$, there exist $\iv'\in Z(E^{k-1}_\alpha)$ and $\iv''\in\mN_\alpha^\varphi$ such that $\iv=\iv'+\iv''$.
This comes by noting that by \eqref{even-steps} there exists $\iv'\in Z(E_\alpha^{k-1})$ such that $\varphi(\iv-\iv')=d^\varphi(\iv,E^{k-1}_\alpha)<\frac{4}{\alpha}$, thus $\iv''=\iv-\iv'\in\mN_\alpha^\varphi$.
This yields \eqref{nucleation1}.
Moreover, again by \eqref{rem:minkowski} we get that the Minkowski sum in \eqref{eq:minkowskik} preserves assumption \eqref{monotone-edges}, so $E^k_\alpha$ still satisfies \eqref{monotone-edges} and the thesis is proved.
\end{proof}

\section{The limit motion}\label{sec:limitmotion}

In this section we characterize the minimizing movements of scheme \eqref{MM-scheme} as $\tau,\e\to0$ in the critical regime $\e=\alpha\tau$ for any positive value of the parameter $\alpha$ outside the singular set $\Lambda^\varphi$, under the assumption that $E(\mN_\alpha^\varphi)$ complies with \eqref{monotone-edges}.

As already explained at the beginning of Section \ref{sec:scaledcheckbd}, we also prove the existence of a value for $\alpha$ depending only on the chosen norm $\varphi$, above which every minimizing movement is trivial.
For every $\alpha$ below the pinning threshold, instead, the limit motion is a family of expanding sets, \emph{nucleating} from the origin with constant velocity, as the limit set $E(t)$ turns out to be a dilation of the (renormalized) polygon
\begin{equation}\label{eq:nucleuspol}
P_\alpha^\varphi := \begin{cases}\displaystyle
\Big(\max_{\iv\in\mN_\alpha^\varphi}i_1\Big)^{-1}\conv(\mN_\alpha^\varphi) & \displaystyle \text{if } \mN_\alpha^\varphi\neq\{(0,0)\} 
\\
\{(0,0)\} & \displaystyle \text{if } \mN_\alpha^\varphi=\{(0,0)\}
\end{cases}
\end{equation}
Note that $\max_{\iv\in\mN_\alpha^\varphi}i_1\in\{2\lfloor\frac{2}{\alpha}\rfloor,\lfloor\frac{4}{\alpha}\rfloor\}$, from the definition of $\mN_\alpha^\varphi$ and the fact that $\varphi(i_1,0)=i_1$, $i_1\in \N$.

\begin{theorem}\label{limit1}
Let $\varphi$ be a symmetric absolute normalized norm satisfying {\rm\ref{norm-derivative}}, let $\alpha>0$ be given such that $\alpha\not\in\Lambda^\varphi$ and let $\mathcal{F}_{\e,\tau}^\varphi$ be defined by \eqref{energy}.
Let $\mN_\alpha^\varphi$ be as in Definition {\rm\ref{nucleus-def}}.
If the set $E(\mN_\alpha^\varphi)$ satisfies assumption \eqref{monotone-edges}, then there exists a unique minimizing movement $E:[0,+\infty)\to\mathcal{X}$ for the scheme \eqref{MM-scheme} at regime $\e=\alpha\tau$ defined by
\begin{equation}\label{MM-limit-formula}
E(t) = v_\alpha^\varphi \, t \, P_\alpha^\varphi \quad \text{for every } t\ge0\,,
\end{equation}
where $P_\alpha^\varphi$ is defined in \eqref{eq:nucleuspol} and $v_\alpha^\varphi=\alpha\max_{\iv\in\mN_\alpha^\varphi}i_1$.
Moreover, there exists a unique discrete solution $E_{\e,\tau}(t)$ of scheme \eqref{MM-scheme} at regime $\e=\alpha\tau$ and there holds
\begin{equation}\label{weakconv}
\chi_{E_{\varepsilon,\tau}(t)}\weakcs\frac{1}{2}\,\chi_{E(t)}, \quad \text{for every } t\ge0 \quad \text{as }\e\to0.
\end{equation}
\end{theorem}

\begin{proof}
By a scaling argument, for every discrete solution $E_{\e,\tau}^k$ of \eqref{MM-scheme} in the regime $\e=\alpha\tau$ we have $E_{\e,\tau}^k=\e E_\alpha^k$ for every $k\ge0$, where $E_\alpha^k$ denotes a discrete solution of \eqref{MM-scheme2}.
Then, by Theorem \ref{thm:nucleation} there exists a unique minimizing movement of scheme \eqref{MM-scheme} at regime $\e=\alpha\tau$.
Since, by Proposition \ref{prop:iden} and \eqref{eq:nucleuspol},
$$
\underbrace{\mN_\alpha^\varphi+\dots+\mN_\alpha^\varphi}_{k\text{-times}} = \Big(k \frac{v_\alpha^\varphi}{\alpha} P_\alpha^\varphi\Big) \cap\Ze,
$$
we get that $\conv(Z_\e(E_{\e,\tau}(t))) = \e \frac{v_\alpha^\varphi}{\alpha} \lfloor\frac{\alpha t}{\varepsilon}\rfloor P_\alpha^\varphi$.
Moreover, noting that $d_{\mathcal{H}}(F,\conv(Z_\e(F)))<\e$ for any $F\in\A_\e$, we get
$$d_{\mathcal{H}}\Big(E_{\e,\tau}(t),v_\alpha^\varphi \, t \, P_\alpha^\varphi\Big) < \e + v_\alpha^\varphi \, \Big(t-\frac{\varepsilon}{\alpha}\Big\lfloor\frac{\alpha t}{\varepsilon}\Big\rfloor\Big)$$
which tends to zero as $\e\to0$, for any $t\ge0$, whence \eqref{MM-limit-formula} follows.
Eventually, from the fact that
$|E_{\e,\tau}(t)\cap A|\to\frac{1}{2}|A| $ as $\e\to0$
for any open set $A\subset E(t)$, we get \eqref{weakconv}.
\end{proof}

\begin{definition}[pinning threshold]\label{pin-def}
We define the \emph{pinning threshold} of the motion obtained by solving scheme \eqref{MM-scheme}
as
\begin{equation}
\alpha_\varphi:=\inf\bigl\{\alpha>0\,:\,\, E^\alpha(t)\equiv\{(0,0)\} \mbox{ for every $E^\alpha$ minimizing movement of \eqref{MM-scheme}}\bigr\}\,.
\label{eq:pinn-thresh}
\end{equation}
\end{definition}

It turns out that $\alpha_\varphi$ is related to the singular set $\Lambda^\varphi$ defined in \eqref{singularset} as follows.

\begin{proposition}\label{prop:pinning}
The pinning threshold is given by $\displaystyle\alpha_\varphi=\frac{4}{\varphi(1,1)}=\max\Lambda^\varphi$.
\end{proposition}
\begin{proof}
We note that $B_\frac{4}{\alpha}^\varphi\cap\Ze=\{(0,0)\}$ if and only if $\alpha>\frac{4}{\varphi(1,1)}$, thus Proposition \ref{firststep} yields the result.
\end{proof}

\begin{remark}\label{remk:nucleation}
The results of Theorems \ref{thm:nucleation} and \ref{limit1} can be extended to solutions of a minimizing-movements scheme with a more general initial datum $E^0$.
Indeed, let $E^0_{\rm disc}\in\A_{\rm conv}^e$ be a set satisfying \eqref{non-degeneracy}, \eqref{monotone-edges} and \eqref{E-symm}. We can apply Proposition \ref{steps} with $E'=E^0_{\rm disc}$ obtaining the first step of the discrete solution corresponding to scheme \eqref{MM-scheme2} with $E_\alpha^0=E^0_{\rm disc}$.
Then, if $E(\mN_\alpha^\varphi)$ satisfies assumption \eqref{monotone-edges}, from the same arguments of the proof of Theorem \ref{thm:nucleation} there exists a unique discrete solution of the scheme
$$
\begin{cases}
E^0_\alpha=E^0_{\rm disc} \\
E_\alpha^{k+1}\in\argmax{E'\in\D,\,E'\supset E}\mathcal{F}_\alpha^\varphi(E',E_\alpha^k) & k\ge1,
\end{cases}
$$
which is given, for any $k\ge1$, by $
Z(E^k_\alpha) = Z(E^0_{\rm disc})+\underbrace{\mN_\alpha^\varphi + \dots + \mN_\alpha^\varphi}_{k\text{-times}}$.

We therefore obtain a limit result analogous to that of Theorem \ref{limit1}, provided the initial datum $E^0$ can be approximated by a sequence of admissible sets $E^0_{\e_j}\in\A_{\e_j}$ whose rescaled sets
$\frac{1}{\e_j}E_{\e_j}^0\in\A^e_{\rm conv}$ satisfy \eqref{non-degeneracy}, \eqref{monotone-edges} and \eqref{E-symm}.
This implies, in particular, that $E^0$ must be a convex symmetric set with respect to coordinated axes and bisectors $x_1=\pm x_2$.
Then there exists (up to subsequences) a minimizing movement $E:[0,+\infty)\to\mathcal{X}$ for the scheme
\begin{equation}\label{initial-datum-scheme}
\begin{cases}
E^0_{\e,\tau}=E^0_\e \\
E^{k+1}_{\e,\tau}\in\argmin{E'\in\D_\e,\, E'\supset E^k_{\e,\tau}}\mathcal{F}_{\e,\tau}^\varphi(E',E) & k\ge1
\end{cases}
\end{equation}
at regime $\e=\alpha\tau$ defined by
\begin{equation}\label{MM-limit-formula2}
E(t) = E^0+v_\alpha^\varphi \, t \, P_\alpha^\varphi \quad \text{for every } t\ge0.
\end{equation}
Moreover, there exists a unique discrete flat flow $E_{\e_j,\tau_j}(t)$ of the scheme \eqref{initial-datum-scheme} along the sequence $\e_j=\alpha\tau_j$ and there holds
$\chi_{E_{\e_j,\tau_j}(t)}\weakcs\frac{1}{2}\,\chi_{E(t)}$ for every $t\ge0$
as $j\to+\infty$.
\end{remark}

\subsection{Examples of explicit evolutions}\label{sec:examples}

We continue our analysis by providing several examples of minimizing movements that can be completely characterized, which exhibit interesting phenomena that may appear due to the discrete nature of our problem.

\begin{figure}[h]
\centering
\def\svgwidth{270pt}
\begingroup%
  \makeatletter%
  \providecommand\color[2][]{%
    \errmessage{(Inkscape) Color is used for the text in Inkscape, but the package 'color.sty' is not loaded}%
    \renewcommand\color[2][]{}%
  }%
  \providecommand\transparent[1]{%
    \errmessage{(Inkscape) Transparency is used (non-zero) for the text in Inkscape, but the package 'transparent.sty' is not loaded}%
    \renewcommand\transparent[1]{}%
  }%
  \providecommand\rotatebox[2]{#2}%
  \newcommand*\fsize{\dimexpr\f@size pt\relax}%
  \newcommand*\lineheight[1]{\fontsize{\fsize}{#1\fsize}\selectfont}%
  \ifx\svgwidth\undefined%
    \setlength{\unitlength}{1274.74731445bp}%
    \ifx\svgscale\undefined%
      \relax%
    \else%
      \setlength{\unitlength}{\unitlength * \real{\svgscale}}%
    \fi%
  \else%
    \setlength{\unitlength}{\svgwidth}%
  \fi%
  \global\let\svgwidth\undefined%
  \global\let\svgscale\undefined%
  \makeatother%
  \begin{picture}(1,0.34432167)%
    \lineheight{1}%
    \setlength\tabcolsep{0pt}%
    \put(0,0){\includegraphics[width=\unitlength,page=1]{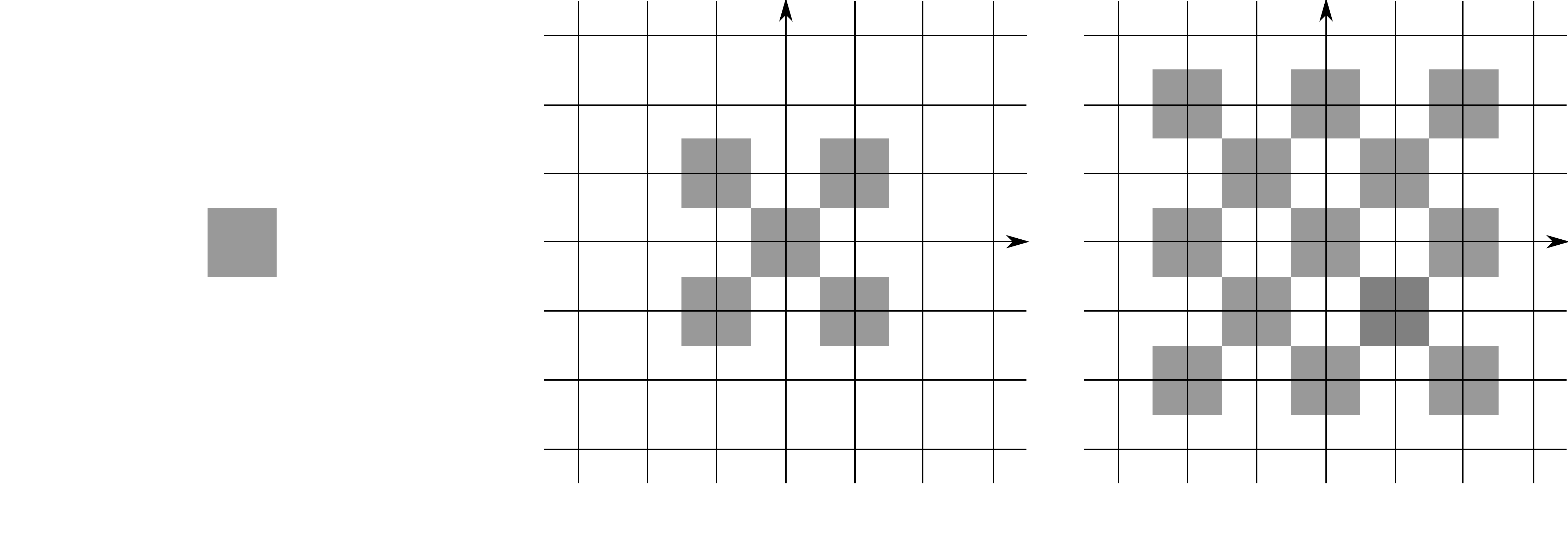}}%
    \put(0.07777961,0.00467693){\color[rgb]{0,0,0}\makebox(0,0)[lt]{\lineheight{0}\smash{\begin{tabular}[t]{l}$E_{\varepsilon,\tau}^{0}$\\ \end{tabular}}}}%
    \put(0.46090678,0.00467693){\color[rgb]{0,0,0}\makebox(0,0)[lt]{\lineheight{0}\smash{\begin{tabular}[t]{l}$E_{\varepsilon,\tau}^{1}$\\ \end{tabular}}}}%
    \put(0.80985752,0.00467693){\color[rgb]{0,0,0}\makebox(0,0)[lt]{\lineheight{0}\smash{\begin{tabular}[t]{l}$E_{\varepsilon,\tau}^{2}$\\ \end{tabular}}}}%
    \put(0,0){\includegraphics[width=\unitlength,page=2]{squarestep.pdf}}%
  \end{picture}%
\endgroup%

\caption{Some steps of the discrete evolution.}
\label{fig:square}
\end{figure}
\begin{example}[the $\ell^\infty$-norm]\label{ex:linfty}
The solutions of the unconstrained scheme \eqref{MM-scheme-unc} have already been analyzed in any dimension in the case $\varphi=\|\cdot\|_\infty$ in \cite{BraSci14}, where it has been proved that every step of the discrete evolution is an even $\e$-checkerboard (see Fig.~\ref{fig:square}). Thus, solutions of \eqref{MM-scheme} and \eqref{MM-scheme-unc} coincide.
The singular set \eqref{singularset} corresponds to $\Lambda^\varphi=\{\frac{4}{n}\}_{n\in\N}$ and the pinning threshold is $\alpha_\varphi=4$.
Here, since $
\mN_\alpha^\varphi=\Big[-\frac{4}{\alpha},\frac{4}{\alpha}\Big]^2\cap\Ze$ for every $\alpha\not\in\Lambda^\varphi$,
$E(\mN_\alpha^\varphi)$ always fulfills \eqref{monotone-edges}.
Therefore, from Theorem \ref{limit1}, for every $\alpha\not\in\Lambda^\varphi$ the minimizing movement is
$$
E(t)=\Big[-\alpha\Big\lfloor\frac{4}{\alpha}\Big\rfloor t,\alpha\Big\lfloor\frac{4}{\alpha}\Big\rfloor t\Big]^2,\quad \text{for every } t\ge0.
$$
We note that, for this choice of the norm $\varphi$, the polygon $P_\alpha^\varphi=[-1,1]^2$ does not depend on $\alpha$.
\end{example}

\begin{figure}[htbp]
\begin{minipage}{0.5\linewidth}
\begin{center}
\includegraphics[scale=.7]{ex-tech-ass1}
\end{center}
\end{minipage}
\begin{minipage}{0.5\linewidth}
\begin{center}
\includegraphics[scale=.7]{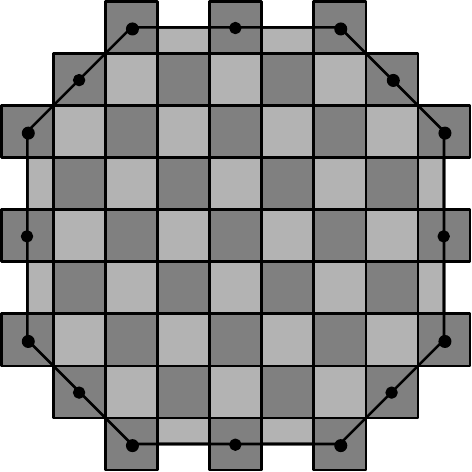}
\end{center}
\end{minipage}
\caption{For different choices of $\alpha$ the polygon $P_\alpha^\varphi$ may have different shapes.}
\label{fig:ex-euclidea}
\end{figure}

\begin{example}[$\alpha$-depending shape of $P_\alpha^\varphi$]

Contrarily to the previous example, in the case of the Euclidean norm the polygon $P_\alpha^\varphi$ may change wih  $\alpha$ (see, for instance, Fig.~\ref{fig:ex-euclidea} corresponding to $\alpha=0.85$ on the left and  $\alpha=0.7$ on the right).
Therefore, the limit motions corresponding to the two different values of $\alpha$ are not homothetic. This phenomenon may happen for those norms $\varphi$ whose balls are not polygons or are polygons having a unit normal vectors different from $(0,\pm1),$ $(\pm1,0)$ or $(\pm\frac{1}{\sqrt{2}},\pm\frac{1}{\sqrt{2}})$.
\end{example}

\begin{figure}[h]
\centering
\def\svgwidth{150pt}
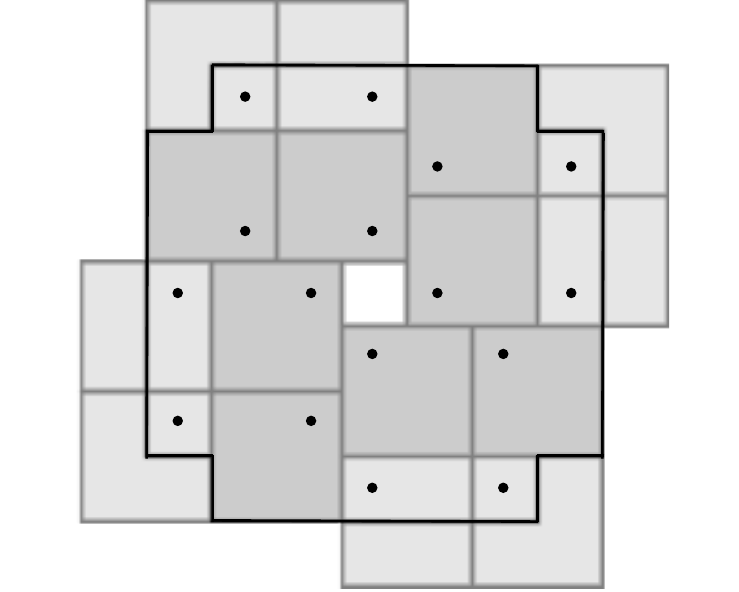
\caption{\small The picture clarifies the $2\times2$-square covering $\mathcal{S}_o(E)$ for a set $E$, whose boundary is marked by a bold black line. The darker $2\times2$-squares are respectively in $\mathcal{S}_o^c(E)$, the lighter ones in $\mathcal{S}_o^b(E)$. The areas in white are those left uncovered.
}
\label{bsc2}
\end{figure}
\begin{example}[the $\ell^1$-norm]\label{ex:l1-norm}
We consider now $\varphi=\|\cdot\|_1$.
Also in this case, as for the $\infty$-norm, the structure of $\varphi$ facilitates the analysis of the unconstrained scheme \eqref{MM-scheme-unc}. 
We then study the rescaled problem
\begin{equation}\label{MM-scheme-l1}
\begin{cases}
E^0=q \\
E_\alpha^{k+1}\in\argmin{E'\in\D}\mathcal{F}_\alpha^\varphi(E',E) & k\ge1\,,
\end{cases}
\end{equation}
where we separately examine the cases in which the minimizer of the first step contains $q$ or not.
For this, in addition to $\mathcal{S}_e(E)$ of Definition~\ref{coverings} we introduce the family 
\begin{equation}\label{eq:oddcovering}
\mathcal{S}_o(E):=\left\{Q(\jv):\, Q(\jv)\cap E\neq\emptyset \mbox{ and } j_1 \mbox{ even},  j_2 \mbox{ odd}\right\}
\end{equation}
which is a covering of $E\setminus q$ (see Fig.~\ref{bsc2}) and, accordingly, we consider the partition $\mathcal{S}_o(E)=\mathcal{S}^b_o(E)\cup\mathcal{S}^c_o(E)$.

\begin{figure}[htp]
\begin{minipage}{0.5\linewidth}
\begin{center}
\includegraphics[width=0.5\linewidth]{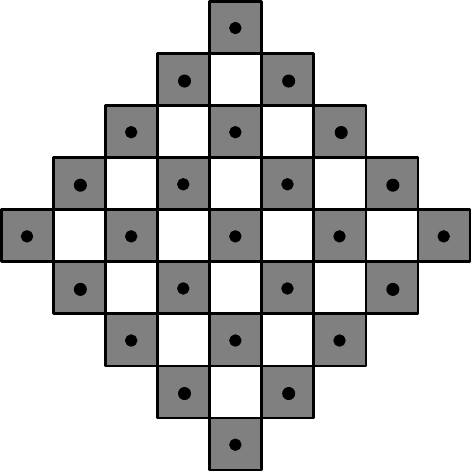}
\end{center}
\end{minipage}
\begin{minipage}{0.5\linewidth}
\begin{center}
\includegraphics[width=0.6\linewidth]{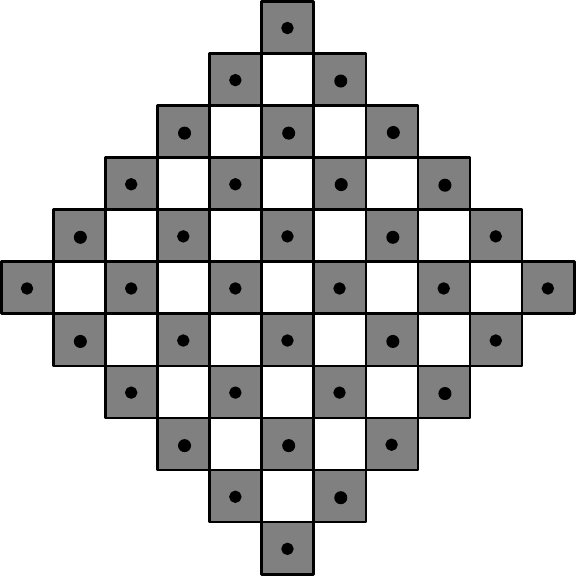}
\end{center}
\end{minipage}
\caption{An example of $E_\alpha^e$ (on the left) and $E_\alpha^o$ (on the right).}
\label{fig:rhombus}
\end{figure}

Now, in the case of scheme \eqref{MM-scheme-l1} with the monotonicity constraint, Proposition \ref{firststep} and \eqref{singularset} ensure that, if $\alpha\not\in\{\frac{2}{n} : n\in\N\}$ then
$
\argmin{E\supset q}\mathcal{F}_\alpha^\varphi(E,q)=E(\Ze\cap B_\frac{4}{\alpha}^\varphi) =: E^e_\alpha$.
In the unconstrained case, an analogous argument as in the proof of Proposition \ref{firststep}, with $\mathcal{S}_o$ in place of $\mathcal{S}_e$, 
shows that if $\alpha\not\in\{\frac{4}{2n-1} : n\in\N\}$ then
$
\argmin{E\not\supset q}\mathcal{F}_\alpha^\varphi(E,q)=E(\Zo\cap B_\frac{4}{\alpha}^\varphi) =: E^o_\alpha$.
This proves that $E^1_\alpha$ is either an even or an odd checkerboard.
We remark that $B_r^\varphi$ is a regular rhombus (of radius $r$) and so are the convex hulls of $Z(E_\alpha^e)$ and $Z(E_\alpha^o)$. The checkerboard sets $E^e_\alpha$ and $E^o_\alpha$ are pictured in Fig.~\ref{fig:rhombus}.

\begin{figure}[htp]
\begin{minipage}{0.5\linewidth}
\begin{center}
\includegraphics[width=0.6\linewidth]{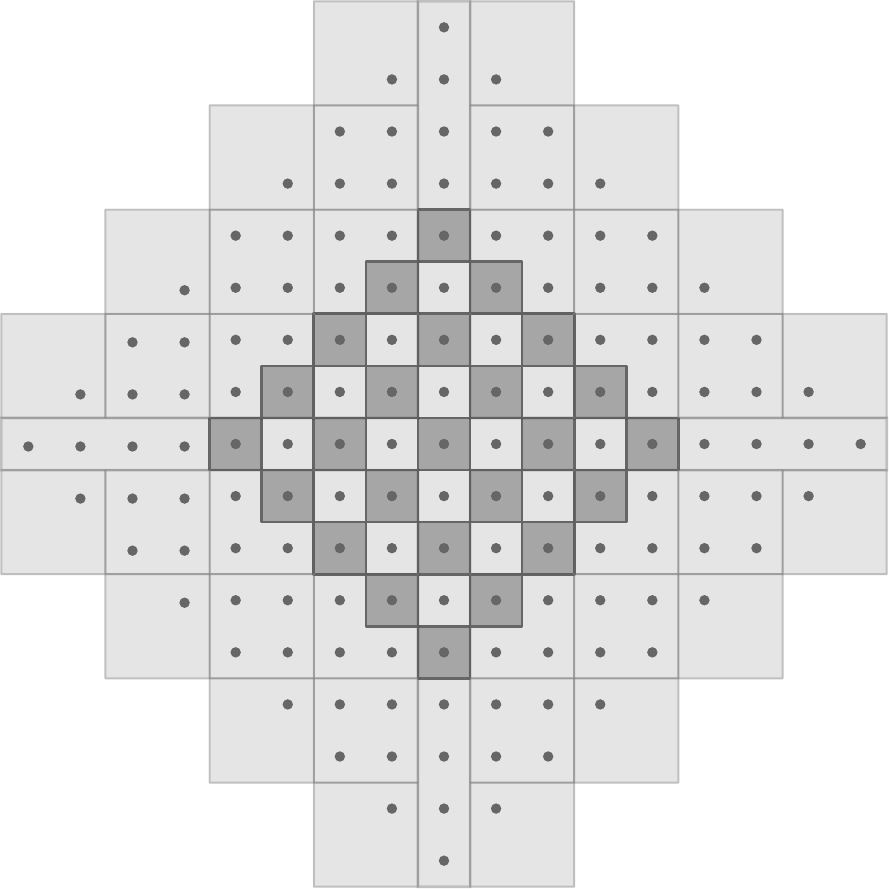}
\end{center}
\end{minipage}
\begin{minipage}{0.5\linewidth}
\begin{center}
\includegraphics[width=0.6\linewidth]{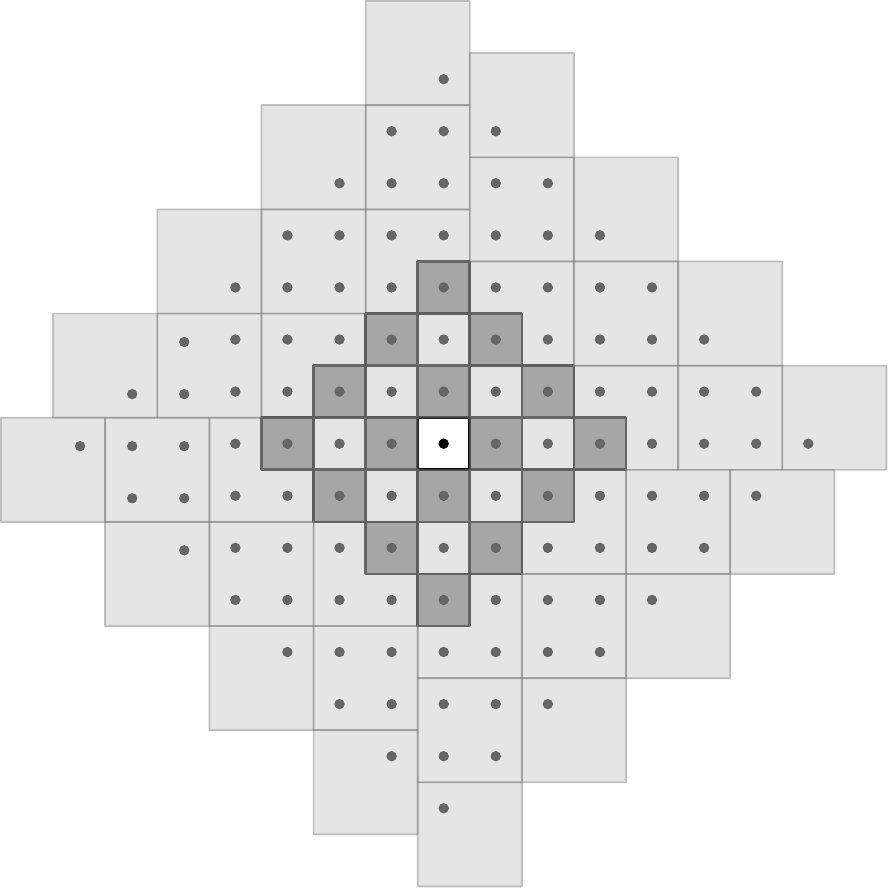}
\end{center}
\end{minipage}
\caption{On the left the covering $\mathcal{S}_e(E(\I^e))$ and the set $E_\alpha^e$.
On the right the covering $\mathcal{S}_o(E(\I^o))$ and the set $E_\alpha^o$.
The darker dots represent $\I^e$ and $\I^o$ respectively.}
\label{fig:1norm}
\end{figure}
The relevant point of this example is that, for this choice of the norm, the shape of the minimizers is very simple and the $2\times2$-square covering argument of Section~\ref{sec:bigsquares} directly applies to the $k^{th}$-step $E_\alpha^k$, $k\geq1$, without any further adjustment.
Moreover, it provides a covering of $\R^2$ (in the even case) or $\R^2\setminus q$ (in the odd case) and not only of $\R^2\setminus E(\conv(Z(E^{k-1}_\alpha)))$, see Figure \ref{fig:1norm}.
Thus, the corresponding localization argument allows to study the unconstrained problem.
Indeed, if $E^1_\alpha=E^e_\alpha$ then
for every $Q(\jv)\in\mathcal{S}_e(E(\I^e))$ we get
$$
\mathcal{F}_\alpha^\varphi(Q(\jv)\cap E(\Ze),E^1_\alpha) = \min_{E\in\mathcal{D}} \mathcal{F}_\alpha^\varphi(Q(\jv)\cap E,E^1_\alpha),
\quad
\mathcal{F}_\alpha^\varphi(\mathcal{C}_0\cap E(\Ze),E^1_\alpha) \le \min_{E\in\mathcal{D}} \mathcal{F}_\alpha^\varphi(\mathcal{C}_0\cap E,E^1_\alpha)
$$
whereas if $E^1_\alpha=E^o_\alpha$ then for every $Q(\jv)\in\mathcal{S}_o(E(\I^o))$ we get
$$
\mathcal{F}_\alpha^\varphi(Q(\jv)\cap E(\Zo),E^1_\alpha) \le \min_{E\in\mathcal{D}} \mathcal{F}_\alpha^\varphi(Q(\jv)\cap E,E^1_\alpha),
$$
where
$$
\I^e=\Big\{\iv\in\Ze\,:\, d^\varphi(\iv,E_\alpha^1)<\frac{4}{\alpha}\Big\},
\quad
\I^o=\Big\{\iv\in\Zo\,:\, d^\varphi(\iv,E_\alpha^1)<\frac{4}{\alpha}\Big\}
$$
which gives that $Z(E^2_\alpha)\in\{\I^e,\I^o\}$.
This yields, after an inductive argument, that $E^k_\alpha$ is either an even or an odd checkerboard.
The parity of $E^k_\alpha$ will be determined by a comparison between the two possible (checkerboard) configurations.
Nevertheless, a change of parity is eventually not energetically favorable. Indeed, assume $E_\alpha^k$ to be \emph{e.g.} an even checkerboard and set $\I=\{\iv\in\Zo\,:\, d^\varphi(\iv,E_\alpha^k)<\frac{4}{\alpha}\}$, we then get
\begin{align*}
\mathcal{F}_\alpha^\varphi(E(\I),E_\alpha^k)-\mathcal{F}_\alpha^\varphi(E_\alpha^k,E_\alpha^k) &\ge -4\#Z(E(\I))+2\alpha\#Z(E_\alpha^k)+4\#Z(E_\alpha^k)+c \\
&\ge -8\Big(\frac{4(k+1)}{\alpha}\Big)^2+2\alpha\Big(\frac{4(k+1)}{\alpha}\Big)^2+-8\Big(\frac{4k}{\alpha}\Big)^2+c \\
&= -c' k+c'' k^2+c,
\end{align*}
for some positive constants $c, c', c''$.
Since for $k$ large enough the contribution above is positive, for every fixed $\alpha\not\in\{\frac{4}{n}\}_{n\in\N}$ there exists an index $k_\alpha\in\N$ such that
$$
Z(E^k_\alpha)=Z(E^{k_\alpha}_\alpha)+\underbrace{\mN_\alpha^\varphi+\dots+\mN_\alpha^\varphi}_{(k-k_\alpha)\text{-times}}, \quad \text{for every } k\ge k_\alpha.
$$


\begin{figure}[h]
\centering
\def\svgwidth{350pt}
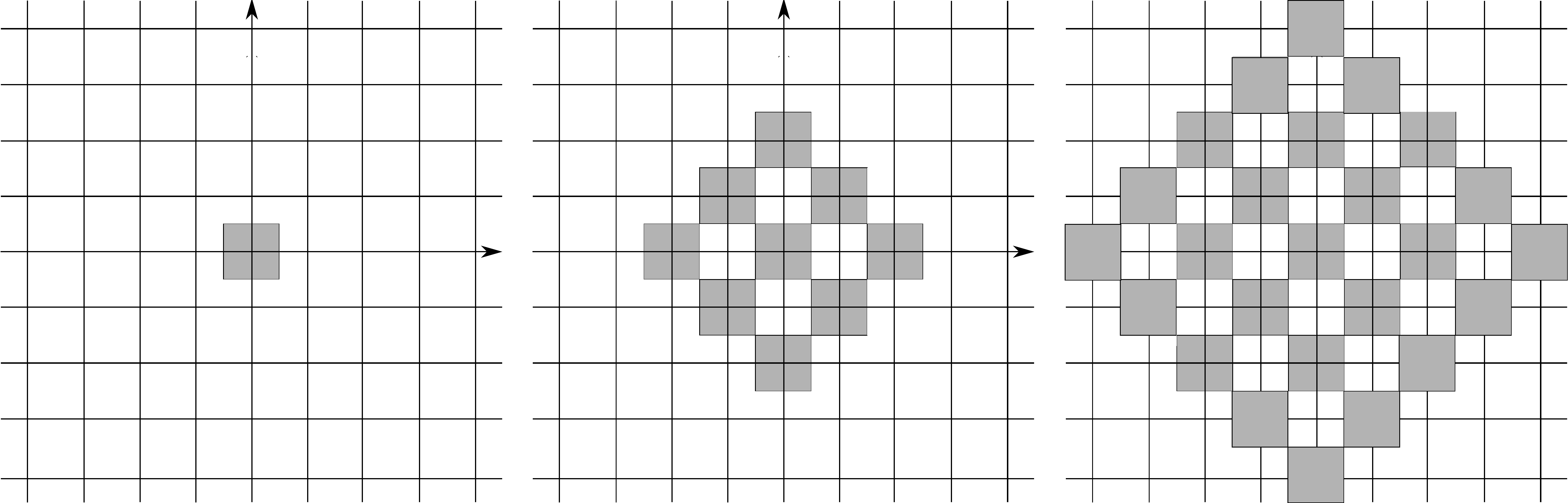
\caption{Some steps of the even evolution.}
\label{fig:rhombusevol}
\end{figure}

We can characterize the limit motion as follows.
For every $\alpha>0$ such that $\alpha\not\in\{\frac{4}{n}\}_{n\in\N}$ there exists a unique minimizing movement of unconstrained scheme \eqref{MM-scheme-unc} $E:[0,+\infty)\to\mathcal{X}$ and it satisfies
\begin{equation}\label{eq:limitrhombus}
E(t)=2\alpha\Big\lfloor\frac{2}{\alpha}\Big\rfloor t \mathcal{R}, \quad \text{for every } t\ge0,
\end{equation}
where $\mathcal{R}$ is the regular rhombus of radius $1$.
Note that, by Theorem \ref{limit1}, this coincides with the minimizing movement of the constrained scheme \eqref{MM-scheme}.

At least for the first step, the comparison between the energies of the two possible minimizers; \emph{i.e.}, $E^e_\alpha$ and $E^o_\alpha$, can be performed by a straightforward computation.
This induces a partition into subintervals of the set $(0,+\infty)\setminus\{\frac{4}{n} : n\in\N\}$, wherein one configuration is energetically more favourable than the other one.
Setting $R:=\lfloor\frac{4}{\alpha}\rfloor$, we get
\begin{align}
\mathcal{F}_\alpha^\varphi(E^e_\alpha,q) &= -4\Big(2\Big\lfloor\frac{R}{2}\Big\rfloor+1\Big)^2+4\alpha\sum_{j=1}^{\lfloor\frac{R}{2}\rfloor} (2j)^2\,, \label{vareven}\\
\mathcal{F}_\alpha^\varphi(E^o_\alpha,q) &= -4\Big(2\Big\lfloor\frac{R+1}{2}\Big\rfloor\Big)^2+4\alpha\sum_{j=1}^{\lfloor\frac{R+1}{2}\rfloor} (2j-1)^2+\alpha\,.\label{varodd}
\end{align}
After comparing the values in \eqref{vareven} and \eqref{varodd} we get that when $R$ is even
$$
\mathcal{F}_\alpha^\varphi(E^e_\alpha,q)<\mathcal{F}_\alpha^\varphi(E^o_\alpha,q) \mbox{ \, if and only if \,  }\alpha<\frac{4(2R+1)}{2R(R+1)-1},
$$
while when $R$ is odd
$$
\mathcal{F}_\alpha^\varphi(E^e_\alpha,q)<\mathcal{F}_\alpha^\varphi(E^o_\alpha,q)\mbox{ \, if and only if \,  }\alpha>\frac{4(2R+1)}{2R(R+1)+1}.
$$
Thus, for the following values of $\alpha$
\begin{equation*}
\alpha_C(R):=
\begin{cases}\displaystyle
\frac{4(2R+1)}{2R(R+1)+1} &\mbox{ if $R$ is odd,}  
\\
\displaystyle\frac{4(2R+1)}{2R(R+1)-1} &\mbox{ if $R$ is even,}
\end{cases}
\end{equation*}
the energies of the two checkerboards coincide and we also obtain that
\begin{equation}\label{l1-first}
E_\alpha^1 =
\begin{cases}
E^e_\alpha & \text{if } \displaystyle \alpha \in \bigcup_{h\ge1}(\alpha_C(2h+1),\alpha_C(2h))\cup(\alpha_C(1),+\infty), \\
E^o_\alpha & \text{if } \displaystyle \alpha \in \bigcup_{h\ge0}(\alpha_C(2h+2),\alpha_C(2h+1)).
\end{cases}
\end{equation}
In particular, \eqref{l1-first} provides an example of a discrete solution having an oscillating behavior; that is, a change of parity from a step to another, at least from $E^0_\alpha=q$ to $E_\alpha^1=E_\alpha^o$.

In this case, the pinning threshold of unconstrained  problem \eqref{MM-scheme-unc} is $\alpha_p=2$, as can be seen in formula \eqref{eq:limitrhombus}.
This is the same as that of the constrained problem \eqref{MM-scheme}, given by Proposition \ref{prop:pinning}.
In the constrained problem, for every $\alpha>2$, since $\mN_\alpha^\varphi=\{(0,0)\}$, $E^k_\alpha=q$ for every $k\ge1$.
Whereas, in the unconstrained problem, by \eqref{l1-first} we get that if $2<\alpha<\frac{12}{5}$ the discrete motion is not trivial; that is, $E^k_\alpha=\bigcup_{\|\iv\|_1=1}q(\iv)$ for every $k\ge1$, even though the limit motion is pinned.
\end{example}

%

\begin{figure}[h]
\centering
\def\svgwidth{130pt}
\begingroup%
  \makeatletter%
  \providecommand\color[2][]{%
    \errmessage{(Inkscape) Color is used for the text in Inkscape, but the package 'color.sty' is not loaded}%
    \renewcommand\color[2][]{}%
  }%
  \providecommand\transparent[1]{%
    \errmessage{(Inkscape) Transparency is used (non-zero) for the text in Inkscape, but the package 'transparent.sty' is not loaded}%
    \renewcommand\transparent[1]{}%
  }%
  \providecommand\rotatebox[2]{#2}%
  \newcommand*\fsize{\dimexpr\f@size pt\relax}%
  \newcommand*\lineheight[1]{\fontsize{\fsize}{#1\fsize}\selectfont}%
  \ifx\svgwidth\undefined%
    \setlength{\unitlength}{393.76185139bp}%
    \ifx\svgscale\undefined%
      \relax%
    \else%
      \setlength{\unitlength}{\unitlength * \real{\svgscale}}%
    \fi%
  \else%
    \setlength{\unitlength}{\svgwidth}%
  \fi%
  \global\let\svgwidth\undefined%
  \global\let\svgscale\undefined%
  \makeatother%
  \begin{picture}(1,0.99999824)%
    \lineheight{1}%
    \setlength\tabcolsep{0pt}%
    \put(0,0){\includegraphics[width=\unitlength,page=1]{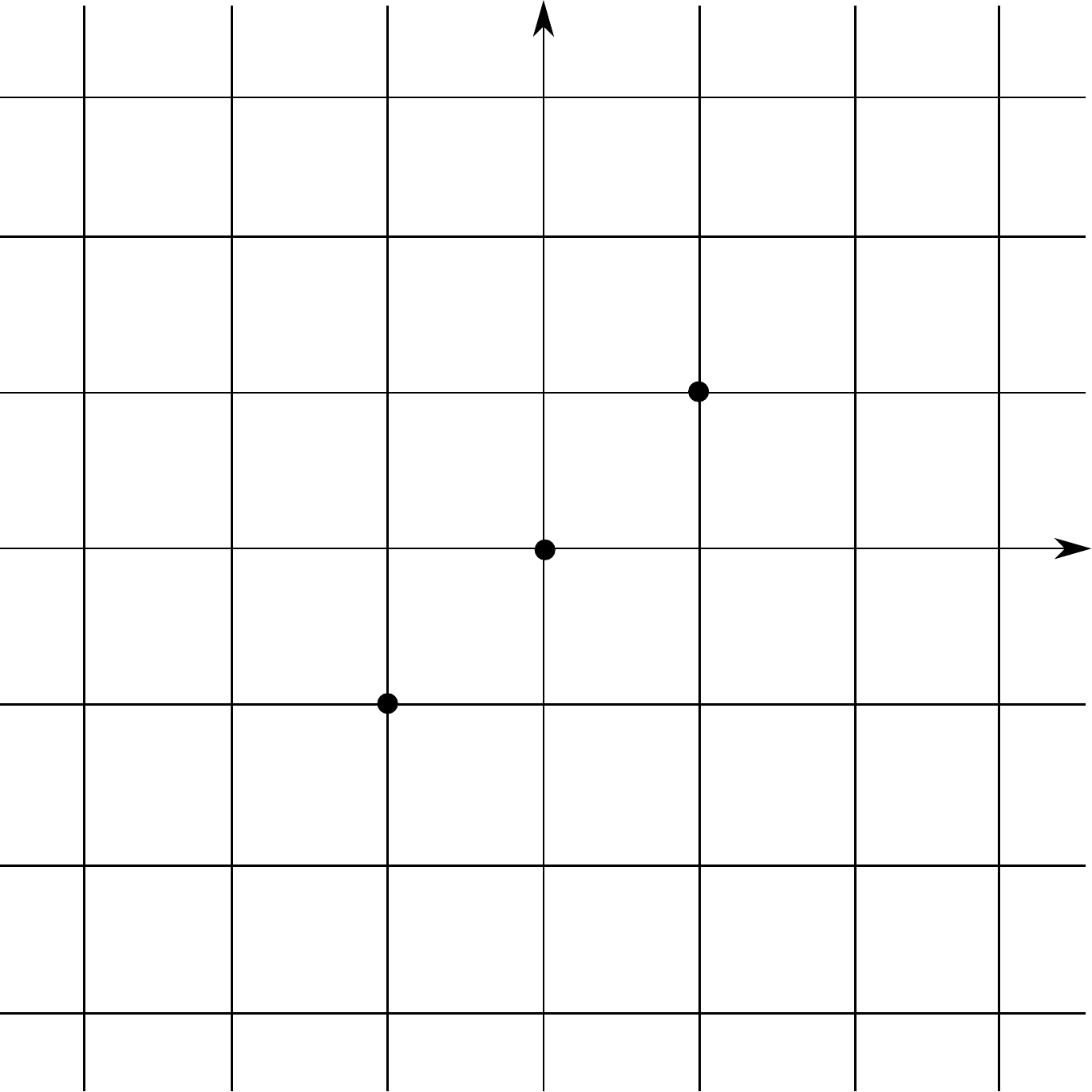}}%
    \put(0.43440851,0.42989249){\color[rgb]{0,0,0}\makebox(0,0)[lt]{\lineheight{1.25}\smash{\begin{tabular}[t]{l}$O$\end{tabular}}}}%
    \put(0,0){\includegraphics[width=\unitlength,page=2]{ellipse.pdf}}%
  \end{picture}%
\endgroup%

\caption{The unit ball of $\varphi$ for $a_{11}=2$ and $a_{12}=-\frac{5}{3}$.}
\label{fig:ellipse}
\end{figure}
\begin{example}[a strongly anisotropic norm]\label{ellypticnorm}
We now give, along the lines of Example~\ref{1D-ex}, another example where the discrete minimizers are (degenerate) checkerboard sets and the limit set is one-dimensional; \emph{i.e.}, a linearly growing segment. For this, we construct \emph{ad hoc} a strongly anisotropic \emph{non-absolute} norm $\varphi$ such that $\varphi(1,1)<\varphi(1,0)=\varphi(0,1)$. Namely, we consider the symmetric positive definite matrix ${\bf A}=(a_{ij})$ such that $a_{11}=a_{22}>1$, $a_{12}<0$ and
\begin{equation}
\frac{1}{8}<a_{11}+a_{12}<\frac{1}{2}\,,\quad 2<a_{11}-a_{12}\,.
\label{ellipassump}
\end{equation}
Correspondingly, we define the elliptic norm
\begin{equation}
\varphi(\x):=\sqrt{\x^t{\bf A}\x}=\sqrt{a_{11}(x_1^2+x_2^2)+2a_{12}x_1x_2}\,,
\label{asymmnorm}
\end{equation}
whose unit ball is pictured in Fig.~\ref{fig:ellipse}.

Assumption \eqref{ellipassump} ensures that $\varphi(1,1)=\sqrt{2(a_{11}+a_{12})}<\sqrt{a_{11}}=\varphi(1,0)=\varphi(0,1)$. In addition, we assume that
\begin{equation}
\frac{4}{\sqrt{a_{11}}}<\alpha \leq\frac{2\sqrt{2}}{\sqrt{a_{11}+a_{12}}}\,.
\label{pinnthresh}
\end{equation}
In this case, if we let $E^0_\alpha=q$, 
the set of centers of the first step is
\begin{equation}
 \mN_\alpha^\varphi=Z(E^1_\alpha)=\left\{\iv\in\Z^2:\quad \varphi(\iv)\leq \frac{4}{\alpha}\right\}=\{(-1,-1), (0,0), (1,1)\},
\label{eq:line1}
\end{equation}
whence, arguing by induction on the step $k$, we infer that
\begin{equation}
Z(E^k_\alpha)=\{(j,j):|j|=0,1,\dots, k\}\,,\quad k\geq1\,.
\label{eq:linek}
\end{equation}
\begin{figure}[h]
\centering
\def\svgwidth{250pt}
\input{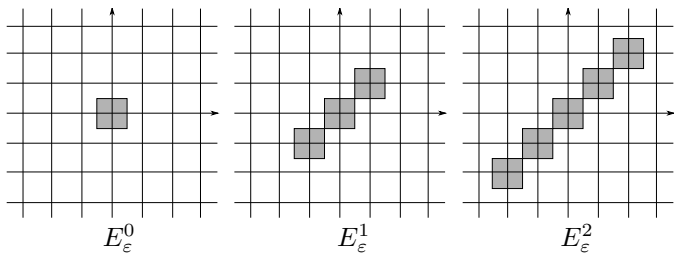}
\caption{Some steps of the evolution.}
\label{zonot}
\end{figure}
A similar computation as for the proof of Theorem~\ref{thm:nucleation} shows that an analogous characterization for $Z(E_\alpha^k)$ by means of the Minkowski sum as in \eqref{nucleation1} holds. The polygon $P_\alpha^\varphi$ here reduces to the line segment $\mathcal{L}$ of length $2\sqrt{2}$ centered at $0$ with slope $1$. In Fig.~\ref{zonot} some steps of the discrete evolution are represented.
Note that the proof of \eqref{eq:line1}-\eqref{eq:linek} does not require any covering argument in the fashion of Section~\ref{sec:coveringstep1} or any monotonicity assumption \eqref{monotone-edges}. 
The following characterization of the limit evolution immediately follows from the proof of Theorem~\ref{limit1}.
\begin{proposition}
Let $\alpha$ be such that \eqref{pinnthresh} holds. Then there exists a unique minimizing movement of \eqref{MM-scheme} $E^\alpha(t)=\alpha\mathcal{L}t$ where $\mathcal{L}$ is the line segment above.
%
\end{proposition}

\end{example}

\subsection{Further results and conjectures}\label{sec:conjectures}

In this section we focus on the non-trivial issue of addressing our problem without the monotonicity constraint. If on the one hand in the case of the $\ell^\infty$-norm (Example~\ref{ex:linfty}), the monotonicity constraint did not play any role, on the other hand in Example~\ref{ex:l1-norm} we proved that the first step of the unconstrained scheme \eqref{MM-scheme-l1} for the $\ell^1$-norm can be either an even or an odd checkerboard set. The idea of the proof was to follow the argument of Proposition~\ref{firststep}, replacing, when using the $2\times2$-square coverings, the family $\mathcal{S}_e(E)$ with $\mathcal{S}_o(E)$ defined in \eqref{eq:oddcovering} 
in the case $E\not\supset q$. This approach works for every absolute norm $\varphi$. Therefore, when removing the monotonicity constraint in the minimization scheme, we find the following generalization of Proposition~\ref{firststep}.

\begin{proposition}\label{firststep-gen}
Let $\varphi$ be an absolute norm, let $\alpha>0$ be such that $\alpha\not\in\Lambda^\varphi$ and let $\mathcal{F}_\alpha^\varphi$ be as in \eqref{scaled}.
Then the first minimization problem of scheme \eqref{MM-scheme-l1} admits the only solutions
\begin{equation*}
E_\alpha^1 = \argmin{E\in\D}\mathcal{F}_\alpha^\varphi(E,q)=
\begin{cases}
E(\Ze\cap B_{4\over\alpha}^\varphi)\in\A^e\,, & \mbox{ if } q\subset E^1_\alpha\,,\\
E(\Zo\cap B_{4\over\alpha}^\varphi)\in\A^o\,, & \mbox{ if } q\not\subset E^1_\alpha\,.
\end{cases}
\end{equation*}
\end{proposition}
At this point, we are forced to depart from Example~\ref{ex:l1-norm} for the determination of the sets $E_\alpha^k$, $k\geq2$, as the delicate construction of a covering needed in the proof of Proposition~\ref{steps} strongly relies on the monotonicity constraint on the discrete evolution and thence is no longer enough to infer an analogous result for the subsequent steps of the evolution. The investigation of this issue has therefore to be deferred to further contributions. Anyway, motivated by the previous ``positive'' examples, we do believe that under suitable assumptions on the norm $\varphi$ and the geometry of the competitors in the minimization problem one can still infer a (checkerboard) structure result as in Proposition~\ref{steps} and a characterization by means of Minkowski sums, analogous to that of Theorem~\ref{thm:nucleation}. Within this scenario, oscillations of the minimizers between checkerboards of different parity, in principle, cannot be excluded. However, energetic considerations suggest that these may occur only for a finite number of steps, depending on $\alpha$: heuristically, a change of parity at step $k$ involves a variation of the perimeter term of order $k$ which cannot match, for $k$ large, the corresponding increasing of the bulk term of order $k^2$.
In order to see this we may assume, without loss of generality, that $Z(E_\alpha^{k+1})=(k+1)\mathcal{N}_\alpha^\varphi\subset\Z_o^2$ and $Z(E_\alpha^{k})=k\mathcal{N}_\alpha^\varphi\subset\Z_e^2$ for some $k\geq1$, as an interchanging of the parity of the sets would provide an analogous estimate. 
Then, by virtue of \eqref{pick}--\eqref{counting}, the variation of the energy $\mathcal{F}_\alpha^\varphi$ from an even checkerboard $E_\alpha^k$ to the odd one $E_\alpha^{k+1}$ is bounded from below by 
\begin{equation}
\begin{split}
&-4\left(\#Z(E_\alpha^{k+1})-\#Z(E_\alpha^{k})\right) + \alpha \min\{\varphi(1,0),\varphi(0,1)\} \left(\#Z(E_\alpha^{k})+\#Z(E_\alpha^{k+1})\right)\\
& =-4\#((k+1)\mathcal{N}_\alpha^\varphi\cap\Z_o^2) + 4\#(k\mathcal{N}_\alpha^\varphi\cap\Z_e^2) + \alpha \bigl(\#(k\mathcal{N}_\alpha^\varphi\cap\Z_e^2)+\#((k+1)\mathcal{N}_\alpha^\varphi\cap\Z_o^2)\bigr)\\
& \geq -4\#((k+1)\mathcal{N}_\alpha^\varphi\cap\Z_o^2) + 4\#(k\mathcal{N}_\alpha^\varphi\cap\Z_e^2) + \alpha \#(k\mathcal{N}_\alpha^\varphi\cap\Z^2) \\
& = \alpha |{\rm conv}(\mathcal{N}_\alpha^\varphi)| k^2 + C'_\alpha k + C''_\alpha\,.
\end{split}
\label{rough-estim}
\end{equation}
Thus, there exists $k_\alpha:=k(\alpha)$ such that the right-hand side in \eqref{rough-estim} is positive for $k\geq k_\alpha$. As a consequence, the change of parity is not energetically favorable (definitely in $k$), and we expect either $E_\alpha^k\in\A^e_{\rm conv}$ or $E_\alpha^k\in\A^o_{\rm conv}$ for every $k\geq k_\alpha$ to hold as a result of iterated Minkowski sums with the even nucleus $\mN_\alpha^\varphi$ of \eqref{MM-scheme-unc}. In conclusion, since a finite number of oscillations is neglected in the limit, an analogous characterization of the limit evolution as in Theorem~\ref{limit1} holds.

We summarize our conjecture as follows.

\begin{conjecture}
Under suitable assumptions on $\varphi$ and for suitable values of $\alpha$, the discrete solutions $\{E^k\}$ of scheme \eqref{MM-scheme-unc} satisfy
\begin{equation*}\label{even-steps-conj}
\mbox{ either\,\, } Z(E_\alpha^k)=\Big\{\iv\in\Ze \,:\, d^\varphi(\iv,E_\alpha^{k-1})<\frac{4}{\alpha}\Big\} \mbox{ \,\, or\,\, }Z(E_\alpha^k)=\Big\{\iv\in\Zo \,:\, d^\varphi(\iv,E_\alpha^{k-1})<\frac{4}{\alpha}\Big\}\,.
\end{equation*}
Moreover, there exists an index $k_\alpha\in\N$ such that
$$
Z(E^k_\alpha)=Z(E^{k_\alpha}_\alpha)+\underbrace{\mN_\alpha^\varphi+\dots+\mN_\alpha^\varphi}_{(k-k_\alpha)\text{-times}}, \quad \text{for every } k\ge k_\alpha\,.
$$
As for the limit evolution, there exists a unique minimizing movement $E:[0,+\infty)\to\mathcal{X}$ for scheme \eqref{MM-scheme-unc} defined by
$E(t) = v_\alpha^\varphi \, t \, P_\alpha^\varphi$ for every $t\ge0$,
where $P_\alpha^\varphi$ and $v_\alpha^\varphi$ are as in the statement of Theorem {\rm \ref{limit1}}.
\end{conjecture}


\section*{Acknowledgements}

A. Braides acknowledges the MIUR Excellence Department Project awarded to the Department of Mathematics, University of Rome Tor Vergata, CUP E83C18000100006.
G. Scilla has been supported by the Italian Ministry of Education, University and Research through the Project “Variational methods for stationary and evolution problems with singularities and interfaces” (PRIN 2017).

\Addresses

\end{document}